\documentclass[11pt, reqno]{amsart}
\usepackage{amsmath,amssymb,amscd,mathrsfs,amscd,wasysym,color,xcolor}
\usepackage{verbatim}
\usepackage[linktocpage=true]{hyperref}

\newdimen\hoogte    \hoogte=12pt    % hoogte  van hokje
\newdimen\breedte   \breedte=14pt   % breedte van hokje
\newdimen\dikte     \dikte=0.5pt    % dikte lijn

\newenvironment{Young}{\begingroup
       \def\vr{\vrule height0.8\hoogte width\dikte depth 0.2\hoogte}
       \def\fbox##1{\vbox{\offinterlineskip
                    \hrule height\dikte
                    \hbox to \breedte{\vr\hfill##1\hfill\vr}
                    \hrule height\dikte}}
       \vbox\bgroup \offinterlineskip \tabskip=-\dikte \lineskip=-\dikte
            \halign\bgroup &\fbox{##\unskip}\unskip  \crcr }
       {\egroup\egroup\endgroup}
\def\diagram#1{\relax\ifmmode\vcenter{\,\begin{Young}#1\end{Young}\,}\else%
              $\vcenter{\,\begin{Young}#1\end{Young}\,}$\fi}

\theoremstyle{plain}
\newtheorem{thm}{Theorem}[section]
\newtheorem{lem}[thm]{Lemma}

\newtheorem{prp}[thm]{Proposition}

\newtheorem{cor}[thm]{Corollary}

\newtheorem{cnj}[thm]{Conjecture}

\newtheorem{exa}[thm]{Example}

\theoremstyle{remark}
\newtheorem{rmk}[thm]{Remark}

\numberwithin{equation}{section}

\newcommand{\af}{\alpha}
\newcommand{\bt}{\beta}
\newcommand{\gm}{\gamma}
\newcommand{\dt}{\delta}
\newcommand{\ep}{\varepsilon}

\newcommand{\te}{\theta}
\newcommand{\ld}{\lambda}
\newcommand{\sm}{\sigma}
\newcommand{\kp}{\kappa}

\newcommand{\Gm}{\Gamma}
\newcommand{\Dt}{\Delta}

\newcommand{\Q}{{\mathbb{Q}}}
\newcommand{\Z}{{\mathbb{Z}}}

\newcommand{\N}{{\mathbb{N}}}

\newcommand{\gl}{{\mathfrak{gl}}}
\newcommand{\g}{{\mathfrak{g}}}
\newcommand{\h}{{\mathfrak{h}}}
\newcommand{\n}{{\mathfrak{n}}}

\renewcommand{\b}{{\mathfrak{b}}}

\newcommand{\Uq}{U_q}

\newcommand{\PBW}{{\mathrm{PBW}}}

\newcommand{\A}{{\mathcal{A}}}

\newcommand{\I}{{\mathrm{I}}}
\newcommand{\F}{{\mathsf{F}}}
\newcommand{\W}{{\mathsf{W}}}
\renewcommand{\L}{\mathsf{L}}
\newcommand{\UU}{\mathsf{U}}
\renewcommand{\r}{\mathrm{R}}
\newcommand{\E}{{\mathrm{E}}}
\newcommand{\rF}{{\rm F}}
\renewcommand{\b}{{\mathrm{b}}}
\newcommand{\ui}{{\mathbf{i}}}
\newcommand{\uj}{{\mathbf{j}}}
\newcommand{\uh}{{\mathbf{h}}}
\newcommand{\uk}{{\mathbf{k}}}
\newcommand{\ul}{{\mathbf{l}}}
\newcommand{\shuffle}{{\diamond}}
\newcommand{\shq}{{\,\shuffle\,}}
\newcommand{\shqi}{{\,\bar{\shuffle}\,}}
\newcommand{\shqqi}{\,\shuffle_{q,q^{-1}}\,}

\newcommand{\preq}{\,{\preceq}\,}

\newcommand{\Wr}{{\W}}

\newcommand{\Lr}{\L}

\newcommand{\ind}{{\operatorname{Ind}}}
\newcommand{\res}{{\operatorname{Res}}}

\newcommand{\End}{{\operatorname{End}}}

\newcommand{\zero}{{\bar{0}}}
\newcommand{\one}{{\bar{1}}}

\newcommand{\iso}{{\mathsf{iso}}}
\newcommand{\niso}{{\mathsf{n}\text{-}\mathsf{iso}}}

\newcommand{\yy}{\mbox{\RIGHTcircle}}

\newcommand{\la}{{\langle}}
\newcommand{\ra}{{\rangle}}

\newcommand{\ad}{{\mathrm{ad}}}

\newcommand{\andeqn}{\,\,\,\,\,\, {\mbox{and}} \,\,\,\,\,\,}
\newcommand{\QED}{\rule{0.4em}{2ex}}

\newcommand{\set}[1]{\left\{#1\right\}}

\newcommand{\ang}[1]{\left\langle#1\right\rangle}

\renewcommand{\bar}[1]{\overline{#1}}

\newcommand{\Qq}{\Q(q)}

\newcommand{\bbinom}[2]{\begin{bmatrix}#1 \\ #2\end{bmatrix}}

\everymath=\expandafter{\the\everymath\displaystyle} %causes all math to be \displaystyle

\newcommand{\mc}{\mathcal}
\newcommand{\mf}{\mathfrak}

\newcommand{\varep}{\varepsilon}

\input xy
\xyoption{all}

\begin{document}
\title[Quantum Shuffles and Quantum Supergroups of Basic Type]{Quantum Shuffles and Quantum Supergroups \\ of Basic Type}

\author[Clark, Hill, Wang]{Sean Clark, David Hill, and Weiqiang Wang}
\address{Department of Mathematics, University of Virginia, Charlottesville, VA 22904}
\email{sic5ag@virginia.edu (Clark), \quad deh4n@virginia.edu (Hill), \quad ww9c@virginia.edu (Wang)}

\thanks{}
%\subjclass[2000]{Primary 20C08; Secondary 17B37}

\begin{abstract}
We initiate the study of several distinguished bases for the positive half of a quantum supergroup $U_q$ associated
to a general super Cartan datum $(\I, (\cdot,\cdot))$ of basic type inside a quantum shuffle superalgebra.
The combinatorics of words
for an arbitrary total ordering on  $\I$ is developed in connection with the root system associated to $\I$.
The monomial, Lyndon, and PBW bases of $U_q$ are constructed, and moreover,
a direct proof of the orthogonality of the PBW basis is provided within the framework of quantum shuffles.
Consequently, the canonical basis is constructed for $U_q$
associated to the standard super Cartan datum of type $\gl(n|1)$, $\mf{osp}(1|2n)$, or $\mf{osp}(2|2n)$
or an arbitrary non-super Cartan datum.
In the non-super case, this refines Leclerc's work and
provides a new self-contained construction of canonical bases.
The canonical bases of $U_q$, of its polynomial modules, as well as of Kac modules
in the case of quantum $\gl(2|1)$ are explicitly worked out.
\end{abstract}

\maketitle

\setcounter{tocdepth}{1}
 \tableofcontents

%%%%%%%%%%%%%
%%%%%%%%%%%%%
\section{Introduction}

%%%%%%%%%%%%%
\subsection{}

The Drinfeld-Jimbo quantum group associated to a simple Lie algebra admits extremely rich structures
with a wide variety of applications in representation theory, low-dimensional topology, and mathematical physics. In particular,
the positive half admits some remarkable bases with interesting geometric and categorical interpretations, including
PBW bases and  canonical bases introduced by Lusztig \cite{lu1, lu2, lu} (see also
\cite{ka} for another approach to canonical bases
from the viewpoint of crystals).

In contrast, the quantum supergroups associated to a simple Lie superalgebra %of basic type
are not well understood beyond the foundational work of Yamane \cite{ya,ya2}.
As  Lie superalgebras form an important extension of Lie algebras,  it is natural to ask which structural features carry over to the super setting.

Some reasons to hope such a structure exists are the recent categorification results for quantum supergroups in \cite{kh, hw, KKO, EL, ks},
following earlier pioneering works of Khovanov, Lauda and Rouquier \cite{khl, rq}.
However, due to various internal difficulties (e.g. lack of integral forms, isotropic odd roots, lack of positivity due to super signs),
no construction of a canonical basis existed or was even conjectured in the super setting until  recently
the authors \cite{CHW} constructed the
canonical bases for the integrable modules and the positive half of quantum supergroups
associated to the ``anisotropic" super Cartan datum, meaning no isotropic odd simple roots occur.
The anisotropic super Cartan datum
is distinguished among all super Cartan datum  in the sense that the corresponding Lie superalgebras and quantum supergroups
admit a semisimple category of integrable modules
in parallel to the usual Kac-Moody setting.
The only anisotropic super Cartan datum of finite type
corresponds to
{the}
Lie superalgebra $\mf{osp}(1|2n)$.

There are many other finite-dimensional simple Lie superalgebras besides $\mf{osp}(1|2n)$, among which
the most important class
are those of basic type.
Similar to semisimple Lie algebras, the Lie superalgebras of basic type admit non-degenerate even bilinear forms,
root systems, triangular decompositions, and so on (cf.  \cite{kac, CW12}).
However, there is no reasonable semisimple category of
finite-dimensional integrable modules for Lie superalgebras of basic type except for $\mf{osp}(1|2n)$.
Another phenomenon is the existence of non-conjugate simple systems for a general Lie superalgebra of basic type.
The quantum supergroups studied in \cite{ya} are associated to these basic Lie superalgebras.

Let $U_q$ denote the positive half of a quantum supergroup of basic type.
Benkart, Kang, and Kashiwara \cite{bkk}  constructed the crystal (but not the global)
bases for the polynomial representations of quantum $\gl(m|n)$,
and subsequently Kwon \cite{Kw1} constructed crystal bases for Kac modules of quantum $\gl(m|n)$
(also cf. \cite{Kw2} in the case of $\mathfrak{osp}(r|2n)$ and \cite{MZ} in the case of $\mathfrak{osp}(1|2n)$);
none of these authors constructed crystal bases or canonical bases for $U_q$.
As the works \cite{CHW,cflw} helped us to lift the mental block on the existence of canonical bases
for a class of quantum supergroups, we are motivated to reexamine the possibilities
for quantum supergroups of basic type.

Since the basic Lie superalgebras include simple Lie algebras as limiting cases, we require
an approach toward canonical bases which would work equally well for the usual quantum group of finite type.
However, Lusztig's geometric approach (via either perverse sheaves or quiver geometry) is not applicable for now,
while Kashiwara's algebraic approach requires a semisimple category of integrable modules and
hence works well only for the anisotropic quantum supergroups.

%%%\subsection{Goal}
\subsection{}
In this paper, we provide a first step toward the construction of canonical bases for quantum supergroups of basic type,
and give a description of $U_q$ which we believe will be useful for future studies on categorification
(cf. \cite{klram2,hks, hmm,mc,bkmc}).
Our approach through quantum shuffles is inspired by
the work of Leclerc \cite{lec} which, in turn, builds on other foundational works
of Lothaire, J.A.~Green,  Lalonde-Ram, and Rosso
\cite{grn, lo, lr, ro2} on relations among combinatorics of words, root systems,
quantum groups and quantum shuffles.
In this paper, we systematically develop a super version of the aforementioned works, and almost always
work in the most general setting of arbitrary (not merely the standard) simple systems of basic type.
The passage from the classical to the super setting is highly nontrivial, due largely to the lack of positivity in the formula
for the shuffle product. Moreover, our results go beyond those appearing in the literature, leading to new combinatorial proofs of classical results on quantized Lie algebras.

Among other results, we construct a family of monomial bases and orthogonal PBW bases of $U_q$,
one for each total ordering of the index set $\I$ labeling the simple roots.
We then construct an integral form in types $\gl(m|n)$, $\mf{osp}(1|2n)$ and $\mf{osp}(2|2n)$,
which yield a canonical basis for $U_q$ when the Cartan data is of type $\gl(m|1)$, $\mf{osp}(1|2n)$ and $\mathfrak{osp}(2|2n)$. We are also able to obtain a bar-invariant \emph{psuedo-canonical} basis for $\gl(m|n)$. However, this basis fails to be almost orthogonal with respect to the bilinear form and is not independent of the chosen ordering on $\I$.

Unlike in the non-super setting, the PBW bases constructed here are not known to be orthogonal \emph{a priori}.
To obtain this result, we generalize a main result of Leclerc \cite[Theorem 36]{lec}
and prove it directly from the combinatorics of Lyndon words
(Leclerc's proof used the orthogonality of PBW bases due to Lusztig); see Lemma~\ref{L:Lyndon max word}
and Theorem~\ref{T:max word}. In the special case of the natural ordering on $\I$ given in Table~ 1,
Yamane \cite{ya} constructs a PBW basis and proves that it is orthogonal through a case-by-case analysis.
Our proof is type independent for almost all orderings on $\I$.
% (see Definition \ref{D:STAR}), though it requires some case-by case analysis for 6 orderings in type $F(3|1)$
%and 2 orderings in type $G(3)$ (however, see also Remark~\ref{R:extension}).
Our argument applies equally well to the Cartan-Killing root datum,
yielding an independent proof of the orthogonality of the PBW bases and a new self-contained algebraic construction of the canonical basis of
the positive half of a Drinfeld-Jimbo quantum group of finite type. After completion of this paper, we learned of a similar construction of orthogonal PBW-type bases for Nichols algebras appearing in \cite{a}.

%%%%%%%%%%%%%%%%%%%
\subsection{}

We now provide a detailed description of the main results of the paper section by section.
In the preliminary Section~\ref{S:QuantGrps}, we collect various basic results on quantum superalgebras of basic type,
most of which can be found in Yamane's papers \cite{ya, ya2}.

In Section~\ref{S:Quantum Shuffle Algebra},
generalizing the work of Rosso \cite{ro2} and Green \cite{grn}, we embed the positive half of a quantum supergroup $U_q$ associated
to a general Cartan datum $(\I, (\cdot,\cdot))$ of basic type in a quantum shuffle superalgebra.
This should be viewed as a dual version to a construction of Lusztig who realized $U_q$ as a quotient of a free algebra
by the radical of a bilinear form. In the super setting we use (a variant of) a
non-degenerate bilinear form on $U_q$ constructed by Yamane \cite{ya}.

The combinatorics of super words, such as dominant words (also known as good words) and Lyndon words, is then
developed systematically in Section~\ref{S: Combinatorics of Words}. Superizing the constructions of Leclerc \cite{lec},
we construct monomial bases of $U_q$. More significantly,
we develop a \emph{highest word} theory for $U_q$ and establish a bijection between the set of dominant Lyndon words
and the reduced root system associated to $\I$,
generalizing a fundamental result of Lalonde-Ram \cite{lr}. Finally, we construct an auxiliary Lyndon basis for $U_q$
and obtain Lemma~\ref{L:Lyndon max word}.

In Section~\ref{S:Orthogonal PBW}, we give a construction of PBW bases of $U_q$. From Lemma~\ref{L:Lyndon max word} we deduce Theorem~\ref{T:max word}, prove a Levendorskii-Soibelman type formula, and prove that these bases are orthogonal, see Theorem~\ref{T:LS formula}, Lemma~\ref{Conj:PBW coproduct} and Theorem~\ref{T:OrthogonalPBW}. We note that Lemma~\ref{Conj:PBW coproduct} can be viewed as a combinatorial analog of \cite[Lemma 3.2]{mc}.

In Section~\ref{S:Computations},
we compute the dominant Lyndon words and root vectors explicitly for quantum supergroups of type $A$-$D$.
These PBW root vectors are very similar to those defined in \cite{ya}, though we express them in the basis of words. Additionally, we compute the inner product between any two root vectors. This information is also contained in \cite[$\S10.3$]{ya}, but as our sign convention on the bilinear form differs from that in \emph{loc. cit.} we derive the formulas directly. Theorem~\ref{T:OrthogonalPBW} explains how to compute the norm of any PBW basis vector.

In Section~\ref{S:Canonical Bases}, we introduce the integral form of $U_q$, where we
have to restrict ourselves to the standard simple systems, and to types $\gl(m|n)$, $\mf{osp}(1|2n)$
and $\mathfrak{osp}(2|2n)$, as well as any non-super type.
In the non-super specialization, this allows us to give a new self-contained
algebraic construction of a canonical basis of $U_q$;
more importantly, we obtain a canonical basis  of $U_q$ in types $\gl(m|1)$, $\mf{osp}(1|2n)$ and $\mf{osp}(2|2n)$.

The case of $\gl(2|1)$ is studied in detail in Section~\ref{S:gl(21)}.
Explicit formulas for the canonical basis of $U_q$ were already given in \cite{kh}.
We show that the canonical basis of $U_q$ descends to a
canonical basis of every polynomial representation and every Kac module of quantum $\gl(2|1)$.
On the other hand, we show that the canonical basis of $U_q$ fails to descend  to a canonical basis for certain
finite-dimensional simple modules of quantum $\gl(2|1)$. We conjecture these phenomena hold for general
$\gl(m|1)$ case.

%%%\subsection{}

\vspace{.2cm}
\noindent \textbf{Acknowledgements.}
W.W. is partially supported
by the NSF grant DMS-1101268.
The authors thank Institute of Mathematics, Academia Sinica,
Taipei for providing an excellent working environment and support, which greatly facilitated this research.
We also thank Bernard Leclerc for helpful discussions and clarifications regarding his paper.

%%%%%%%%%%%%%%%%%%%%%
%%%%%%%%%%%%%%%%%%%%%
\section{Quantum Supergroups of Basic Type}
\label{S:QuantGrps}

In this section, we review some fundamental properties of the positive half of a quantum supergroup of basic type,
including the bilinear form and defining relations.

%%%%%%%%%%%%%
\subsection{Root Data}\label{SS:Root Data}

Let $\g=\g_\zero\oplus\g_\one$ be a complex basic Lie superalgebra of rank $m+n+1=N$ of type $A$-$G$ \cite{kac, CW12}.
Let $\widetilde{\Phi}=\widetilde{\Phi}_\zero\sqcup\widetilde{\Phi}_\one$ be the root system for $\g$, and let
$$\Phi=\Phi_\zero \sqcup \Phi_\one =\left\{\bt\in\widetilde{\Phi}\,\big|\, \frac12\bt\notin\widetilde{\Phi}\right\}$$
be the reduced root system for $\g$, where $\Phi_s = \Phi \cap \widetilde{\Phi}_s$, for $s\in \{\zero,\one\}$;
as usual $\zero$ and $\one$ here and below indicate the even and odd (roots) respectively.
We will work with $\Phi$ and not $\widetilde{\Phi}$ until Section~\ref{S:Canonical Bases}.
Let $\Pi =\Pi_\zero \sqcup \Pi_\one =\{\af_i\mid i\in \I\}$ be a simple system for $\widetilde{\Phi}$ which is
labelled by $\I=\I_\zero\sqcup\I_\one=\{1,\ldots,N\}$, and let
$\Phi^+ \subseteq \Phi$ be the corresponding set of positive roots.
We define the parity function $p(\cdot)$ on $\I$ by letting $p(i) =s$ for $i\in \I_s$ with $s \in \{\zero, \one\}$.
Let $Q$ be the root lattice.
The monoid $Q^+ :=\bigoplus_{i\in \I}\Z_{\geq0}\af_i$ is $\Z_2$-graded by declaring $p(\af_i)=p(i)$ and extending linearly.
We further decompose
$$\Phi_\one=\Phi_\iso\sqcup\Phi_\niso$$
where $\Phi_\iso$ (resp. $\Phi_\niso$) is the set of isotropic (resp. non-isotropic) odd roots. Decompose
$\Pi_\one=\Pi_\iso\sqcup\Pi_\niso$ (resp. $\I_\one=\I_{\iso}\sqcup\I_{\niso}$) accordingly.

In Table~1 below,  we list the Dynkin diagrams which arise from an arbitrary choice of $\Phi^+$ (for type $A$-$D$)
and label the simple roots according to the labels on the nodes of the corresponding diagram.
The diagrams labelled with $(\star)$ in types $F(3|1)$ and $G(3)$ will be referred to as \textbf{distinguished diagrams}
($F(3|1)$ is often referred to as $F(4)$ in literature).
The simple roots may be even, odd isotropic, or odd non-isotropic,
and we will label the corresponding nodes $\fullmoon$, $\otimes$, and $\newmoon$, respectively.
We will use the notation $\odot$ to denote a simple root which may be either
odd isotropic or even, and $\yy$ for a simple root which may be either odd non-isotropic or even.

 \begin{center}
  \vspace{.5cm}
 {Table 1: Dynkin diagrams for general simple systems}
  \vspace{.5cm}
%\begin{table}[t]\label{T:Dynkin}
% \caption{Dynkin diagrams}
%\begin{center}
\begin{tabular}{|c|c|}\hline
$A(m,n)$&
$$ \xy
(-30,0)*{\odot};(-20,0)*{\odot}**\dir{-};(-15,0)*{\cdots};(-10,0)*{\odot};(0,0)*{\odot}**\dir{-};
(0,0)*{\odot};(10,0)*{\odot}**\dir{-};
(15,0)*{\cdots};(20,0)*{\odot};(30,0)*{\odot}**\dir{-};
(-30,-4)*{\scriptstyle 1};(-20,-4)*{\scriptstyle 2};(-10,-4)*{\scriptstyle n};(0,-4)*{\scriptstyle n+1};(10,-4)*{\scriptstyle n+2};
(20,-4)*{\scriptstyle m+n};(30,-4)*{\scriptstyle m+n+1};
(0,-8)*{};(0,8)*{};
\endxy$$
\\\hline
$B(m,n+1)$&
$$ \xy
(-30,0)*{\odot};(-20,0)*{\odot}**\dir{-};(-15,0)*{\cdots};(-10,0)*{\odot};(0,0)*{\odot}**\dir{-};
(0,0)*{\odot};(10,0)*{\odot}**\dir{-};
(15,0)*{\cdots};(20,0)*{\odot};(30,0)*{\yy}**\dir{=};(25,0)*{>};
(-30,-4)*{\scriptstyle 1};(-20,-4)*{\scriptstyle 2};(-10,-4)*{\scriptstyle n};(0,-4)*{\scriptstyle n+1};(10,-4)*{\scriptstyle n+2};
(20,-4)*{\scriptstyle m+n};(31,-4)*{\scriptstyle m+n+1};
(0,-8)*{};(0,8)*{};
\endxy$$
\\\hline
$C(n+1)$&
$$ \xy
(-20,0)*{\odot};(-10,0)*{\odot}**\dir{-};(-5,0)*{\cdots};(0,0)*{\odot};(10,0)*{\odot}**\dir{-};
(10,0)*{\odot};(20,0)*{\fullmoon}**\dir{=};(15,0)*{<};
(-20,-4)*{\scriptstyle 1};(-10,-4)*{\scriptstyle 2};
(10,-4)*{\scriptstyle n};(20,-4)*{\scriptstyle n+1};
(0,-8)*{};(0,8)*{};
\endxy$$
\\
$D(m,n+1)$&
$$\xy
(-30,0)*{\odot};(-20,0)*{\odot}**\dir{-};(-15,0)*{\cdots};
(-10,0)*{\odot};(0,0)*{\odot}**\dir{-};(0,0)*{\odot};(10,0)*{\odot}**\dir{-};
(15,0)*{\cdots};(20,0)*{\odot};(27,5)*{\fullmoon}**\dir{-};(20,0)*{\odot};(27,-5)*{\fullmoon}**\dir{-};
(-30,-4)*{\scriptstyle 1};(-20,-4)*{\scriptstyle 2};(-10,-4)*{\scriptstyle n};(0,-4)*{\scriptstyle n+1};(10,-4)*{\scriptstyle n+2};
(33,5)*{\scriptstyle m+n};(34,-5)*{\scriptstyle m+n+1};
(0,-8)*{};(0,8)*{};
\endxy$$
\\
&
$$\xy
(-30,0)*{\odot};(-20,0)*{\odot}**\dir{-};(-15,0)*{\cdots};
(-10,0)*{\odot};(0,0)*{\odot}**\dir{-};(0,0)*{\odot};(10,0)*{\odot}**\dir{-};
(15,0)*{\cdots};(20,0)*{\odot};(27,5)*{\otimes}**\dir{-};(20,0)*{\odot};(27,-5)*{\otimes}**\dir{-};
(27,5)*{\otimes};(27,-5)*{\otimes}**\dir{=};
(-30,-4)*{\scriptstyle 1};(-20,-4)*{\scriptstyle 2};(-10,-4)*{\scriptstyle n};(0,-4)*{\scriptstyle n+1};(10,-4)*{\scriptstyle n+2};
(33,5)*{\scriptstyle m+n};(34,-5)*{\scriptstyle m+n+1};
(0,-8)*{};(0,8)*{};
\endxy$$
\\\hline
$F(3|1)$&
$$\xy (0,8)*{};(-20,0)*{(\star)};
(-15,0)*{\fullmoon};(-5,0)*{\fullmoon}**\dir{-};
(-5,0)*{\fullmoon};(5,0)*{\fullmoon}**\dir{=};(0,0)*{>};(5,0)*{\fullmoon};(15,0)*{\otimes}**\dir{-};
(-15,-4)*{\scriptstyle 1};(-5,-4)*{\scriptstyle 2};(5,-4)*{\scriptstyle 3};(15,-4)*{\scriptstyle 4};
(0,-8)*{};
\endxy$$
\\&
$$\xy (0,8)*{};
{\ar@3{-}(-15,0)*{\fullmoon};(-5,0)*{\otimes}};(-10,0)*{>};
(-5,0)*{\otimes};(5,0)*{\fullmoon}**\dir{=};(0,0)*{<};(5,0)*{\fullmoon};(15,0)*{\fullmoon}**\dir{-};
(-15,-4)*{\scriptstyle 1};(-5,-4)*{\scriptstyle 2};(5,-4)*{\scriptstyle 3};(15,-4)*{\scriptstyle 4};
(0,-8)*{};
\endxy$$ \hspace{.25in}
$$\xy (0,8)*{};
{\ar@3{-}(-15,0)*{\fullmoon};(-5,0)*{\otimes}};(-10,0)*{>};
(-5,0)*{\otimes};(5,0)*{\fullmoon}**\dir{-};(5,0)*{\fullmoon};(15,0)*{\fullmoon}**\dir{=};(10,0)*{<};
(-15,-4)*{\scriptstyle 1};(-5,-4)*{\scriptstyle 2};(5,-4)*{\scriptstyle 3};(15,-4)*{\scriptstyle 4};
(0,-8)*{};
\endxy$$
\\&
$$\xy (0,8)*{};
{\ar@3{-}(-10,0)*{\otimes};(-5,7)*{\fullmoon}};(-5,7)*{\fullmoon};(0,0)*{\otimes}**\dir{-};
(-10,0)*{\otimes};(0,0)*{\otimes}**\dir{=};(0,0)*{\otimes};(10,0)*{\fullmoon}**\dir{=};(5,0)*{<};
(-10,-4)*{\scriptstyle 1};(-5,9)*{\scriptstyle 2};(0,-4)*{\scriptstyle 3};(10,-4)*{\scriptstyle 4};
(0,-8)*{};
\endxy$$ \hspace{.25in}
$$\xy (0,8)*{};
(-10,0)*{\otimes};(-5,7)*{\fullmoon}**\dir{-};(-5,7)*{\fullmoon};(0,0)*{\otimes}**\dir{-};
(-10,0)*{\otimes};(0,0)*{\otimes}**\dir{=};(0,0)*{\otimes};(10,0)*{\fullmoon}**\dir{=};(5,0)*{<};
(-10,-4)*{\scriptstyle 1};(-5,9)*{\scriptstyle 2};(0,-4)*{\scriptstyle 3};(10,-4)*{\scriptstyle 4};
(0,-8)*{};
\endxy$$
\\\hline
$G(3)$&
$$ \xy(0,5)*{};(-15,0)*{(\star)};
{\ar@3{-}(0,0)*{\fullmoon};(10,0)*{\fullmoon}};(5,0)*{<};(-10,0)*{\otimes};(0,0)*{\fullmoon}**\dir{-};
(-10,-4)*{\scriptstyle 1};(0,-4)*{\scriptstyle 2};(10,-4)*{\scriptstyle 3};
(0,-8)*{};(0,8)*{};
\endxy$$ \hspace{.25in}
$$ \xy(0,5)*{};
{\ar@3{-}(0,0)*{\otimes};(10,0)*{\fullmoon}};(5,0)*{<};(-10,0)*{\otimes};(0,0)*{\otimes}**\dir{-};
(-10,-4)*{\scriptstyle 1};(0,-4)*{\scriptstyle 2};(10,-4)*{\scriptstyle 3};
(0,-8)*{};(0,8)*{};
\endxy$$
\\&
$$ \xy(0,5)*{};
{\ar@3{-}(0,0)*{\otimes};(10,0)*{\fullmoon}};(5,0)*{<};(-10,0)*{\newmoon};(0,0)*{\otimes}**\dir{-};
(-10,-4)*{\scriptstyle 1};(0,-4)*{\scriptstyle 2};(10,-4)*{\scriptstyle 3};
(0,-8)*{};(0,8)*{};
\endxy$$ \hspace{.25in}
$$\xy (0,5)*{};
{\ar@2{-}(-5,0)*{\otimes};(0,7)*{\fullmoon}};(0,7)*{\fullmoon};(5,0)*{\otimes}**\dir{-};
{\ar@3{-}(-5,0)*{\otimes};(5,0)*{\otimes}};
(-5,-4)*{\scriptstyle 1};(0,9)*{\scriptstyle 2};(5,-4)*{\scriptstyle 3};
(0,-8)*{};(0,8)*{};
\endxy$$
\\\hline
$D(2|1;\af)$&
$$\xy
(0,0)*{\otimes};(7,5)*{\fullmoon}**\dir{-};(0,0)*{\otimes};(7,-5)*{\fullmoon}**\dir{-};(2,4)*{\scriptstyle -1};(2,-4)*{\scriptstyle 1+\af};
(-3,0)*{\scriptstyle 1};(10,5)*{\scriptstyle 2};(10,-5)*{\scriptstyle 3};
(0,8)*{};(0,-8)*{};
\endxy$$ \hspace{.25in}
$$\xy
(0,0)*{\otimes};(7,5)*{\fullmoon}**\dir{-};(0,0)*{\otimes};(7,-5)*{\fullmoon}**\dir{-};(2,4)*{\scriptstyle -\af};(2,-4)*{\scriptstyle 1+\af};
(-3,0)*{\scriptstyle 1};(10,5)*{\scriptstyle 2};(10,-5)*{\scriptstyle 3};
(0,8)*{};(0,-8)*{};
\endxy$$
\\$(\af\in\Z_{>0})$&
$$\xy
(-5,0)*{\otimes};(0,7)*{\otimes}**\dir{-};
(-5,0)*{\otimes};(5,0)*{\otimes}**\dir{-};
(0,7)*{\otimes};(5,0)*{\otimes}**\dir{-};
(-6,-3)*{\scriptstyle 1};(0,10)*{\scriptstyle 2};(6,-3)*{\scriptstyle 3};
(4,4)*{\scriptstyle \af};(0,-2)*{\scriptstyle -1-\af};
(0,8)*{};(0,-8)*{};
\endxy$$
\\\hline
\end{tabular}
\end{center}
%\end{table}

The basic Lie superalgebras are examples of symmetrizable
contragredient Lie superalgebras associated to (super generalized) Cartan matrices \cite{kac}, which
are endowed with a non-degenerate even supersymmetric bilinear form.
Let $A=(a_{ij})_{i,j\in \I}$ be a symmetrizable Cartan matrix for $\g$.
Let $d_i$, $i\in \I$, be positive integers satisfying $d_ia_{ij}=d_ja_{ji}$, and $\gcd(d_i\mid i\in \I)=1$.
Define a symmetric bilinear form $(\cdot,\cdot):Q\times Q\longrightarrow \Z$ by letting
$$ (\af_i,\af_j) =d_i a_{ij},\;\;\;i,j\in \I.$$
In particular, we have the following basic property.

\begin{lem}
The following are equivalent for $i\in \I$:\

\[(1) \; a_{ii}=0;\qquad (2) \; i\in I_{\iso};\qquad (3) \; (\af_i,\af_i)=0.\]

\end{lem}

We set the notation
\begin{align}\label{E:pi}
\pi=-1,
\end{align}
which will be used to keep track of super-signs. Set
\begin{align}\label{E:signs}
s_{ij}=\begin{cases} 1&\mbox{if }(\af_i,\af_j)\geq0\\\pi&\mbox{if }(\af_i,\af_j)<0.\end{cases}
\end{align}
We call the triple $(\I,\Pi, (\cdot,\cdot)\ )$ a {\bf Cartan datum of basic type}.

%%%%%%%%%%%%%%%%
\subsection{Quantum superalgebra $U_q$}

 Let $\g=\n^-\oplus \h \oplus \n^+$ be the triangular decomposition of $\g$.
The quantized enveloping algebra $U_q(\g)$ with Chevalley generators $e_i,f_i,k_i^{\pm1}$ ($i\in \I$)
has been systematically defined and studied in \cite{ya} (here we choose to adopt a more standard version
without an extra parity operator denoted by $\sigma$ in {\em loc. cit.}).
Let $\Uq=\Uq(\n^+)$ be the subalgebra of $\Uq(\g)$ generated by the elements $e_i$ ($i\in \I$).
By definition, $\Uq$ is a quotient of a free superalgebra on the generators $e_i$
by the radical of the bilinear form,
just as  defined by \cite[Part~ I]{lu} in the non-super setting. We will use a rescaling of this bilinear form;
see Proposition \ref{T:BilinearForm} below.

The algebra $\Uq$ is $Q^+$-graded by declaring that the degree of $e_i$ is $\af_i$:
$$\Uq=\bigoplus_{\nu\in Q^+}U_{q,\nu}.$$
For homogeneous $u\in \Uq$, we write $|u|$ for the degree of $u$ in this grading.
There is also a $\Z_2$-grading on $\Uq$ by setting $p(u)=p(\nu)$ if $|u|=\nu$.

The next proposition is standard (see e.g. \cite{ya}); in the case of $B(0,n+1)$ the novel bar involution was introduced
in \cite{hw}.

\begin{prp}
 \label{P:autoUq}
The algebra $\Uq$ admits the following symmetries:
\begin{enumerate}
\item A $\Q(q)$-linear anti-automorphism $\tau:\Uq\longrightarrow\Uq$ defined by
\begin{align}\label{E:tau}
\tau(e_i)=e_i
\text{ for all }i\in I
\andeqn \tau(uv)=\tau(v)\tau(u).
\end{align}
\item A $\Q$-linear automorphism $\overline{}:\Uq\longrightarrow\Uq$ (called a bar involution) defined by
\begin{align}\label{E:bar}
\overline{q}=\begin{cases}\pi q^{-1}&\mbox{if $\Uq$ is of type }B(0,n+1),\\
        q^{-1}&\mbox{otherwise,}\end{cases}
         \quad \overline{e_i} =e_i
         \text{ for all }i\in I,
\andeqn \overline{uv}=\bar{u}\;\bar{v}.
\end{align}
\item A $\Q$-linear anti-automorphism $\sm:\Uq\longrightarrow\Uq$ defined by
\begin{align}\label{E:sigma}
\sigma(u)=\overline{\tau(u)}.
\end{align}
\end{enumerate}
\end{prp}

\begin{proof}
The existence of the anti-automorphism $\tau$ is proved in \cite[Lemma 6.3.1]{ya2}. The existence of the bar involution
can be proved using similar arguments to those in \cite[\S 1.2.12]{lu} (see also \cite[Cor 1.4.4]{CHW1}).
\end{proof}

The algebra $\Uq$ has the structure of a twisted bi-superalgebra with coproduct defined on the generators by
$$\Dt(e_i)=e_i\otimes 1+1\otimes e_i.$$
The coproduct is an algebra homomorphism $\Dt:\Uq\longrightarrow\Uq\otimes\Uq$ with respect to the twisted multiplication on $\Uq\otimes\Uq$:
$$(a\otimes b)(c\otimes d)=\pi^{p(b)p(c)}q^{-(|b|,|c|)}ac\otimes bd,$$
for $a,b,c,d\in \Uq$ homogeneous in the $(Q^+\times\Z_2)$-grading.

\subsection{Bilinear Forms on $U_q$}
The goal of this section is to establish the existence of the bilinear form described in Proposition~\ref{T:BilinearForm},
a variant of which first appeared in \cite{ya}.
Indeed, let $(\cdot,\cdot)_{\mathrm{sgn}}$ be the form appearing in \emph{loc.cit.}.
This form satisfies Conditions (B1)-(B3) in the statement of
Proposition \ref{T:BilinearForm}
below, but with the $(q,\pi)$-bialgebra structure on $\Uq\otimes\Uq$ replaced
by a $(q^{-1},\pi)$-bialgebra structure and with the bilinear form satisfying
\begin{equation}
\label{eq:sgn}
(x'\otimes x'',y'\otimes y'')_{\mathrm{sgn}}=\pi^{p(x'')p(y')}(x',y')_{\mathrm{sgn}}(x'',y'')_{\mathrm{sgn}}.
\end{equation}

In order to deduce the proposition, we begin with some general comments about rescaling of bilinear forms.
To this end, let $t: Q^+\times Q^+\rightarrow \Q(q)^\times$ be a function such that
\[t(\lambda,\nu)=t(\nu,\lambda),\quad t(\lambda+\nu,\eta)=t(\lambda,\eta)t(\nu,\eta),\quad
t(\lambda,\nu+\eta)=t(\lambda,\nu)t(\lambda,\eta).\]

\begin{lem}\label{L:symmetric twist} Assume we have a bilinear form $\{\cdot,\cdot\}$ on $U_q$ such that
\begin{enumerate}
\item For $\mu\neq \nu$ in $Q^+$, $\{U_{q,\mu},U_{q,\nu}\}=0$;
\item $\{1,1\}=1$ and $\{e_i,e_i\}\neq 0$, for all $i\in \I$;
\item $\{xy,z\}=\{x\otimes y,\Dt(z)\}$, for $x,y,z \in U_q$, where $\{x\otimes y,x'\otimes y'\}=t(|y|,|x'|)\{x,x'\}\{y,y'\}$.
\end{enumerate}
 Then there is a symmetric bilinear form $(\cdot,\cdot)$ on $U_q$ such that
\begin{enumerate}
\item[(a)] For $\mu\neq \nu$, $(U_{q,\mu},U_{q,\nu})=0$;
\item[(b)] $(1,1)=1$ and $(e_i,e_i)\neq 0$, for all $i\in \I$;
\item[(c)] $(xy,z)=(x\otimes y,\Dt(z))$, for $x,y,z \in U_q$, where $(x\otimes y,x'\otimes y')=(x,x')(y,y')$.
\end{enumerate}
Specifically, the bilinear form is given by
$(x,y)=t(|x|)^{-1}\{x,y\}$, where \[t(\af_{i_1}+\ldots +\af_{i_n})=\prod_{r<s} t(\af_{i_r},\af_{i_s}).\]
\end{lem}

\begin{proof}
Note that $t(\af_{i_1}+\ldots + \af_{i_n})$ defined above does not depend on the order
because $t$ is symmetric. Since this rescaling is well defined on each
weight space, it suffices to show that the given bilinear form satisfies the
required properties. (a) and (b) are trivially true, and the form
$(\cdot,\cdot)$ is clearly symmetric. For (c), let $x,y,z$ be homogeneous
and $\Dt(z)=\sum z_1\otimes z_2$. Then

\begin{align*}
(xy,z)&=t(|x|+|y|)^{-1}\set{xy,z}=t(|x|+|y|)^{-1}\set{x\otimes y,\Dt(z)}\\
&=t(|x|+|y|)^{-1}\sum t(|y|,|z_1|)\set{x,z_1}\otimes \set{y,z_2}\\
&=t(|x|+|y|)^{-1}\sum t(|y|,|x|)t(|x|)t(|y|)(x,z_1)\otimes (y,z_2).
\end{align*}

Observing that $t(|x|,|y|)t(|x|)t(|y|)=t(|x|+|y|)$ finishes the proof.
\end{proof}

The following  is a variant of a theorem due to Yamane \cite[Section 2]{ya}.

\begin{prp}  \label{T:BilinearForm}
There exists a unique nondegenerate symmetric bilinear form
$(\cdot,\cdot):\Uq\times\Uq\longrightarrow\Q(q)$ satisfying
\begin{enumerate}
\item[(B1)] $(1,1)=1$;
\item[(B2)] $(e_i,e_j)=\dt_{ij}$, for all $i,j \in \I$;
\item[(B3)] $(x,yz)=(\Dt(x),y\otimes z)$, for all $x,y,z \in U_q$.
\end{enumerate}
Here we have used $(x'\otimes x'',y'\otimes y''):=(x',y')(x'',y'').$
\end{prp}

\begin{proof}
%This follows  almost immediately from \cite[Section 2]{ya}.
Let $(\cdot,\cdot)_{\mathrm{sgn}}$ be the bilinear form appearing in   \cite[Section 2]{ya}. This bilinear form
was shown to satisfy the 3 properties in the proposition with respect to \eqref{eq:sgn}.
Take $t(\mu,\nu)=\pi^{p(\mu)p(\nu)}$
and $\{x,y\}=\overline{(\overline{x},\overline{y})}_{\mathrm{sgn}}$, $x,y\in\Uq$.
Then the bilinear form $(\cdot,\cdot)$ obtained from $\{\cdot,\cdot\}$ satisfies the same properties, by Lemma~\ref{L:symmetric twist}.
\end{proof}

%\begin{rmk}\label{R:Signed Bilinear Form}
In \cite[Proposition 3.3]{hw}, the authors showed directly that the unsigned version of the bilinear form for $U_q$
of type $B(0,n)$ (and other anisotropic Kac-Moody types)  is well-defined.
Our preference for this form is due to the fact that it agrees with a bilinear form arising from categorification.
%\end{rmk}

\begin{prp}\label{P:BilinearForm}
Let $e_i':\Uq\longrightarrow\Uq$ denote the adjoint of left multiplication by $e_i$ with respect to the binear form:
$$(e_iu,v)=(u,e_i'(v)).$$
Then, $e_i'$ satisfies
\begin{enumerate}
\item $e_i'(e_j)=\dt_{ij}$;
\item $e_i'(uv)=e_i'(u)v+\pi^{p(u)p(i)}q^{-(\af_i,|u|)}ue_i'(v)$ for homogenous $u,v\in \Uq$;
\item for homogeneous $u\in \Uq$,  $e_i'(u)=0$ for all $i\in \I$ if and only if $|u|=0$.
\end{enumerate}
\end{prp}

\begin{proof}
Property (1) is obvious from the definition. To prove Property (2), let $x\in\Uq$ and write $\Dt(x)=\sum x_1\otimes x_2$. Then,
\begin{align*}
(x,e_i'(uv))&=(e_ix,uv)\\
    &=((e_i\otimes 1+1\otimes e_i)\Dt(x),u\otimes v)\\
    &=\sum(e_ix_1\otimes x_2,u\otimes v)+\sum \pi^{p(x_1)p(i)}q^{-(\af_i,|x_1|)}(x_1\otimes e_ix_2,u\otimes v)\\
    &=\sum(e_ix_1,u)(x_2,v)+\sum\pi^{p(x_1)p(i)}q^{-(\af_i,|x_1|)}(x_1,u)(e_ix_2,v).
\end{align*}
Note that if a summand of the second sum in the last line above is nonzero, then $|x_1|=|u|$ and $p(x_1)=p(u)$. Therefore,
\begin{align*}
(x,e_i'(uv))&=\sum(e_ix_1,u)(x_2,v)+\sum\pi^{p(u)p(i)}q^{-(\af_i,|u|)}(x_1,u)(e_ix_2,v)\\
&=\sum(x_1,e_i'(u))(x_2,v)+\sum\pi^{p(u)p(i)}q^{-(\af_i,|u|)}(x_1,u)(x_2,e_i'(v))\\
&=\sum(x_1\otimes x_2, e_i'(u)\otimes v+\pi^{p(u)p(i)}q^{-(\af_i,|u|)}u\otimes e_i'(v))\\
&=(x,e_i'(u)v+\pi^{p(u)p(i)}q^{-(\af_i,|u|)}ue_i'(v)).
\end{align*}
Since the form is nondegenerate, (2) follows.

Finally, to prove (3), note that if $|u|=\nu$, then $|e_i'(u)|=\nu-\af_i$. In particular, if $|u|=0$, then $e_i'(u)=0$ for all $i\in \I$.
Conversely, if $e_i'(u)=0$ for all $i$, then we have $(e_{i_1}\cdots e_{i_d}, u)=0$ for all $i_1,\ldots,i_d\in \I$ and $d\geq 1$.
As these monomials span $\bigoplus_{\nu\neq 0}U_{q,\nu}$, and the form is nondegenerate, we must have $|u|=0$.
\end{proof}

\begin{cor}\label{cor:derivationiso}
The subalgebra $\mathcal E$ of $\End_{\Qq}(U_q)$ generated by the $e_i'$ for $i\in I$ is isomorphic to $U_q$ under the identification $e_i\mapsto e_i'$.
\end{cor}
\begin{proof}
Since the bilinear form is nondegenerate, the map $e_i\mapsto e_i'$ defines an anti-isomorphism between $U_q$ and $\mathcal E$;
Composing with the map $\tau$ defined in Proposition \ref{P:autoUq} yields the desired isomorphism.
\end{proof}

%%%%%%%%%%%%%%%%%%%%%%
\subsection{Defining Relations for $U_q$}

Define the $q$-commutator on homogeneous $u,v\in\Uq$ by
$$\ad_q u(v)=[u,v]_q=uv-\pi^{p(u)p(v)}q^{(|u|,|v|)}vu.$$
Define the usual quantum integer and its super analogue for $n\in \Z_{\ge 0}$:
$$[n]=\frac{q^n-q^{-n}}{q-q^{-1}}\andeqn \{n\}=\frac{\pi^nq^n-q^{-n}}{\pi q-q^{-1}}.$$
More generally, for $i\in I$, set $q_i=q^{d_i}$, $\pi_i=\pi^{p(i)}$, and define
\begin{align*}
[n]_i &=\begin{cases}\frac{\pi_i^nq_i^n-q_i^{-n}}{\pi_iq_i-q_i^{-1}}&\mbox{if } i\in\I_\niso,\\
    \frac{q_i^n-q_i^{-n}}{q_i-q_i^{-1}}&\mbox{otherwise,}\end{cases}
\\
 \begin{bmatrix}n\\k\end{bmatrix}_i
 &=\frac{[n]_i[n-1]_i\cdots[n-k+1]_i}{[k]_i!},
\end{align*}
where $n\in\Z$ and $k\in\Z_{\geq 0}$.

\begin{prp}  \cite{ya, ya2}
\label{P:SerreRelations}
The algebra $\Uq$ satisfies the following relations whenever the given Dynkin subdiagram appears:

\begin{enumerate}
\item[(Iso)] $e_ie_j=-e_je_i$ for $i,j\in \I_\one$ with $a_{ij}=0$.
\item[(N-Iso)] For $i\in \I_\zero\cup \I_\niso$ and $i\neq j$,
$$\sum_{r+s=1+|a_{ij}|}(-1)^r  \pi_i^{p(i,j;r)}
\begin{bmatrix}1+|a_{ij}|\\r\end{bmatrix}_i e_i^r e_j e_i^s =0,
$$
where $p(i,j;r)=\binom{r}{2}p(i)+rp(i)p(j)$.
\item[(AB)] For
$$\hspace{.75in}\xy
(-10,0)*{\odot};(0,0)*{\otimes}**\dir{-};(0,0)*{\otimes};(10,0)*{\odot}**\dir{-};
(-10,-4)*{\scriptstyle i};(0,-4)*{\scriptstyle j};(10,-4)*{\scriptstyle k};
\endxy\quad(s_{ij}\neq s_{jk})$$
or
$$\xy
(-10,0)*{\yy};(0,0)*{\otimes}**\dir{=};(0,0)*{\otimes};(10,0)*{\odot}**\dir{-};(-5,0)*{<};
(-10,-4)*{\scriptstyle i};(0,-4)*{\scriptstyle j};(10,-4)*{\scriptstyle k};
\endxy$$
$$
\ad_q e_j\circ\ad_{q}e_k\circ \ad_{q} e_j (e_i)=0.
$$

\item[(CD1)] For
$$\xy
(-10,0)*{\fullmoon};(0,0)*{\otimes}**\dir{=};(0,0)*{\otimes};(10,0)*{\otimes}**\dir{-};(-5,0)*{>};
(-10,-4)*{\scriptstyle i};(0,-4)*{\scriptstyle j};(10,-4)*{\scriptstyle k};
\endxy$$
$$\ad_q e_j\circ \ad_q (\ad_q e_j(e_k))\circ \ad_q e_i\circ \ad_q e_j (e_k)=0.$$
\item[(CD2)] For
$$\xy
(-10,0)*{\odot};(0,0)*{\fullmoon}**\dir{-};(0,0)*{\fullmoon};(10,0)*{\otimes}**\dir{-};(15,0)*{<};(10,0)*{\fullmoon};(20,0)*{\fullmoon}**\dir{=};
(-10,-4)*{\scriptstyle i};(0,-4)*{\scriptstyle j};(10,-4)*{\scriptstyle k};(20,-4)*{\scriptstyle l};
\endxy$$
$$\ad_q e_k\circ \ad_q e_j \circ \ad_q e_k\circ \ad_q e_l\circ \ad_q e_k\circ \ad_q e_j (e_i)=0.$$
\item[(D)] For
$$\xy
(0,0)*{\odot};(7,5)*{\otimes}**\dir{-};(0,0)*{\odot};(7,-5)*{\otimes}**\dir{-};(7,5)*{\otimes};(7,-5)*{\otimes}**\dir{=};
(-4,0)*{\scriptstyle i};(10,5)*{\scriptstyle j};(10,-5)*{\scriptstyle k};
\endxy$$
$$\ad_q e_k\circ \ad_q e_j (e_i)=\ad_q e_j \circ \ad_q e_k(e_i).$$
\item[(F1)] For
$$\xy
{\ar@3{-}(-15,0)*{\fullmoon};(-5,0)*{\otimes}};(-10,0)*{>};
(-5,0)*{\otimes};(5,0)*{\fullmoon}**\dir{=};(0,0)*{<};(5,0)*{\fullmoon};(15,0)*{\fullmoon}**\dir{-};
(-15,-4)*{\scriptstyle 1};(-5,-4)*{\scriptstyle 2};(5,-4)*{\scriptstyle 3};(15,-4)*{\scriptstyle 4};
\endxy$$
$$\ad_q E\circ \ad_q E \circ \ad_q e_4 \circ \ad_q e_3 \circ \ad_q e_2=0,$$
where $E=\ad_q(\ad_q e_1(e_2))\circ \ad_q e_3(e_2)$.
\item[(F2)] For
$$\xy
{\ar@3{-}(-15,0)*{\fullmoon};(-5,0)*{\otimes}};(-10,0)*{>};
(-5,0)*{\otimes};(5,0)*{\fullmoon}**\dir{=};(0,0)*{<};(5,0)*{\fullmoon};(15,0)*{\fullmoon}**\dir{-};
(-15,-4)*{\scriptstyle 1};(-5,-4)*{\scriptstyle 2};(5,-4)*{\scriptstyle 3};(15,-4)*{\scriptstyle 4};
\endxy$$
\begin{align*}\ad_q(\ad_q e_1(e_2)) & \circ \ad_q (\ad_q e_3(e_2)) \circ \ad_q e_3(e_4)
\\
&= \ad_q(\ad_q e_3(e_2))\circ \ad_q (\ad_q e_1(e_2)) \circ \ad_q e_3(e_4).
\end{align*}

\item[(F3)] For
$$\xy
(-10,0)*{\otimes};(0,0)*{\otimes}**\dir{=};(0,0)*{\otimes};(10,0)*{\fullmoon}**\dir{=};(5,0)*{<};
(-10,-4)*{\scriptstyle 1};(0,-4)*{\scriptstyle 3};(10,-4)*{\scriptstyle 4};
\endxy$$
$$\ad_q e_3\circ\ad_{q}e_1\circ \ad_{q} e_3 (e_4)=0.$$
\item[(F4)]  For
$$\xy
{\ar@3{-}(-5,0)*{\otimes};(0,7)*{\otimes}};(-5,0)*{\otimes};(5,0)*{\otimes}**\dir{=};(0,7)*{\otimes};(5,0)*{\otimes}**\dir{-};
(-8,0)*{\scriptstyle 1};(0,10)*{\scriptstyle 2};(8,0)*{\scriptstyle 3};
\endxy$$
$$[3]\ad_q e_i\circ \ad_q e_j (e_k)+[2]\ad_q e_j \circ \ad_q e_i(e_k)=0.$$
\item[(G1)] For
$$ \xy
{\ar@3{-}(0,0)*{\otimes};(10,0)*{\fullmoon}};(5,0)*{<};(-10,0)*{\otimes};(0,0)*{\otimes}**\dir{=};
(-10,-4)*{\scriptstyle 1};(0,-4)*{\scriptstyle 2};(10,-4)*{\scriptstyle 3};
\endxy$$
$$\ad_q E\circ \ad_q E \circ \ad_q E \circ \ad_q e_2(e_1)=0,$$  	
where $E=\ad_q e_2(e_3)$.

\item[(G2)] For
$$ \xy
{\ar@3{-}(0,0)*{\otimes};(10,0)*{\fullmoon}};(5,0)*{<};(-10,0)*{\newmoon};(0,0)*{\otimes}**\dir{=};
(-10,-4)*{\scriptstyle 1};(0,-4)*{\scriptstyle 2};(10,-4)*{\scriptstyle 3};
\endxy$$
$$\ad_q e_2\circ \ad_q e_3 \circ \ad_q e_3\circ \ad_q e_2(e_1)=
  	\ad_q e_3\circ \ad_q e_2 \circ \ad_q e_3\circ \ad_q e_2(e_1).$$

\item[(G3)] For
$$\xy
{\ar@3{-}(-5,0)*{\otimes};(5,0)*{\otimes}};(-5,0)*{\otimes};(0,7)*{\fullmoon}**\dir{=};(0,7)*{\fullmoon};(5,0)*{\otimes}**\dir{-};
(-8,0)*{\scriptstyle 1};(0,10)*{\scriptstyle 2};(8,0)*{\scriptstyle 3};
\endxy$$
$$\ad_q e_1\circ \ad_q e_2 (e_3)-[2]\ad_q e_2 \circ \ad_q e_1(e_3)=0.$$

\item[($D\af$)]  For
$$\xy
(-5,0)*{\otimes};(0,7)*{\otimes}**\dir{-};
(-5,0)*{\otimes};(5,0)*{\otimes}**\dir{-};
(0,7)*{\otimes};(5,0)*{\otimes}**\dir{-};
(-8,0)*{\scriptstyle 1};(0,10)*{\scriptstyle 2};(8,0)*{\scriptstyle 3};
(4,4)*{\scriptstyle \af};(0,-2)*{\scriptstyle -1-\af};
\endxy$$  				
$$[\alpha+1]\ad_q e_1\circ \ad_q e_3 (e_2)+[\alpha]\ad_q e_3 \circ \ad_q e_1(e_2)=0.$$
\end{enumerate}
\end{prp}

\begin{thm}\cite[Proposition 10.4.1]{ya}\label{T:Distinguished Diagrams} If the Dynkin diagram for $\Uq$ is of type $A$-$D$,
or the distinguished diagram in types $F$ and $G$, then the relations given in Proposition \ref{P:SerreRelations} are defining relations for $U_q$.
\end{thm}

%%%%%%%%%%%%%%%%%%%%
%%%%%%%%%%%%%%%%%%%%
\section{Quantum Shuffle Superalgebras}\label{S:Quantum Shuffle Algebra}

In this section, we formulate a quantum shuffle superalgebra associated to a Cartan datum of basic type,
and construct an embedding of the half-quantum superalgebra $U_q$ into a quantum shuffle superalgebra.
These form super generalizations of constructions of Green \cite{grn} and Rosso \cite{ro2}.

%%%%%%%%%%
\subsection{The Homomorphism $\Psi$, I}

Let $(I, \Pi, (\cdot,\cdot))$ be a Cartan datum of basic type.
Let $\F=\F(\I)$ be the free associative superalgebra over $\Q(q)$ generated by $\I$, with parity prescribed by $p(\cdot)$ on $\I$.
Let $\W=\sqcup_{d\geq0} \I^d$ be the set of words in $\F$, i.e., the monoid generated by $\I$.
The identity element is the empty word $\emptyset$, and a general word will be denoted by
$$\mathbf i =(i_1,i_2,\ldots,i_d)=i_1  i_2\cdots i_d.
$$
For $i\in I$ and $k\in \N$, we will use the notation $i^k=\underbrace{ii\ldots i}_{k}$.
Note that $\F$ has a weight space decomposition
$\F=\bigoplus_{\nu\in Q^+}\F_\nu$
by setting $|(i_1, \ldots,i_d)|=\af_{i_1}+\ldots +\af_{i_d}$ and extending linearly.
We define
\begin{align}\label{E:Wnu}
\W_\nu=\W\cap\F_\nu.
\end{align}
Finally, define the length function $\ell:\W\longrightarrow\Z_{\geq0}$ as
\begin{align}\label{E:length}
\ell(i_1,\ldots,i_d)=d.
\end{align}

Let $v\in \Q(q)$. We define the $v$-quantum shuffle product $\shuffle_v:\F\times \F\longrightarrow \F$ inductively by the formula
\begin{align}\label{E:Shuffle}
(xi)\shuffle_v (y  j)=(x\shuffle_v (y j))i +\pi^{(p(x)+p(i))p(j)}v^{-(|x|+\af_i,\af_j)}((xi)\shuffle_v y)j,
\end{align}
and $x\shuffle_v\emptyset=\emptyset\shuffle_v x=x,$ for homogenous $x,y\in\F$ and $i,j\in \I$.
The quantum shuffle products of interest will be those for $v=q$ or $v=q^{-1}$,
so when there is no chance of confusion we will write $\shq=\shuffle_q$ and $\shqi=\shuffle_{q^{-1}}$.

Iterating \eqref{E:Shuffle} above, we obtain
\begin{align}\label{E:Shuffle2}
(i_1,\ldots,i_a)\shq(i_{a+1},\ldots,i_{a+b})=\sum_\sm \pi^{\ep(\sm)}q^{-e(\sm)}(i_{\sm(1)},\ldots,i_{\sm(a+b)}),
\end{align}
where the sum is over minimal coset representatives in $S_{a+b}/S_a\times S_b$,
\begin{align}\label{E:ShuffleExponents}
\ep(\sm)=\sum_{\substack{r\leq a<s\\ \sm(r)<\sm(s)}}p(i_{\sm(r)})p(i_{\sm(s)}),
\andeqn e(\sm)
=\sum_{\substack{r\leq a<s\\ \sm(r)<\sm(s)}}(\af_{i_{\sm(r)}},\af_{i_{\sm(s)}}).
\end{align}
We call each $(i_{\sm(1)},\ldots,i_{\sm(a+b)})$ in \eqref{E:Shuffle2} a shuffle of $(i_1,\ldots,i_a)$ and $(i_{a+1},\ldots,i_{a+b})$.
More generally, given $x,y\in \F$ such that $x=\sum c_w w$ and $y=\sum d_w w$, we say that a word $z\in \W$ occurs as a shuffle
in $x\shq y$ if $z$ is a shuffle of words $w_1,w_2\in \W$ such that $c_{w_1}d_{w_2}\neq 0$.
\begin{prp}\label{P:ShuffleProductProperty} The shuffle product is associative and satisfies
$$x\shq y=\pi^{p(x)p(y)}q^{-(|x|,|y|)} y\shqi x,$$
where we have used the notation $\shqi=\shuffle_{q^{-1}}$.
\end{prp}

\begin{proof}
The proof is straightforward using \eqref{E:Shuffle2}.
\end{proof}

We call $(\mathsf F, \shq)$ the quantum shuffle (super)algebra associated to $\I$.

We now describe the bialgebra structure on $\F$ with respect to the concatenation product, and explain the relationship with the shuffle product. Equip $\mathsf F\otimes \mathsf F$ with the associative product
$$(w\otimes x)(y\otimes z) = \pi^{p(x)p(y)}q^{-(|x|,|y|)}(wy)\otimes(xz),$$
where we use the concatenation product on each tensor factor.
Then, $\delta:\mathsf F\rightarrow
\mathsf F\otimes \mathsf F$ given by
$\delta(i)=i\otimes 1+ 1\otimes i$ is an algebra homomorphism with respect to the concatenation product on both sides.

\begin{lem}
The algebra $\mathsf F$ admits a symmetric bilinear form $(\cdot,\cdot)$
such that $(1,1)=1$,
\begin{align*}
(i,j) =\delta_{i,j}, & \qquad \text{ for  } i,j \in \I,
 \\
(\mathbf i \mathbf j,
\mathbf k) =(\mathbf i\otimes \mathbf j, \delta(\mathbf k) ),
& \qquad \text{ for  }\mathbf i, \mathbf j \in \W
\end{align*}
where $(\mathbf i_1\otimes \mathbf i_2,\mathbf j_1\otimes \mathbf j_2)=(\mathbf i_1,\mathbf j_1)(\mathbf i_2, \mathbf j_2)$.
\end{lem}
\begin{proof}
This can be proved by a standard argument; cf. \cite{lu,CHW1}.
\end{proof}

Note that there is an obvious surjective algebra homomorphism $\psi:\F\rightarrow U_q$
given by $i\mapsto e_i$; moreover, $\Delta\circ \psi=(\psi\otimes \psi) \circ \delta$,
and hence by Proposition~\ref{T:BilinearForm},
$(\mathbf i,\mathbf j)=(\psi(\mathbf i),\psi(\mathbf j))$.

Suppose that $\mathbf i=i_1\cdots i_n$. For any $a<b\in \N$, set $[a.b]=\set{a,a+1,\ldots, b-1, b}$. Then for any subset
$P=\{ k _1<\ldots <  k _m\}$ of
$[1.n]$, define
$\mathbf i_P=i_{ k _1}\cdots i_{ k _m}$
so that $\mathbf i_P$ is a word of length $m\leq n$.
We have
\[\delta(\mathbf i)=\prod_{ k \in [1.n]} \delta(i_ k )
=\prod_{ k \in [1.n]} (i_ k \otimes 1 + 1\otimes i_ k ),\]
where this non-commuting product is taken in the order
$ k =1,\ldots, n$. The last product can be expanded
as a sum $\sum_{P\subseteq [1.n]} z(P)$,
where $z(P)=z_1\ldots z_n$ with $z_ k =i_ k \otimes 1$
if $ k  \in P$ and $z_ k =1\otimes i_ k $
if $ k \in P^c=[1.n]\setminus P$.
Now expanding $z(P)$ using the tensor multiplication rule gives us
\[z(P)=\pi^{\ep(\sm_P)}q^{-e(\sm_P)}
\mathbf i_{P}\otimes \mathbf i_{P^c},\]
where $\sm_P$ is the minimal coset representative in $S_n/S_{n-m}\times S_{m}$ satisfying
$\sm_P([n-m+1.n])=P$ and $\ep(\sm_P)$ and $e(\sm_P)$ are defined in \eqref{E:ShuffleExponents}.
Hence,
%\begin{lem}
for a word $\mathbf i\in \W$ of length $n$, we have
\begin{equation}  \label{eq:deltai}
\delta(\mathbf i)=\sum_{P\subseteq [1.n]}
\pi^{\ep(\sm_P)}q^{-e(\sm_P)}
\mathbf i_P\otimes \mathbf i_{{P^{c}}}.
\end{equation}
%\end{lem}

Let $\mathsf F^*$ be the graded dual of $\mathsf F$.
Then for any word $\mathbf i$ in $\mathsf F$, we
set $f_{\mathbf i}$ to be the dual basis element:
\[f_{\mathbf i}(\mathbf j)=\delta_{\mathbf i \mathbf j},
\qquad \text{ for all } \mathbf i, \mathbf j \in \W.
\]
We endow $\mathsf F^*$ with an associative algebra structure
with multiplication defined by
$$
(fg)(x)=(g\otimes f)(\delta(x)),
\qquad \text{ for }f,g \in \mathsf F^*, x\in \mathsf F.
$$

\begin{lem}
The map $\phi:\mathsf F^*\rightarrow (\mathsf F,\diamond)$
$f_{\mathbf i}\mapsto \mathbf i$ is an isomorphism of algebras.
\end{lem}

\begin{proof}
It is clear that the given map is a vector space isomorphism;
it remains to show the products match.
Let $\mathbf{i}=(i_1,\ldots, i_n)$ and $\mathbf{j}=(j_1,\ldots, j_m)$,
and suppose that $\mathbf{k}$ has weight $|\mathbf i|+ |\mathbf j|$. Then by \eqref{eq:deltai} we have
\[\delta(\mathbf k)=\sum_{P\subseteq [1.n+m]}
\pi^{\ep(\sm_P)}q^{-e(\sm_P)}
\mathbf k_P\otimes \mathbf k_{{P^{c}}}.\]
Then we see that $\lambda_{\mathbf i,\mathbf j}^{\mathbf k}:=(f_{\mathbf{j}}\otimes f_{\mathbf{i}})(\delta(\mathbf k))
=\sum\pi^{\ep(\sm_P)}q^{-e(\sm_P)}$,
where the sum is over $P\subset [1.n+m]$ such that
$\mathbf k_P=\mathbf j$ and $\mathbf k_{{P^{c}}}=\mathbf i$.
Therefore,
\begin{align}\label{E:Tag}
f_{\mathbf i}f_{\mathbf j}=\sum \lambda_{\mathbf i, \mathbf j}^{\mathbf k} f_{\mathbf k}.
\end{align}

On the other hand, by \eqref{E:Shuffle2}
that $\mathbf i \diamond \mathbf j=\sum_{\sigma} \pi^{\ep(\sm)}q^{-e(\sm)} (l_{\sm(1)},\ldots,l_{\sm(m+n)})$,
where $\mathbf i \cdot \mathbf j=(l_1,\ldots,l_{m+n})$,
$\sigma\in S_{n+m}/S_n\times S_m$ is a minimal coset representative, and $P=\set{\sigma(n+1),\ldots, \sigma(n+m)}$.

Let $\mathbf k\in \W_{|\mathbf i|+|\mathbf j|}$. Then $\mathbf k$ appears as a summand
of $\ui\shq \uj$ if and only if $\mathbf k=(l_{\sm(1)},\ldots,l_{\sm(m+n)})$ for some $\sm$
such that $\mathbf k_{\sm([n+1.n+m]}=\mathbf j$ and $\mathbf k_{\sm([1.n])}=\mathbf i$.
In particular, $\sm$ satisfies $\sm=\sm_P$ for $P=\sm([n+1.n+m])$.
Therefore,
\[
\mathbf i \diamond \mathbf j
 =\sum_{\mathbf k} \sum_{\substack{P\subset [1.n+m]\\ \mathbf k_P=\mathbf j,\ \mathbf k_{P^c}=\mathbf i}} \pi^{\ep(\sm_P)}q^{-e(\sm_P)}  \mathbf k
 =\sum_{\mathbf k} \lambda_{\mathbf i, \mathbf j}^{\mathbf k} \mathbf k.
\]
Comparing this to \eqref{E:Tag} shows that $\phi$ is an algebra isomorphism.
\end{proof}

\begin{cor}
 \label{cor:psiUF}
There exists an algebra embedding
$\Psi:U_q\rightarrow (\mathsf F, \diamond_q)$
such that $\Psi(e_i)=i$.
\end{cor}

\begin{proof}
The epimorphism $\psi:\mathsf F\rightarrow U_q$
induces an injective homomorphism of graded duals
$\psi^*:U_q^*\rightarrow \mathsf F^*$.
But since $(\cdot,\cdot)$ on $U_q$ is nondegenerate, $U_q^*\cong U_q$;
on the other hand, we just proved that
$\mathsf F^*\cong (\mathsf F,\diamond)$, and so the composition
$\Psi:U_q \stackrel{\cong}{\longrightarrow} U_q^*\stackrel{\psi^*}{\longrightarrow}
 \mathsf F^*\stackrel{\cong}{\longrightarrow}  (\F,\shq)$ is the desired map.
\end{proof}

Define $\UU=\Psi(U_q)$ to be the subalgebra of $(\F,\shq)$ generated by $\I$.

%%%%%%%%%%%%%%%%%%%%%
\subsection{The Homomorphism $\Psi$, II}

In the case where the diagram for $\Uq$ in
Table~1 is of type $A$-$D$ or the distinguished diagram in types $F$ and $G$,
we give an alternate description of the homomorphism $\Psi$ above.
This new description  of $\Psi$ and then $\UU$ is suitable for computations later on.

For $x,y\in\F$, introduce the notation
\begin{align}\label{E:qtShuffle}
x \shuffle_{q,t}y=x\shuffle_q y - x\shuffle_{t} y.
\end{align}
Then Proposition \ref{P:ShuffleProductProperty} can be rephrased as
%\begin{prp}\label{P:qqi shuffle}
%
\begin{equation}  \label{P:qqi shuffle}
x\shq y - \pi^{p(x)p(y)}q^{(|x|,|y|)}y\shq x=x\shqqi y,
\end{equation}
for $x,y\in \F$ homogeneous.
%\end{prp}
We denote  $\displaystyle i^{\shq r}=\underbrace{ i\shq\cdots\shq i}_{r\mbox{ times}}$ below, and recall $s_{ij}$ from \eqref{E:signs}.

\begin{lem}\label{L:Rk2Shuffle}
The following identities hold in $\F$ whenever the indicated Dynkin subdiagram associated to $\Uq$ appears:
\begin{enumerate}
\item[(Iso)] $i\shq j+j\shq i=0$ for $i,j\in \I_\one$ with $a_{ij}=0$;
\item[(N-Iso)] If $i\neq j$ and $i\in \I_\zero \cup \I_\niso$,
$$\sum_{r+s=1+|a_{ij}|}(-1)^r\pi_i^{p(i,j;r)}\begin{bmatrix}1+|a_{ij}|\\r\end{bmatrix}_i i^{\shq r}\shq j\shq i^{\shq s} =0.
$$

\item[(A/B)] For
$$\hspace{.75in}\xy
(-10,0)*{\odot};(0,0)*{\otimes}**\dir{-};(0,0)*{\otimes};(10,0)*{\odot}**\dir{-};
(-10,-4)*{\scriptstyle i};(0,-4)*{\scriptstyle j};(10,-4)*{\scriptstyle k};
\endxy\quad (s_{ij}\neq s_{jk})$$
or
$$\xy
(-10,0)*{\yy};(0,0)*{\otimes}**\dir{=};(0,0)*{\otimes};(10,0)*{\odot}**\dir{-};(-5,0)*{<};
(-10,-4)*{\scriptstyle i};(0,-4)*{\scriptstyle j};(10,-4)*{\scriptstyle k};
\endxy$$
$$
j\shqqi(k\shqqi(j\shqqi i))=0.
$$
\item[(CD1)] For
$$\xy
(-10,0)*{\fullmoon};(0,0)*{\otimes}**\dir{=};(0,0)*{\otimes};(10,0)*{\otimes}**\dir{-};(-5,0)*{>};
(-10,-4)*{\scriptstyle i};(0,-4)*{\scriptstyle j};(10,-4)*{\scriptstyle k};
\endxy$$
$$j\shqqi((j\shqqi k)\shqqi (i\shqqi (j\shqqi k)))=0.$$
\item[(CD2)] For
$$\xy
(-10,0)*{\odot};(0,0)*{\fullmoon}**\dir{-};(0,0)*{\fullmoon};(10,0)*{\otimes}**\dir{-};(15,0)*{<};(10,0)*{\fullmoon};(20,0)*{\fullmoon}**\dir{=};
(-10,-4)*{\scriptstyle i};(0,-4)*{\scriptstyle j};(10,-4)*{\scriptstyle k};(20,-4)*{\scriptstyle l};
\endxy$$
$$k\shqqi (j \shqqi (k\shqqi (l\shqqi (k\shqqi (j\shqqi i)))))=0.$$
\item[(D)] For
$$\xy
(0,0)*{\odot};(7,5)*{\otimes}**\dir{-};(0,0)*{\odot};(7,-5)*{\otimes}**\dir{-};(7,5)*{\otimes};(7,-5)*{\otimes}**\dir{=};
(-4,0)*{\scriptstyle i};(10,5)*{\scriptstyle j};(10,-5)*{\scriptstyle k};
\endxy$$
$$k\shqqi (j\shqqi i)=j \shqqi (k\shqqi i).$$
\item[(F1)] For
$$\xy
{\ar@3{-}(-15,0)*{\fullmoon};(-5,0)*{\otimes}};(-10,0)*{>};
(-5,0)*{\otimes};(5,0)*{\fullmoon}**\dir{=};(0,0)*{<};(5,0)*{\fullmoon};(15,0)*{\fullmoon}**\dir{-};
(-15,-4)*{\scriptstyle 1};(-5,-4)*{\scriptstyle 2};(5,-4)*{\scriptstyle 3};(15,-4)*{\scriptstyle 4};
\endxy$$
$$\mathcal{E}\shqqi (\mathcal{E} \shqqi (4 \shqqi (3 \shqqi 2)))=0,$$
where
$$\mathcal{E}=(1\shqqi 2)\shqqi (3\shqqi 2)=(q^5+q^2-q^{-2}-q^{-5})(3122+1322)+(q^2-q^{-2})(1232).$$
\item[(F2)] For
$$\xy
{\ar@3{-}(-15,0)*{\fullmoon};(-5,0)*{\otimes}};(-10,0)*{>};
(-5,0)*{\otimes};(5,0)*{\fullmoon}**\dir{=};(0,0)*{<};(5,0)*{\fullmoon};(15,0)*{\fullmoon}**\dir{-};
(-15,-4)*{\scriptstyle 1};(-5,-4)*{\scriptstyle 2};(5,-4)*{\scriptstyle 3};(15,-4)*{\scriptstyle 4};
\endxy$$
\begin{align*}(1\shqqi 2)\shqqi ( (3\shqqi 2)& \shqqi (3\shqqi 4))\\&=
  	(3\shqqi 2)\shqqi ( (1\shqqi 2) \shqqi (3\shqqi 4)).
\end{align*}
\item[(F3)] For
$$\xy
(-10,0)*{\otimes};(0,0)*{\otimes}**\dir{=};(0,0)*{\otimes};(10,0)*{\fullmoon}**\dir{=};(5,0)*{<};
(-10,-4)*{\scriptstyle 1};(0,-4)*{\scriptstyle 3};(10,-4)*{\scriptstyle 4};
\endxy$$
$$3\shqqi(1\shqqi (3\shqqi 4))=0.$$
\item[(F4)] For
$$\xy
{\ar@3{-}(-5,0)*{\otimes};(0,7)*{\otimes}};(-5,0)*{\otimes};(5,0)*{\otimes}**\dir{=};(0,7)*{\otimes};(5,0)*{\otimes}**\dir{-};
(-8,0)*{\scriptstyle 1};(0,10)*{\scriptstyle 2};(8,0)*{\scriptstyle 3};
\endxy$$
$$[3](i\shqqi (j\shqqi k))+[2](j \shqqi (i\shqqi k))=0.$$
\item[(G1)] For
$$ \xy
{\ar@3{-}(0,0)*{\otimes};(10,0)*{\fullmoon}};(5,0)*{<};(-10,0)*{\otimes};(0,0)*{\otimes}**\dir{=};
(-10,-4)*{\scriptstyle 1};(0,-4)*{\scriptstyle 2};(10,-4)*{\scriptstyle 3};
\endxy$$
$$\mathcal{E}\shqqi ( \mathcal{E} \shqqi ( \mathcal{E} \shqqi (2\shqqi 1)))=0,$$  	
where $\mathcal{E}=(2\shqqi 3)=-(q^3-q^{-3})(23)$.

\item[(G2)] For
$$ \xy
{\ar@3{-}(0,0)*{\otimes};(10,0)*{\fullmoon}};(5,0)*{<};(-10,0)*{\newmoon};(0,0)*{\otimes}**\dir{=};
(-10,-4)*{\scriptstyle 1};(0,-4)*{\scriptstyle 2};(10,-4)*{\scriptstyle 3};
\endxy$$
$$2\shqqi (3 \shqqi (3\shqqi (2\shqqi 1)))=
  	3\shqqi (2 \shqqi (3\shqqi (2\shqqi 1))).$$

\item[(G3)] For
$$\xy
{\ar@3{-}(-5,0)*{\otimes};(5,0)*{\otimes}};(-5,0)*{\otimes};(0,7)*{\fullmoon}**\dir{=};(0,7)*{\fullmoon};(5,0)*{\otimes}**\dir{-};
(-8,0)*{\scriptstyle 1};(0,10)*{\scriptstyle 2};(8,0)*{\scriptstyle 3};
\endxy$$
$$1\shqqi (2\shqqi 3)-[2](2 \shqqi (1\shqqi 3))=0.$$
\item[($D\af$)] For
$$
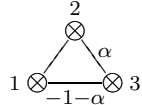
\xy
(-5,0)*{\otimes};(0,7)*{\otimes}**\dir{-};
(-5,0)*{\otimes};(5,0)*{\otimes}**\dir{-};
(0,7)*{\otimes};(5,0)*{\otimes}**\dir{-};
(-8,0)*{\scriptstyle 1};(0,10)*{\scriptstyle 2};(8,0)*{\scriptstyle 3};
(4,4)*{\scriptstyle \af};(0,-2)*{\scriptstyle -1-\af};
\endxy$$  				
$$[\alpha+1](1\shqqi (3\shqqi 2))+[\alpha](3 \shqqi (1\shqqi 2))=0.$$
\end{enumerate}
\end{lem}

\begin{proof}
This follows from Corollary~\ref{cor:psiUF} and the corresponding relations for $U_q$ given in Proposition~\ref{P:SerreRelations}.
These can also be deduced directly by tedious (but straightforward) computer calculation, which we omit.
\end{proof}

\begin{lem}\label{L:epsilonqder}
For each $i\in \I$, define the $\Qq$-linear operator $\ep_i':\F\longrightarrow\F$ by
$$\ep_i'(i_1,\ldots,i_d)=\dt_{i,i_d}(i_1,\ldots,i_{d-1}) \andeqn \ep_i'(\emptyset)=0.$$
Then, the endomorphisms $\ep_i'$ satisfy
$$\ep_i'(j)=\dt_{ij}
\quad
\text{ and } \quad
\ep_i'(x\shq y)=\ep_i'(x)\shq y +\pi^{p(x)p(y)}q^{-(\af_i,|x|)}x\shq\ep_i'(y).$$
\end{lem}

\begin{proof}
This is immediate from the definition and \eqref{E:Shuffle}.
\end{proof}

Given $\ui=(i_1,\ldots, i_d)\in \I^d$, define
\begin{align}\label{E:eui and epui}
e_\ui'=e_{i_1}'e_{i_2}'\cdots e_{i_d}'\andeqn \ep_\ui'=\ep_{i_1}'\ep_{i_2}'\cdots\ep_{i_d}'.
\end{align}
Define a $\Q(q)$-linear map
$$\Psi:\Uq\longrightarrow\F$$
by letting
\begin{equation}
\label{eq:psiu}
\Psi(u)=\sum_{\ui\in \W_\nu}e_\ui'(u)\ui,
\qquad \text{ for } u\in U_{q,\nu}.
\end{equation}
 (Here we have abused the same notation $\Psi$ as before, as it follows immediately by Proposition~\ref{T:UintoF}
below that they coincide.)
Since $e_\ui'(u)\in U_{q,0}=\Q(q)$, this map is well defined.
By Proposition~\ref{P:BilinearForm}, $\Psi$ is injective and $\Psi(e_i)=i$ for $i\in \I$.

\begin{prp}  \label{T:UintoF}
When the diagram for $\Uq$ is of type $A$-$D$ or the distinguished diagram in types $F$ and $G$,
the map $\Psi:\Uq\longrightarrow(\F,\shq)$ given by \eqref{eq:psiu} is an injective algebra
homomorphism (and hence coincides with the $\Psi$ given in Corollary~\ref{cor:psiUF}).
\end{prp}

\begin{proof}
We have just seen the injectivity of $\Psi$ above.
In the cases we are considering, we have by Lemma~\ref{L:Rk2Shuffle} and Theorem~\ref{T:Distinguished Diagrams}
that there exists an algebra homomorphism $\Upsilon:\Uq\longrightarrow(\F,\shq)$
such that $\Upsilon(e_i)=i$ for all $i\in \I$. Using Lemma \ref{L:epsilonqder}, this map satisfies
$\Upsilon\circ e_i'(u)=\ep_i'\circ\Upsilon(u)$. Let $u\in U_{q,\nu}$, and $\ui\in \W_\nu$.
Set $\gm_\ui(u)$ to be the coefficient of $\ui$ in $\Upsilon(u)$. Then,
$$\gm_\ui(u)=\ep_\ui'\circ\Upsilon(u)=\Upsilon\circ e_\ui'(u)=e_\ui'(u)\Upsilon(1)=e_\ui'(u),$$
where $\ep_\ui'=\ep_{i_1}'\cdots\ep_{i_d}'$. Hence $\Psi(u)=\Upsilon(u)$ and so $\Psi$ is an algebra homomorphism.

The $\Psi$ here and  the $\Psi$ given in Corollary~\ref{cor:psiUF} coincide
since both are algebra homomorphisms satisfying $\Psi(e_i)=i$ for $i\in \I$.
\end{proof}

Let $\Gm$ be the Dynkin diagram associated to $\UU$ and
let $\la \Gm\ra$ be the set of subdiagrams inducing relations associated to {\rm (AB)-(D)} in Lemma~\ref{L:Rk2Shuffle}.
Then using \eqref{P:qqi shuffle}, we may rewrite the relation corresponding to $\Gm'\in \la \Gm\ra$ in the form
\begin{equation}\label{eq:subdiagram coeffs}
\sum_{\ui=(i_1,\ldots,i_d)\in\W}\vartheta_{\Gm'}(\ui)(i_1\shq i_2\shq\cdots\shq i_d)=0,\quad \text{ for } \vartheta_{\Gm'}(\ui)\in \Q(q).
\end{equation}

\begin{exa}
Let $\UU$ be associated to the diagram
$$\hspace{.75in}\xy
(-10,0)*{\fullmoon};(0,0)*{\otimes}**\dir{-};(0,0)*{\otimes};(10,0)*{\fullmoon}**\dir{-};
(-10,-4)*{\scriptstyle i};(0,-4)*{\scriptstyle j};(10,-4)*{\scriptstyle k};
\endxy\quad(s_{ij}=-1\neq s_{jk}=1).$$
The only subdiagram causing a relation of the form {\rm (AB)-(D)} is the whole diagram (which corresponds to {\rm (AB)})
so $\la \Gm\ra=\set{\set{i,j,k}}$ (where we identify the subdiagram with its set of labels).
We have
\[\vartheta_{\set{i,j,k}}(\ui)=\begin{cases}
1&\text{ if } \ui\in\set{jkji, jijk, kjij, ijkj};\\
-q&\text{ if } \ui\in\set{jjik, jikj};\\
-q^{-1}&\text{ if } \ui\in\set{kijj;jkij};\\
0&\text{ otherwise.}
\end{cases}
\]
\end{exa}

\begin{prp}\label{T:In UU}
Let $\UU$ be associated to a diagram of type $A$-$D$, or to the distinguished diagram of type $F$ or $G$.
The element $x=\sum_{\uk\in\W}\gm_\uk(x)\uk\in\F$ belongs to $\UU$ if and only if
the following statements hold for all $\uh,\uh'\in\W$.
\begin{enumerate}
\item For all $i,j\in \I_\iso$ with $a_{ij}=0$,
\[\gm_{\uh\cdot ij\cdot\uh'}(x)+\gm_{\uh\cdot ji\cdot\uh'}(x)=0;\]
\item For all $i\in \I_\zero\cup \I_\niso$ and $j\in \I$ with $i\neq j$,
\[\sum_{r+s=1+|a_{ij}|}(-1)^r  \pi_i^{p(i,j;k)}
\begin{bmatrix}1+|a_{ij}|\\r\end{bmatrix}_i \gm_{\uh\cdot i^r\cdot j\cdot i^{s} \cdot \uh'}(x) =0;\]
\item For all $\Gm'\in \la \Gm\ra$, and with $\vartheta_{\Gm'}$ defined as in \eqref{eq:subdiagram coeffs},
\begin{equation*}
\sum_{\ui\in\W}\vartheta_{\Gm'}(\ui)\gm_{\uh\cdot\ui\cdot\uh'}(x)=0.
\end{equation*}
\end{enumerate}
\end{prp}

\begin{proof}
Let $\mathsf{V}$ be the subspace of $\F$ spanned by those elements that satisfy the statements (1)-(3). Let
$$x=\Psi(u)=\sum_{\uk\in\W,|\uk|=\nu}\gm(\uk)\uk\in\UU_\nu$$
be the image of some $u\in U_q$. Then, for $\uk=(k_1,\ldots,k_d)$, $\gm(\uk)=e'_\uk(u)=(e_{k_1}\cdots e_{k_d},u)$
by definition. Then by Corollary \ref{cor:derivationiso}, $x\in\mathsf{V}$.

Conversely,
note that by Lemma \ref{L:Rk2Shuffle} $x\in F$ satisfies (1)-(3) exactly when $x$ is orthogonal to a subspace of $F^*$
isomorphic to the kernel of the algebra surjection $F\rightarrow (F,\shq)$. Therefore,
we see that $\mathsf{V}_\nu=\F_\nu\cap\mathsf{V}$
has the same dimension as $U_{q,\nu}$. As $\Psi$ is injective, $\dim\UU_\nu=\dim\mathsf{V}_\nu$, and therefore $\UU=\mathsf{V}$.
\end{proof}

%%%%%%%%%%%%%%%%%%%%%
\subsection{Automorphisms of $\UU$}

For $\nu=\sum_{i\in \I}c_i\af_i\in Q^+$, we set
    \begin{align}\label{E:Nnu}
    N(\nu)=\frac12\Big((\nu,\nu)-\sum_{i\in \I}c_i(\af_i,\af_i)\Big),
    \qquad  P(\nu)=\frac12\Big(p(\nu)^2-\sum_{i\in\I}c_ip(\af_i)\Big),
    \end{align}
where here we interpret $p(\af_i)\in\{0,1\}$ and $p(\nu) =\sum_{i\in\I} c_ip(\af_i)$ as integers.
Below we realize certain automorphisms of $\UU$, whose counterparts for $U_q$ were given in Proposition~\ref{P:autoUq},
as restrictions of simple linear maps on $\F$ (compare \cite[Proposition 6]{lec}).

\begin{prp}\label{P:automorphisms}
\begin{enumerate}
\item Let $\tau:\F\longrightarrow\F$ be the $\Q(q)$-linear map defined by
    $$\tau(i_1,\ldots,i_d)=(i_d,\ldots,i_1).$$
    Then, $\tau(x\shq y)=\tau(y)\shq \tau(x)$ for all $x,y\in\F$. In particular, $\tau\Psi(u)=\Psi\tau(u)$ for all $u\in\Uq$, see \eqref{E:tau}.

\item Let $x\mapsto \overline{x}$ be the $\Q$-linear map $\F\longrightarrow\F$ such that
    $$\overline{q}=\begin{cases}\pi q^{-1}&\mbox{if $\Uq$ is of type }B(0,n+1),\\
        q^{-1}&\mbox{otherwise,}\end{cases}$$
    and
    $$\overline{(i_1,\ldots,i_d)}=\pi^{\sum_{s<t}p(i_s)p(i_t)}q^{-\sum_{s<t}(\af_{i_s},\af_{i_t})}(i_d,\ldots,i_1).$$
    Then, $\overline{x\shq y}=\overline{x}\shq\overline{y}$ and $\overline{\Psi(u)}=\Psi(\overline{u})$ for all $u\in\Uq$.

\item Let $\sm:\F\longrightarrow\F$ be the $\Q$-linear map defined by $\sm(x)=\overline{\tau(x)}$.
Then, $\sm\Psi(u)=\Psi\sm(u)$ for all $u\in\Uq$ and for $\nu=\sum_{i\in \I}c_i\af_i\in Q^+$ and $\ui\in \W_\nu$,
    $$\sm(\ui)=\pi^{P(\nu)}q^{-N(\nu)}\ui.$$
\end{enumerate}
\end{prp}

\begin{proof}
First note that
$$P(\alpha_{i_1}+\ldots+\alpha_{i_n})=\sum_{s<t}p(i_s)p(i_t)\andeqn N(\alpha_{i_1}+\ldots+\alpha_{i_n})=\sum_{s<t}(\af_{i_s},\af_{i_t}),$$
so (3) follows from (1) and (2).
We need only check (1) and (2) when $x,y\in \W$.
Note that (1) is clear from \eqref{E:Shuffle2}. To prove (2), proceed by induction.
Suppose (2) holds provided $\ell(x)+\ell(y)\leq n$ (the case $n=1$ being trivial).
Applying $\tau$ to the expression for $(\tau(y)j)\shq(\tau(x)i)$ given by \eqref{E:Shuffle}, we have
$$(ix)\shq(jy)=\pi^{p(i)(p(y)+p(j))}q^{-(\af_i,|y|+\af_j)}i(x\shq(jy))+j((ix)\shq y).$$
Therefore, assuming $\ell(xi)+\ell(yj)=n+1$, we have
\begin{align*}
\overline{(ix)\shq(jy)}&=\overline{\pi^{p(i)(p(y)+p(j))}q^{-(\af_i,|y|+\af_j)}i(x\shq jy) + j(ix\shq y)}\\
    &=\pi^{p(i)p(x)}q^{-(\af_i,|x|)}\overline{(x\shq jy)}i+\pi^{p(j)(p(x)+\af_i+p(y))}q^{-(\af_j,\af_i+|x|+|y|)}\overline{(ix\shq y)}j\\
    &=\pi^{p(i)p(x)+p(j)p(y)}q^{-(\af_i,|x|)-(\af_j,|y|)}(\overline{x}\shq \overline{y}j)i\\
    &+\pi^{p(i)p(x)+p(j)p(y)+p(j)(p(x)+p(i))}q^{-(\af_i,|x|) -(\af_j,\af_i+|x|+|y|)}(\overline{x}i\shq \overline{y})j\\
    &=\pi^{p(i)p(x)+p(j)p(y)}q^{-(\af_i,|x|)-(\af_j,|y|)}(\overline{x}i\shq \overline{y}j)\\
    &=(\overline{ix}\shq\overline{jy}).
\end{align*}
This proves (2).
\end{proof}

%%%%%%%%%%%%%%%%%%%%%%%%
\subsection{The Bialgebra Structure of $\UU$}
We now transport the bilinear form from $\Uq$ to $\UU$ via $\Psi$.

\begin{prp}\label{P:ShuffleCoproduct}
Let $\Dt:\F\longrightarrow\F\otimes\F$ be the map
$$\Dt(i_1,\ldots,i_d)=\sum_{0\leq k\leq d}(i_{k+1},\ldots,i_d)\otimes(i_1,\ldots,i_{k}).$$
Then, ${\Dt}(x\shq y)={\Dt}(x)\shq{\Dt}(y),$ where we define the shuffle product on $\F\otimes\F$ by
$$(w\otimes x)\shq (y\otimes z) = \pi^{p(x)p(y)}q^{-(|x|,|y|)}(w\shq y)\otimes(x\shq z).$$
In particular, we have ${\Dt}\Psi=(\Psi\otimes\Psi)\Dt$.
\end{prp}

\begin{proof}
For $x\in \W$, we write ${\Dt}(x)=\sum x_2\otimes x_1$. Then, for any $i\in \I$,
$${\Dt}(xi)={\Dt}(x)\cdot(i\otimes 1)+1\otimes xi=\sum x_2i\otimes x_1 + 1\otimes xi,$$
where we have used the associative multiplication $(w\otimes x)\cdot(y\otimes z)=wy\otimes xz$.

Let $x,y\in \W$ and $i,j\in \I$. Assume the proposition is proved provided $\ell(x)+\ell(y)\leq n$ (the case $n=1$ being trivial). Suppose that $\ell(xi)+\ell(yj)=n+1$. Write $\Delta(x)=\sum x_2\otimes x_1$ and $\Delta(y)=\sum y_2\otimes y_1$. We compute
\begin{align*}
\Dt(xi&\shq yj)\\
    =&\Dt((x\shq yj)i + \pi^{(p(x)+p(i))p(j)}q^{-(|x|+\af_i,\af_j)}(xi\shq y)j)\\
    =&\Dt(x\shq yj)\cdot(i\otimes 1)+1\otimes (x\shq yj)i+\pi^{(p(x)+p(i))p(j)}q^{-(|x|+\af_i,\af_j)}(\Dt(xi\shq y)\cdot(j\otimes 1)\\
    &+1\otimes(xi\shq y)j)
\end{align*}
By induction, this equals
\begin{align*}
(\Dt(x)&\shq\Dt(yj))\cdot(i\otimes 1)+\pi^{(p(x)+p(i))p(j)}q^{-(|x|+\af_i,\af_j)}(\Dt(xi)\shq\Dt(y))\cdot(j\otimes 1)\\
    &+1\otimes(xi\shq yj)\\
    =&\left[\left(\sum x_2\otimes x_1\right)\shq\left(\sum y_2j\otimes y_1+1\otimes yj\right)\right]\cdot(i\otimes 1)\\
    &+ \pi^{(p(x)+p(i))p(j)}q^{-(|x|+\af_i,\af_j)}\left[\left(\sum x_2i\otimes x_1+1\otimes xi\right)\shq\left(\sum y_2\otimes y_1\right)\right]\cdot(j\otimes 1)\\ &+1\otimes(xi\shq yj)\\
    =&\sum \pi^{p(x_1)(p(y_2)+p(j))}q^{-(|x_1|,|y_2|+\af_j)}(x_2\shq y_2j)i\otimes(x_1\shq y_1)+\sum x_2i\otimes(x_1\shq yj)\\
    &+\pi^{(p(x)+p(i))p(j)}q^{-(|x|+\af_i,\af_j)}\sum \pi^{p(x_1)p(y_2)}q^{-(|x_1|,|y_2|)}(x_2i\shq y_2)j\otimes(x_1\shq y_1)\\
    &+\pi^{(p(x)+p(i))p(j)}q^{-(|x|+\af_i,\af_j)}\sum \pi^{(p(x)+p(i))p(y_2))}q^{-(|x|+\af_i,|y_2|)}y_2j\otimes(xi\shq y_1)\\
    &+1\otimes(xi\shq yj)\\
    =&\sum \pi^{p(x_1)(p(y_2)+p(j))}q^{-(|x_1|,|y_2|+\af_j)}((x_2\shq y_2j)i\\
    &+\pi^{(p(x_2)+p(i))p(j)}q^{-(|x_2|+\af_i,\af_j)}(x_2i\shq y_2)j)\otimes(x_1\shq y_1)+\sum x_2i\otimes(x_1\shq y_j)\\
    &+\sum \pi^{(p(x)+p(i))(p(y_2)+p(j))}q^{-(|x|+\af_i,y_2+\af_j)}y_2j\otimes(xi\shq y_1)+1\otimes(xi\shq yj)\\
    =&\Dt(xi)\shq\Dt(yj).
\end{align*}
This completes the proof.
\end{proof}

\begin{rmk}
The formulas in this paper differ slightly from those appearing in \cite{klram2}, where multiplication and comultiplication correspond to induction and restriction at the categorified level. If we regard the shuffle product in this paper as a map
$$m_\shq:\UU\otimes\UU\longrightarrow\UU,$$
then the precise relationship with induction and restriction in a categorification of $\UU$ will be
$$[\ind]=\tau\circ m_\shq\circ(\tau\otimes\tau)\andeqn[\res]=(\tau\otimes\tau)\circ\Dt\circ\tau.$$
\end{rmk}

As a consequence of Proposition \ref{P:ShuffleCoproduct}, we obtain the following counterpart of
Proposition~\ref{P:ShuffleCoproduct} via
the algebra isomorphism $\Psi: U_q\rightarrow \UU$.

\begin{prp}\label{P:bilinearformonF} There exists a symmetric nondegenerate bilinear form
$$(\cdot,\cdot):\UU\otimes\UU\longrightarrow\Q(q)$$
satisfying %$(1,1)=1$, and
\begin{enumerate}
\item $(1,1)=1$;
\item $(i,j)=\dt_{ij}$, for $i,j \in \I$;
\item $(x,y\shq z)=(\Dt(x), y\otimes z)$, for $x,y, z \in \UU$.
\end{enumerate}
\end{prp}

%%%%%%%%%%%%%%%%%
%%%%%%%%%%%%%%%%%
\section{Combinatorics of Words}\label{S: Combinatorics of Words}

In this section, we will develop word combinatorics for the $q$-shuffle superalgebra following closely \cite[Section 3]{lec}
(which was in turn built on \cite{lo, lr}).

%%%%%%%%%%%%%%%%%
\subsection{Dominant Words and Monomial Bases}

We now fix a total ordering, $\leq$, on $\I$. Let $\Wr=(\W,\leq)$ be the ordered set with respect to the corresponding lexicographic order:
$$\ui=(i_1,\ldots,i_d)<(j_1,\ldots,j_k)=\uj$$
if there exists an $r$ such that $i_r<j_r$ and $i_s=j_s$ for $s<r$, or if $d<k$ and $i_s=j_s$ for $s=1,\ldots d$ (i.e., $\ui$ is a proper left factor of $\uj$).

%Let $\Wl=(\W,\leq)$ be the ordered set with respect to the co-lexicographic ordering on $\W$:
%$$\ui=(i_1,\ldots,i_d)<(j_1,\ldots,j_k)=\uj$$
%if there exists an $r$ such that $i_{d-r}<j_{k-r}$ and $i_{d-s}=j_{k-s}$ for $s<r$, or if $d>k$ and $i_{d-s}=j_{k-s}$ for $s=1,\ldots k$ (that is, $\ui$ is a right factor of $\uj$).

%The relation between $\Wr$ and $\Wl$ is given by the opposite ordering, $\preq$, on $\I$. That is, $i\preq j$ if and only if $j\leq i$. Of course, there are versions of $\Wr$ and $\Wl$ for this ordering, and we will use the notation $(\Wr,\preq)$, $(\Wl,\preq)$, $(\Wr,\leq)$, and $(\Wl,\leq)$ when we need to emphasize which ordering we are using.

%\begin{lem}\label{L:reverse order}  The map $\tau$ defines an order reversing bijection
%$$(\Wr,\leq)\longrightarrow(\Wl,\preceq).$$
%That is, for $\ui,\uj\in\W$, $\ui\leq\uj$  in $(\Wr,\leq)$ if and only if $\tau(\uj)\preq\tau(\ui)$ in $(\Wl,\leq)$.
%\end{lem}

%For the remainder of the section, we will work entirely with the ordering $\leq$, and will simply write $\Wr$ and $\Wl$ for the corresponding ordered sets (the ordering on $\I$ is arbitrary at this point, so we could just as well redefine $\leq$ in the opposite way). Using Lemma \ref{L:reverse order}, we can translate all results for $\Wr$ into results for $\Wl$.

For $x\in\F$, we set $\max(x)=\ui$ if  $\kp_\ui\neq 0$ in the expansion $x=\sum_{\uj \in \Wr}\kappa_\uj \uj$ (where $\kappa_\uj\in\Q(q)$)
and $\kp_\uj =0$ unless $\ui\geq\uj$.
A word $\ui\in \Wr$ is called \textbf{dominant} (also called {\em good} in \cite{lec})
if $\ui=\max(u)$ for some $u\in\UU$, and let $\Wr^+$ denote the subset of dominant words of $\Wr$.

The following proposition proves that the set $\Wr^+$ labels bases of $\Uq$ and $\UU$. The proof proceeds
exactly as in \cite[Proposition 12]{lec}.

\begin{prp}\label{P:monomial basis}\begin{enumerate}
\item There exists a unique basis of  homogeneous vectors $\{m_\uj \mid \uj \in \Wr^+\}$ in $\UU$ such that
$$
\ep_\ui'(m_\uj)=\dt_{\ui\uj} \qquad \text{ if } \; |\ui| =|\uj|,
$$
where $\ep_\ui'$ is defined in Lemma \ref{L:epsilonqder} and \eqref{E:eui and epui}.
\item The set $\{ e_\ui=e_{i_1}\cdots e_{i_d}\mid \ui=(i_1,\ldots,i_d)\in \Wr^+\}$ is a basis (called monomial basis) of $\Uq$.
\end{enumerate}
\end{prp}

For $\ui=(i_1,\ldots,i_d)\in\Wr$, define $\ep_\ui=i_1\shq\cdots\shq i_d=\Psi(e_\ui)$. Define the \textbf{monomial basis} for $\UU$ to be
\begin{align}\label{E:monomialbasis}
\{\ep_\ui\mid\ui\in\Wr^+\}.
\end{align}

The next lemma generalizes \cite{lr} (cf.  \cite{lec}).

\begin{lem}\label{L:factor} Every factor of a dominant word is dominant.
\end{lem}

\begin{proof} This follows from the fact that $\UU$ is stable under the action of $\ep_i'$ and
$\ep_i''=\tau\ep_i'\tau$, $i\in \I$. See \cite[Lemma 13]{lec}.
\end{proof}

%\subsubsection{} Given Given $x\in\F$, write $x=\sum_{\ui\in \Wl}\kappa_\ui \ui$ for some $\kappa_\ui\in\Q(q)$ and set
%$\min(x)=\ui$ if $\kp_\ui\neq 0$ and $\ui\preq\uj$ for all $\uj\in\Q$ such that $\kp_\uj\neq 0$.
%Call a word \textbf{co-standard} if $\ui=\min(u)$ for some $u\in\UU$, and let $\Wl^-$ denote the set of co-standard words.

\subsection{Lyndon Words}

\subsubsection{} A word $\ui=(i_1,\ldots,i_d)\in \Wr$ is called \textbf{Lyndon} if it is smaller than any of its proper right factors:
\begin{align}\label{E:Lyndon dfn}
\ui<(i_r,\ldots,i_d),\quad \text{ for } 1<r\leq d.
\end{align}
Let $\Lr$ denote the set of Lyndon words in $\Wr$.

Let $\ui\in\L$. Call the decomposition $\ui=\ui_1\ui_2$ the \textbf{co-standard factorization} of $\ui$
if $\ui_1,\ui_2\neq\emptyset$, $\ui_1\in\L$, and the length of $\ui_1$ is maximal among all such decompositions.
In this case, it is known that $\ui_2\in\L$ as well, see \cite[Chapter 5]{lo}.
Call the decomposition $\ui=\ui_1\ui_2$ the \textbf{standard factorization} if $\ui_1,\ui_2\neq\emptyset$, $\ui_2\in\L$,
and the length of $\ui_2$ is maximal among all such decompositions. As above, we have $\ui_1\in\L$ as well.

We will frequently use the following lemma.

\begin{lem}\label{L:costd factorization}\cite[Lemma 14]{lec} Let $\ui\in\L$, and let $\ui=\ui_1\ui_2$ be its co-standard factorization. Then, $\ui_2=\ui_1^r\ui_1'i$ where $r\geq 0$, $\ui_1'$ is a (possibly empty) proper left factor of $\ui_1$, and $\ui_1'i>\ui_1$.
\end{lem}

We also have the following converse to this lemma.

\begin{lem}\label{L:costd converse} If $\ui\in\L$ and $\uj=\ui^r\ui'i$ where $r\geq 1$, $\ui'$ is a (possibly empty) proper left factor of $\ui$, and $\ui<\ui'i$, then $\uj\in\L$.
\end{lem}

\begin{proof} It is enough to prove the statement when $r=1$, the general case being similar. To this end, assume $\ui=(i_1,\ldots,i_d)\in\L$ and $\uj=\ui\ui'i$ satisfies the conditions of the lemma. Then $\uj=(i_1,\ldots,i_d,i_1,\ldots,i_k,i)$. If $\uj''$ is a right factor of $\uj$ then either
\begin{enumerate}
\item $\uj''=(i_r,\ldots,i_d,i_1,\ldots,i_k,i)$, or
\item $\uj''=(i_r,\ldots,i_k,i)$.
\end{enumerate}
In case (1), we have $\ui=(i_1,\ldots,i_d)<(i_r,\ldots,i_d)$ since $\ui\in\L$. As $\ell(i_r,\ldots,i_d)<\ell(\ui)$ we my conclude that $\uj<\uj''$. For case (2), we have $\ui<(i_r,\ldots,i_d)<(i_r,\ldots,i_k,i)$, so $\uj<\uj''$ as well. This completes the proof.
\end{proof}

Let $\Lr^+$ be the set of dominant Lyndon words in $\Wr$. Note that
$$
\Lr^+ =\Lr \cap \Wr^+ \subset \Wr^+ \subset \Wr.
$$
It is well known \cite{lo} that every word $\ui\in \W$ has a \textbf{canonical factorization} as a product of non-increasing Lyndon words:
\begin{align}\label{E:word factorization}
\ui=\ui_1\cdots\ui_d,\;\;\;\ui_1,\ldots\ui_d\in\L,\;\ui_1\geq\cdots\geq\ui_d.
\end{align}

\begin{lem}\label{L:max shuffle} Let $\ui\in\Lr$ and $\uj\in \Wr$.
Assume that $\ui\geq \uj$, and further assume $\ui\neq \uj$ if $|\ui|\in Q^+$ is  isotropic odd.
Then $\max(\ui\shq \uj)=\ui\uj$.
\end{lem}

\begin{proof}
We will prove a slightly stronger statement. Namely, we will prove that $\max(\ui\shq\uj)\leq\ui\uj$ and
\begin{enumerate}
\item if $\ui>\uj$, then the coefficient of $\ui\uj$ in $\ui\shq\uj$ is $\pi^{p(\ui)p(\uj)}q^{-(|\ui|,|\uj|)}$ and,
\item if $\ui=\uj$, then the coefficient of $\ui\ui$ in $\ui\shq\ui$ is $1+\pi^{p(\ui)}q^{-(|\ui|,|\ui|)}$.
\end{enumerate}

Let $\ui=(i_1,\ldots,i_d)$ and $\uj=(j_1,\ldots,j_k)$. We prove this statement by a double induction on $\ell(\ui)=d$ and $\ell(\uj)=k$. To this end, suppose $\ell(\ui)=1$, i.e. $\ui=i_1=i\in\I$. If $i>\uj$, then $i>j_1$, so clearly $\max(i\shq\uj)=i\uj$ and $i\uj$ occurs with the coefficient given in (1). If $\uj=\ui$, then $\uj=j_1=i$ and
$$i\shq i=(1+\pi^{p(i)}q^{-(\af_{i},\af_{i})})(ii).$$
Hence (2) follows.

Now, suppose that $\ell(\uj)=1$, so $\uj=j_1=j\in\I$. The case $\ui=\uj$ is treated above, so assume that $\ui>j$. Then, $j<i_1$. Assume $$\uk=(k_1,\ldots,k_{d+1})=(i_1,\ldots,i_{r-1},j,i_r,\ldots,i_{d})$$
is any word occurring as a nontrivial shuffle in $\ui\shq j$. Then, $k_r=j<i_1\leq i_r$, so $\uk<\ui j$ and (1) holds.

We now proceed to the inductive step.

\bigskip

\noindent\underline{\textbf{Case 1:}} $\ui>\uj$.

Let $\ui=\ui_1\ui_2$ be the co-standard factorization of $\ui$ and recall that $\ui_2$ is of the form $\ui_2=\ui_1^r\ui_1'i$, see Lemma \ref{L:costd factorization}. Then, if $\uk$ occurs as a nontrivial shuffle in $\ui\shq\uj$, there exists a factorization $\uj=\uj_1\uj_2$ such that $\uk$ occurs in $(\ui_1\shq\uj_1)(\ui_2\shq\uj_2)$.

If $\ui_1\geq\uj_1$, then by induction on $\ell(\ui)$, $\max(\ui_1\shq\uj_1)\leq \ui_1\uj_1$. It now follows that
$$\uk\leq \ui_1\uj_1\max(\ui_2\shq\uj_2).$$
Since $\uj<\ui<\ui_2$, induction on $\ell(\ui)$ implies that $\max(\ui_2\shq\uj)=\ui_2\uj$ and any nontrivial shuffle is strictly smaller.
Now, since any word occurring in $\uj_1(\ui_2\shq\uj_2)$ is a proper shuffle in $\ui_2\shq(\uj_1\uj_2)=\ui_2\shq\uj$, we have
$$\uk\leq\ui_1\uj_1\max(\ui_2\shq\uj_2)<\ui_1\max(\ui_2\shq\uj)=\ui\uj.$$

Assume $\ui_1<\uj_1$. Since $\ui>\uj$, we must have $\uj_1=\ui_1\uj_1'$ with $\uj_1'\uj_2<\ui_2$. Note that any shuffle occurring in $\ui_1\shq\uj_1$ must occur in $(\ui_{11}\shq\ui_1)(\ui_{12}\shq\uj_1')$ for some factorization $\ui_1=\ui_{11}\ui_{12}$. By induction, $\max(\ui_{11}\shq\ui_1)\leq\ui_1\ui_{11}$, so
$$\uk\leq \ui_1\ui_{11}\max(\ui_{12}\shq\uj_1')\max(\ui_2\shq\uj_2).$$
Any word occurring in $\ui_{11}(\ui_{12}\shq\uj_1')$ must also occur in $\ui_1\shq\uj_1'$, and any word occurring in
$\ui_1(\ui_1\shq\uj_1')(\ui_2\shq\uj_2)$ also occurs in $\ui_1(\ui\shq (\uj_1'\uj_2)).$

Set $\uh=\uj_1'\uj_2$. If $\uh<\ui$, then induction on $\ell(\uj)$ implies that $\max(\ui\shq\uh)=\ui\uh$ and any proper shuffle is strictly smaller. Hence,
$$\uk\leq \ui_1\ui_{11}\max(\ui_{12}\shq\uj_1')\max(\ui_2\shq\uj_2)<\ui_1\max(\ui\shq\uh)= \ui_1\ui_1^{r+1}\ui_1'i\uh<\ui<\ui\uj.$$
We may, therefore, assume that $\uh\geq\ui$.

Recall that $\uh<\ui_2=\ui_1^r\ui_1'i$. If $\uh\leq \ui_1^r\ui_1'$, then $\uh<\ui$ since $\ui_1^r\ui_1'$ is a left factor of $\ui$. This contradicts our assumption, leaving us to consider the case where $\uh>\ui_1^r\ui_1'$.

Since $\uh<\ui_2$, it follows that $\uh=\ui_1^r\ui_1'\uh'$, where $\uh'<i$. Suppose for the moment that $\uh'=j\in\I$, i.e. $\uh=\ui_1^r\ui_1'j$, $j<i$. Since $\uh>\ui$, $\ui_1'j>\ui_1$ and, therefore, $\uh\in\L$ by Lemma~\ref{L:costd converse}. Since $\ell(\uh)<\ell(\ui)$ we may apply induction to conclude that $\max(\ui\uh)\leq\uh\ui$. Hence,
$$\uk\leq\ui_1\ui_1^r\ui_1'j\ui<\ui=\ui_1^{r+1}\ui_1'i<\ui\uj.$$
More generally, when $\uh'=j\uh''$ is not a letter, any word in $\ui_1(\ui\shq\uh)$ can be obtained by first shuffling $\ui_1^r\ui_1'j$ into $\ui$ to obtain a word $\ui_1\ul=\ui_1(\ul_1j\ul_2)$, and then shuffling $\uh''$ into $\ul_2$. Since we already have proved that the maximum of the $\ui_1\ul_1j\ul_2$ appearing this way is $\ui_1^{r+1}\ui_1'j\ui$, $\ui_1^{r+1}\ui_1'j<\ui$ and $\ell(\ui_1^{r+1}\ui_1'j)=\ell(\ui)$, the same holds in general. This finishes Case 1 and proves (1).

\bigskip

\noindent\underline{\textbf{Case 2:}} $\ui=\uj$.

This case is almost identical to Case 1 except in the last step where now $\uh=\ui_1^r\ui_1'i$. From this we see that there are exactly two ways in which $\ui\ui$ occurs in $\ui\shq\ui$ and (2) follows.
\end{proof}

The next statement follows immediately from the proof above.

\begin{cor}\label{C:not max} Assume that $\ui\in\L$ and $|\ui|=\nu$ is isotropic odd, then $\max(\ui\shq\ui)<\ui\ui$.
\end{cor}

The next proposition  now follows as in \cite[Proposition 16]{lec}.

\begin{prp}  \label{P:L16}
Let $\ui\in\Lr^+$ and $\uj\in\Wr^+$ with $\ui\geq\uj$, and further assume $\ui\neq \uj$ if $|\ui|\in Q^+$ is  isotropic odd. Then, $\ui\uj\in\Wr^+$.
\end{prp}

\begin{thm} \label{thm:canonical factorization of dominant words}
The map $\ui\mapsto|\ui|$ defines a bijection from $\L^+$ to $\Phi^+$. Moreover, $\ui\in\W^+$ if and only if its canonical factorization is of the form
$\ui=\ui_1\cdots \ui_r$,
where $\ui_1,\ldots,\ui_r\in\Lr^+$ satisfy $\ui_1\geq \ldots \geq \ui_r$
and $\ui_s$ appears only once whenever $|\ui_s|$ is isotropic odd.
\end{thm}

\begin{proof}
We prove both statements simultaneously by induction. Let $\L^+_n=\{\ui\in\L^+\mid\ell(\ui)=n\}$, $\Phi^+_n=\{\bt\in\Phi^+\mid\mathrm{ht}(\bt)=n\}$ and let $\W^\oplus$ be the set of words in $\W$ satisfying the conditions of the theorem. By Proposition~\ref{P:L16}, $\W^\oplus\subset\W^+$.

Assume that for $r<n$ there is a bijection $\L^+_r\longrightarrow\Phi^+_r$, and $\W^\oplus_\nu=\W^+_\nu$ whenever $\mathrm{ht}(\nu)<n$. The base case is the bijection $$\L^+_1=\I\leftrightarrow\Pi=\Phi^+_1.$$

We now proceed to the inductive step. Let $\preq$ be an arbitrary total ordering on $\Phi^+$. For $\nu\in Q^+$, let $d(\nu)=\dim U_{q,\nu}$, and define
$$d'(\nu)=|\{(\bt_1,\ldots,\bt_d)\in(\Phi^+)^d\mid d\geq 2,\,\bt_1\preq\cdots\preq\bt_d,\, \bt_1+\ldots+\bt_d=\nu\}|.$$
Then, by the PBW theorem for $U_q$ (cf. \cite{ya}), $d(\nu)=1+d'(\nu)$ if $\nu\in\Phi^+$, and $d(\nu)=d'(\nu)$ otherwise.

Assume that $\ui\in\Lr^+_n$, $|\ui|=\nu\in Q^+$. By induction, $|\W^\oplus_\nu\backslash\{\ui\}|\geq d'(\nu)$. Since $\W^\oplus_\nu\subset\W^+_\nu$, and $|\W^+_\nu|=d(\nu)$, we have
$$d(\nu)=|\W^+_\nu|\geq|\W^\oplus_\nu|\geq 1+d'(\nu)\geq d(\nu).$$
This forces $d(\nu)=1+d'(\nu)$ and, therefore, $\nu\in\Phi^+_n$. Moreover, it follows that $\ui\in\W^+_\nu$ is the unique Lyndon word of its degree. Hence, the map $\Lr^+_n\longrightarrow\Phi^+_n$ is injective and $\W^\oplus_\nu=\W^+_\nu$ whenever $\mathrm{ht}(\nu)=n$ and $\L^+_\nu\neq\emptyset$.

We now prove this map is surjective. To this end, let $\bt\in\Phi^+_n$. By induction $|\W^\oplus_\bt|\geq d'(\bt)$ \
and $|\W^\oplus_\bt|>d'(\bt)$ if and only if $\L^+_\bt\neq\emptyset$ (in which case there is a unique $\ui(\bt)\in\L^+_\bt$).
Suppose that the map is not surjective; that is,
$|\W^\oplus_\bt|=d'(\bt)$. Then, there exists $\uj\in\W^+_\bt\backslash\W^\oplus_\bt$
with $\uj=\uj_1\cdots\uj_r$ with $\ui=\uj_s=\uj_{s+1}$ odd isotropic for some $s$.
If $\uj\neq\ui\ui$, then $\ui\ui\in\W^+$ by Lemma~\ref{L:factor}. Since $\ell(\ui\ui)<\ell(\uj)$, $\W^\oplus_{2|\ui|}=\W^+_{2|\ui|}$
and so $\ui\ui\in W^\oplus$, contradicting the definition of $W^\oplus$.
But, the only alternative is $\uj=\ui\ui$, which implies both $2|\ui|=\bt$ and $|\ui|$ are in $\Phi^+$,
contradicting the fact that $\Phi^+$ is reduced. It now follows that
$$|\W^\oplus_\nu|=d(\nu)=|\W^+_\nu|$$
for all $\nu\in Q^+$, which completes the proof of both statements of the theorem.
\end{proof}

\subsection{Bracketing and Triangularity}\label{SS:BrackT}
%For the purposes of the following definitions, let $\L$ denote either $\Lr$ or $\Ll$. Let $\ui\in\L\backslash \I$.

For homogeneous $x,y\in\F$, define
\begin{align}\label{E:qbracket}
[x,y]_q=xy-\pi^{p(x)p(y)}q^{(|x|,|y|)}yx.
\end{align}
When $\ui\in\Lr^+$, we define $[\ui]^+\in\F$ inductively by $[\ui]^+=i$ if $\ui=i\in \I$ and, otherwise, $[\ui]^+=[\ui_1,\ui_2]_q$,
where $\ui=\ui_1\ui_2$ is the co-standard factorization of $\ui$.

The next two propositions are proved exactly as in \cite[Propositions 19 and 20]{lec}.

\begin{prp} For $\ui\in\Lr^+$, $[\ui]^+=\ui+x$ where $x$ is a linear combination of words $\uj\in\Wr^+$ satisfying $\uj>\ui$.
\end{prp}

Now, for $\ui\in\Wr$, let $\ui=\ui_1\cdots\ui_r$, where $\ui_1,\ldots,\ui_r\in\Lr^+$ and $\ui_1\geq \ldots \geq \ui_r$,
be its canonical factorization. Define
$$[\ui]^+=[\ui_1]^+\cdots[\ui_r]^+.$$

\begin{prp} The set $\{[\ui]^+\mid\ui\in\Wr\}$ is a basis for $\F$.
\end{prp}

Now, let $\Xi:(\F,\cdot)\longrightarrow(\F,\shq)$ be the algebra homomorphism defined by $\Xi(i_1,\ldots,i_d)=i_1\shq\cdots\shq i_d$.
Obviously, we have $\Xi(\F)=\UU$. The next lemma generalizes \cite[Lemma 21]{lec} with an identical proof.

\begin{lem}\label{L:Standard Triangularity}
A word $\ui\in\Wr$ is dominant if and only if it cannot be expressed modulo $\ker\Xi$ as a linear combination of words $\uj>\ui$.
\end{lem}

%\subsubsection{} When $\ui\in\Ll$, we define $[\ui]^-\in\F$ inductively by $[\ui]^-=i$ if $\ui=i\in \I$ and, otherwise, $[\ui]^-=[\ui_1,\ui_2]_q$, where $\ui=\ui_1\ui_2$ is the standard factorization of $\ui$. Note that $\tau([\ui]^-)=[\tau(\ui)]^+$.

%Let $\ui\in\Wl$, and $\ui=\ui_1'\cdots\ui_s'$, $\ui_1',\ldots,\ui_s'\in{}^\tau\Ll$, $\ui_t\succeq\ui_{t+1}$ be its canonical factorization as a %nonincreasing product of co-Lyndon words. Define
%$$[\ui]^-=[\ui_1']^-\cdots[\ui_s']^-.$$

%%%%%%%%%%%%%%%%%%
\subsection{Lyndon Bases}
For $\ui\in\Wr^+$ we define $\r_\ui=\Xi([\ui]^+)$.

\begin{prp}\label{P:rofL} Let $\ui\in\Lr^+$ and $\ui=\ui_1\ui_2$ be the co-standard factorization of $\ui$. Then,
$$\r_\ui = \r_{\ui_1}\shqqi \r_{\ui_2}.$$
\end{prp}

\begin{proof}
Observe that $\ui_1,\ui_2\in\Lr^+$ by Lemma \ref{L:factor} and $\S$\ref{SS:BrackT}. Therefore, we compute that
\begin{align*}
\r_\ui&=\Xi([[\ui_1]^+,[\ui_2]^+]_q)\\
    &=\Xi([\ui_1]^+)\shq\Xi([\ui_2]^+)-\pi^{p(\ui_1)p(\ui_2)}q^{-(|\ui_1|,|\ui_2|)}\Xi([\ui_2]^+)\shq\Xi([\ui_1]^+)\\
    &=\r_{\ui_1}\shq \r_{\ui_2}-\pi^{p(\ui_1)p(\ui_2)}q^{-(|\ui_1|,|\ui_2|)}\r_{\ui_2}\shq \r_{\ui_1}.
\end{align*}
The proposition now follows by applying Proposition \ref{P:ShuffleProductProperty}.
\end{proof}

Recall the monomial basis from \eqref{E:monomialbasis}. The next proposition generalizes \cite[Proposition~ 22]{lec}.

\begin{prp}
   \label{P:Lyndon triangularity}
For $\ui\in\Wr^+$, we have
$$\r_\ui=\ep_\ui+\sum_{\uj\in\Wr^+,\,\uj>\ui}\chi_{\ui\uj}\,\ep_\uj,$$
for some $\chi_{\ui\uj}\in\Q(q)$. In particular, the set
$\{\r_\ui\mid\ui\in\Wr^+\}$
is a basis for $\UU$.
\end{prp}

\begin{proof} By Lemma \ref{L:Standard Triangularity} we have
$$[\ui]^+\in\ui+\sum_{\uj\in\Wr^+,\,\uj>\ui}\chi_{\ui\uj}\,\uj\,+\ker\Xi,$$
for some $\chi_{\ui\uj}\in\Z[q,q^{-1}]$. Therefore, the first statement follows by applying $\Xi$.
The second statement follows since the transition matrix from the monomial basis is triangular.
\end{proof}

Call the basis $\{\r_\ui\mid \ui\in\Wr^+\}$ the \textbf{Lyndon basis} for $\UU$.
The following theorem is an analogue of \cite[Theorem 23]{lec} and is immediate from
Theorem \ref{thm:canonical factorization of dominant words} and the definitions.

\begin{prp}
The Lyndon basis has the form
$$\left\{\r_{\ui_1}\shq\cdots\shq \r_{\ui_k}\;\bigg|\; \begin{matrix}\ui_1,\ldots,\ui_k\in\Lr^+,
                \ui_1\geq\cdots\geq \ui_k\mbox{ and}\\
                \ui_{s-1}>\ui_s>\ui_{s+1}\mbox{ if }|\ui_s|\in\Phi^+_\one\mbox{ is isotropic}
                \end{matrix}\right\}.
$$
\end{prp}

%\subsubsection{} For $\ui\in\Wl^-$, define $\r_\ui^-=\Xi([\ui]^-)$. We have
%$$\{\r_\ui^-\mid \ui\in\Wl^-\}$$
%is another basis for $\UU$, the \textbf{co-Lyndon Basis}.

\subsection{Computing Dominant Lyndon Words}
Given $\ui\in\Lr^+$, write $\ui=\ui^+(\bt)$ if $\bt\in\Phi^+$ is the image of $\ui$
under the bijection $\Lr^+\longrightarrow\Phi^+$ (i.e. $|\ui|=\bt$).

\begin{prp} Let $\bt_1,\bt_2\in\Phi^+$ be such that $\bt_1+\bt_2=\bt\in\Phi^+$.
If $\ui^+(\bt_1)<\ui^+(\bt_2)$, then $\ui^+(\bt_1)\ui^+(\bt_2)\leq\ui^+(\bt)$.
\end{prp}

\begin{proof}
This proof essentially proceeds as in \cite[Proposition 24]{lec}.
Indeed, write $\ui_1=\ui^+(\bt_1)$, $\ui_2=\ui^+(\bt_2)$ and $\ui=\ui^+(\bt)$. We have that
$\r_{\ui_1}\shq \r_{\ui_2}=\sum_{\uj\in\Wr^+,\,\uj\geq\ui_1\ui_2}z_\uj \; \r_\uj,$
where $z_\uj\in\Z[q,q^{-1}]$. It is therefore necessary to show that $z_{\ui}\neq 0$.

For this, we appeal to \cite[Theorem 10.5.8]{ya} which provides a specialization
$x\mapsto \underline{x}$ from $\Uq$ to $U(\n)$. Write $s_{\uj}=\Psi^{-1}(\r_\uj)$ for $\uj\in\Wr^+$.
Then $\underline{s_\uj}\in\n$ being an iterated bracket of Chevalley generators. We have that
$\underline{s_{\ui}}=[\underline{s_{\ui_1}},\underline{s_{\ui_2}}]$ belongs to the $\bt$-weight space of $\n$,
which is 1-dimensional and spanned by $\underline{s_{\ui}}$. Therefore,
$$\underline{s_{\ui_1}}\, \underline{s_{\ui_2}}=\pi^{p(\ui_1)p(\ui_2)}\underline{s_{\ui_2}}\, \underline{s_{\ui_1}}+\ld \underline{s_{\ui}}\in U(\n)$$
for some nonzero $\ld\in\Z$. It now follows that $z_{\ui}\neq 0$ and hence $\ui\geq\ui_1\ui_2$.
\end{proof}

This yields an inductive method for computing dominant Lyndon words as described in \cite[$\S4.3$]{lec}. We recall it here. Let
$$C(\bt)=\{(\bt_1,\bt_2)\in\Phi^+\times\Phi^+\mid\bt_1+\bt_2=\bt\mbox{ and }\ui^+(\bt_1)<\ui^+(\bt_2)\}.$$
Then, the next proposition is a super-analogue of \cite[Proposition 25]{lec}.

\begin{prp}\label{P:ComputingLyndonWords} For $\bt\in\Phi^+$,
$$\ui^+(\bt)=\max\{\ui^+(\bt_1)\ui^+(\bt_2)\mid(\bt_1,\bt_2)\in C(\bt)\}.$$
Moreover, if $(\bt_1,\bt_2)\in C(\bt)$ achieves the maximum,
then $\ui^+(\bt)=\ui^+(\bt_1)\ui^+(\bt_2)$ is the co-standard factorization of $\ui^+(\bt)$.
\end{prp}

\begin{cor}\label{C:min std word}\cite[Corollary 27]{lec} For $\bt\in\Phi^+$, $\ui^+(\bt)$ is the smallest dominant word of its degree.
\end{cor}

%\subsubsection{}
%Given $\ui\in\Ll^-$, write $\ui=\ui^-(\bt)$ if $\bt\in\Phi^+$ is the image of $\ui$ under the bijection $\Ll^-\longrightarrow\Phi^+$ (i.e. $|\ui|=\bt$). %Then,
%$$\ui^-(\bt)=\min\{\ui^-(\bt_1)\ui^-(\bt_2)\mid(\bt_1,\bt_2)\in C(\bt)\},$$
%and if $(\bt_1,\bt_2)\in C(\bt)$ achieves the minimum, then $\ui^-(\bt)=\ui^-(\bt_1)\ui^-(\bt_2)$ is the standard factorization of $\ui^-(\bt)$.

\subsection{Further Properties of Lyndon Bases}

\begin{lem}\label{L:Lyndon left factor} Let $\ui=(i_1,\ldots,i_d)\in\Lr^+$.
Then, $i_1$ is a left factor of every word appearing in the expansion of $\r_{\ui}$.
\end{lem}

\begin{proof} Proceed by induction on the length of $\ui$, the case $\ui=i_1\in\I$ being trivial.

For the inductive step, let $\ui=\ui_1\ui_2$ be the costandard factorization of $\ui$. By \cite[Lemma~ 14]{lec},
$\ui_2=\ui_1^r\ui_1'i$ where $r\geq 0$, $\ui_1'$ is a (possibly empty) left factor of $\ui_1$
and $i\in\I$ is such that $\ui_1'i>\ui_1$. By Proposition \ref{P:rofL},
$$\r_\ui=\r_{\ui_1}\shqqi\r_{\ui_2}.$$
By induction, $i_1$ is a left factor of every word in the expansion of $\r_{\ui_1}$. If $\ui_2=\ui_1^r\ui_1'i$ with either $r>0$ or
$\ui_1'\neq\emptyset$, then $i_1$ is a left factor of every word in the expansion of $\r_{\ui_2}$ and therefore
the same holds for $\r_\ui$. Otherwise, $\ui_2=i$, and, if $\uk=(i_1,k_2,\ldots,k_{d-1})$ is a word appearing in the expansion of $\r_{\ui_1}$ then
$$\uk\shq i=\pi^{p(i_1)p(i)}q^{-(\af_{i_1},\af_i)}i_1((k_2,\ldots,k_{d-1})\shq i)+i\uk.$$
In particular, $i_1$ is a left factor of every word appearing in $\uk\shqqi i$. This proves the lemma.
\end{proof}

%\begin{dfn}\label{D:STAR}
%The ordering $(\I,\leq)$ is called {\bf regular} if either
%\begin{enumerate}
%\item $\UU$ is of type $A$-$D$,
%\item $\UU$ corresponds to the distinguished diagram of type $F(3|1)$ in Table~1 and $3\in\I$ is not minimal, or
%\item $\UU$ corresponds to the distinguished diagram of type $G(3)$ in Table~1 and $2\in\I$ is not minimal.
%\end{enumerate}
%Above, $3\in\I$ (resp. $2\in\I$) refer to the labels appearing in Table~1 %\ref{T:Dynkin}
%for the distinguished diagrams, $(\star)$.
%\end{dfn}

\begin{lem}
  \label{L:Lyndon max word}
For $\ui\in\Lr^+$, we have $\max(\r_\ui)=\ui$.
\end{lem}

\begin{proof}
We proceed by induction on the length $\ell(\ui)$, the case $\ui=i\in\I$ being clear.
For the inductive step, let $\ui=\ui_1\ui_2$ be the co-standard factorization of $\ui\in\Lr^+$. Induction applies to $\ui_1$ and $\ui_2$, so
$\max(\r_{\ui_1})=\ui_1$ and $\max(\r_{\ui_2})=\ui_2$. In particular, $\max(\r_{\ui_1}\shq \r_{\ui_2})\leq \max(\ui_1\shq \ui_2)$.
Since $\ui_1<\ui_2$ and the words appearing as shuffles in $\ui_1\shq \ui_2$ are the same as the words
appearing as shuffles in $\ui_1\shqi \ui_2$ and $\ui_2\shq \ui_1$, Lemma \ref{L:max shuffle} implies that
\[\max(\r_{\ui})=\max(\r_{\ui_1}\shqqi\r_{\ui_2})\leq\ui_2\ui_1.\]
Now $\ui_2\ui_1$ only appears in $\r_{\ui_1}\shq\r_{\ui_2}$ as a summand of $\ui_1\shq \ui_2$, and using \ref{E:Shuffle2} we see that
it appears with coefficient equal to 1, hence
\[\max(\r_{\ui})<\ui_2\ui_1.\]

We will prove that if $\uk\in\W^+$ occurs as a shuffle in $\r_{\ui_1}\shq\r_{\ui_2}$, and $\ui_1\ui_2\leq\uk<\ui_2\ui_1$,
then $\uk=\ui_1\ui_2$. To this end, we use Lemma~\ref{L:costd factorization}, which says that $\ui_2=\ui_1^r\ui_1'i$
where $r\geq0$, $\ui_1'$ is a (possibly empty) left factor of $\ui_1$ and $i\in\I$ is such that $\ui_1'i>\ui_1$.

Assume $\uk=\uk_1\cdots\uk_n$
is the canonical factorization of $\uk$ into a nonincreasing product of dominant Lyndon words.
Write $\ui_1=(i_1,\ldots,i_d)$ and $\ui_2=(i_1,\ldots,i_r)$. If $\uk$ occurs in $\r_{\ui_1}\shq\r_{\ui_1'i}$,
then by Lemma \ref{L:Lyndon left factor}, $\uk_1=(i_1,\ldots)$.
As $\ui_1$ is Lyndon, we have $i_1\leq i_s$ for any $s\leq d$.
In particular, the inequality $\uk_1\geq\uk_t$ now implies that $\uk_t=(i_1,\ldots)$ for all $t$.

%Assume now that $(I,\leq)$ is regular.
Assume until the last paragraph of this proof that if
$\UU$ of type $F(3|1)$ in Table~1 we consider only its distinguished diagram and $3\in\I$ is not minimal, or
if $\UU$ is of type $G(3)$ in Table~1 we consider only its distinguished diagram and $2\in\I$ is not minimal.
Here, $3\in\I$ (resp. $2\in\I$) refer to the labels appearing in Table~1
for the distinguished diagrams marked by $(\star)$.

An inspection of the root systems of basic Lie superalgebras implies that $n\leq 3$ since
$|\uk|\in\Phi^+$, and $n\af_{i_1}$ appears in its support. It follows that if $i_1$ occurs only once in $\ui$, then $\uk=\uk_1$ is Lyndon.
Since $|\uk|=|\ui|$ we must have $\uk=\ui$ as $\ui$ is the unique dominant Lyndon word of its degree. The $n=3$ case can only occur in type $G(3)$ (see \cite[p.45]{ya}) and corresponds to $|\ui|$
being a root of the Lie algebra of type $G_2$ where the result can be verified by inspection of \cite[$\S5.5.4$]{lec}.

Let us now consider the case where $i_1$ appears twice in $|\ui|$ and suppose $\uk=\uk_1\uk_2$ is the canonical factorization
of a word $\uk\in\W^+$ appearing in $\r_{\ui_1}\shq\r_{\ui_2}$. We want to show that $\uk_2=\emptyset$, so suppose otherwise.
By the assumption in the cases of $F(3|1)$ and $G(3)$,
we have $\ui=\ui_1\ui_2$, where $\ui_2=\ui_1'i$ and $\ui'$ is a left factor of $\ui_1$ (now, possibly empty or equal to $\ui$).

Suppose first that $\ui_1'\neq\emptyset$. Let $\uh$ be any word occurring as a summand
in $\r_{\ui_1}$,
let $\ul$ be any word occurring as a summand in $\r_{\ui_2}$, and assume that $\uk$ occurs as a shuffle in $\uh\shq\ul$.
First observe that $\uh=(i_1,h_2,\ldots, h_d)$ and $\ul=(i_1,l_1,\ldots, l_e)$ with $i_1<h_s$ and $i_1<l_t$ for all $s$, $t$.
Note  $\uk_1\neq\uh$ unless $\uh=\ui_1$ and, since $\uk_2\in \L^+$ is the unique dominant Lyndon word of
weight $|\ui|-|\ui_1|$, $\uk_2=\ui_2=\ul$ . Similarly,
$\uk_1\neq\ul$ unless $\ul=\ui_2$ and $\uk_2=\ui_1=\uh$.
The case $\uk_1=\ul$ contradicts the fact that $\uk<\ui_2\ui_1$, and the case $\uk_1=\ul$ contradicts $\uk_1>\uk_2$. So in either case, we arrive at a contradiction.

Next, observe that $\uk_1$ is not a proper left factor of $\uh$.
If it were, then $\uk_1\uk_2<\uh\leq\ui_1<\ui_1\ui_2$,
since $\uk_1\uk_2=(i_1,h_2,\ldots,h_r,i_1,\ldots)$ for some $r<d$ and $i_1<h_{r+1}$,
which is a contradiction with the choice of $\uk$.
Similarly, $\uk_1$ is not a proper left factor of $\ul$. If it were, then it would be less-than-or-equal-to
the corresponding left factor of $\ui_2$. As $\ui_2=\ui_1'i$, any proper left factor of $\ui_2$ is a left factor of $\ui_1$.
Hence, following the analysis of left factors of $\uh$, we arrive at a contradiction.
But then if $\uk_1$ is not equal to a left factor of $\uh$ or $\ul$, it must contain both $i_1$'s,
contradicting the assumption that $\uk_2\neq \emptyset$.

We are, therefore, left to consider the case where $\ui_1'=\emptyset$, so $\ui=\ui_1i$. Then, $\ui_1=\uj_1\uj_2'$
where $\ui=\uj_1\uj_2$ is the standard factorization of $\ui$ and $\uj_2=\uj_2'i$
(i.e. $\uj_2$ is a Lyndon word of maximal length).
We clearly have $\uj_1$ and $\uj_2$ of the form $\uj_1=(i_1,\ldots)$ and $\uj_2=(i_1,\ldots)$ and, since $\ui$ is Lyndon, $\uj_1<\uj_2$.
In fact, since $\uj_1\uj_2'=\ui_1$ is Lyndon,
\begin{align}\label{E:Bigger that left factor}
\uj_1<\uj_2'.
\end{align}
We make the following.

{\bf Claim ($\star$).}
%\begin{align}
%\label{E:Lyndon vector std factorization}
$\r_{\ui}=\r_{\uj_1}\shqqi\r_{\uj_2}.
$
%\end{align}

Assume the claim ($\star$) for the moment. Then, any $\uk=\uk_1\uk_2\in\Wr^+$ occurring in $\r_{\ui}$ must occur as a shuffle $\uh\shq\ul$
where $\uh\leq \uj_1$ occurs in $\r_{\uj_1}$ and $\ul\leq\uj_2$ occurs in $\r_{\uj_2}$. As before, $\uk_1$ cannot be a left factor of $\uh$
as this would imply $\uk=\uk_1\uk_2\leq\uj_1\uj_2=\ui$. We also cannot have $\uk_1$ as a left factor of $\ul$ unless $\uk_1\leq\uj_1$
(in which case $\uk<\ui$). Otherwise, write $\ul=\uk_1\ul''$. Then, $|\uk_2|=|\uj_1|+|\ul''|$. While it is not necessarily true that $|\ul''|\in\Phi^+$,
there exists $\bt\in\Phi^+\cup\{0\}$ and $\gm\in\Phi^+$ such that $|\uj_1|+\bt\in\Phi^+$ and $|\uj_1|+\bt+\gm=|\uk_2|$
(choose $\af_r\in\Pi$ in the support of $|\ul''|$ such that $|\uj_1|+\af_r\in\Phi^+$ and continue this process one simple root at a time
until arriving at $\bt$ such $|\ul''|-\bt\in\Phi^+$). Let ${\bf s}\in\Lr^+$ be the unique word of degree $|\uj_1|+\bt$. Since $i_1$ is not in
the support of $|\ul''|$, it is not in the support of $\bt$. Consequently, $\uj_1\ui(\bt)>\uj_1\uj_2=\ui$.
Therefore, by Proposition~\ref{P:ComputingLyndonWords}, it follows that ${\bf s}\geq \uj_1\cdot\ui(\bt)>\ui$. Hence,
$$\uk_2\geq{\bf s}\cdot\ui(\gm)>{\bf s}>\ui.$$
Appealing to Proposition \ref{P:ComputingLyndonWords} again, we see that $(|\uk_2|,|\uk_1|)\in C(|\ui|)$ and $\uk_2\uk_1>\ui$,
contradicting the maximality of $\ui$.
But again, if $\uk_1$ is not equal to a left factor of $\uh$ or $\ul$, it must contain both $i_1$'s,
contradicting the assumption that $\uk_2\neq \emptyset$. Then we see that $\uk_2=\emptyset$ and $\uk$ is Lyndon,
in which case the claim was already proven.
Therefore, we see that $\max(\r_{\ui_1}\shq \r_{\ui_2})\leq \ui$.
On the other hand, $\r_{\ui}=\r_{\ui_1}\shq \r_{\ui_2}$ is a nonzero element in $\UU_{|\ui|}$, hence
has a dominant word appearing with nonzero coefficient. Then by Corollary \ref{C:min std word}, this
implies $\ui$ appears with a nonzero coefficient and so the Lemma holds assuming ($\star$).

Finally, we prove the claim ($\star$) by induction on $\ell(\uj_2)$. To begin induction,
we note that $\ui=\ui_1i$, where $\ui_1={\uj_1}{\uj_2}'$, is the co-standard factorization and the computation below
will eventually reduce to the case where the standard and co-standard factorization of ${\uj_1}{\uj_2}'$
coincide (i.e. ${\uj_2}'={\uj_1}'j$ with ${\uj_1}'$ a left factor of ${\uj_1}$).

We now proceed to the inductive step. Observe that, by \eqref{P:qqi shuffle},
\begin{align}\label{E:Useful commutation}
\r_{\uj_1}\shq i = \pi^{p({\uj_1})p(i)}q^{-(|{\uj_1}|,\af_i)}i\shqi\r_{1_\ui},
\end{align}
since every word appearing in $\r_{\uj_1}$ is homogeneous of degree $|{\uj_1}|$.

Now, the co-standard factorization of ${\uj_2}$ is ${\uj_2}=({\uj_2}')i$, so
\begin{align*}
 \r_{\uj_1}\shqqi\r_{\uj_2}&
   =\r_{\uj_1}\shqqi(\r_{{\uj_2}'}\shqqi i)\\
    =&\r_{\uj_1}\shq(\r_{{\uj_2}'}\shq i)-\r_{\uj_1}\shq(\r_{{\uj_2}'}\shqi i) -\r_{\uj_1}\shqi(\r_{{\uj_2}'}\shq i)+\r_{\uj_1}\shqi(\r_{{\uj_2}'}\shqi i)\\
    =&\r_{\uj_1}\shq(\r_{{\uj_2}'}\shq i)-\pi^{p({\uj_2}')p(i)}q^{(|{\uj_2}'|,\af_i)}\r_{\uj_1}\shq(i\shq\r_{{\uj_2}'})\\
    &-\pi^{p({\uj_2})p(i)}q^{-(|{\uj_2}'|,\af_i)}\r_{\uj_1}\shqi(i\shqi\r_{{\uj_2}'})
    +\r_{\uj_1}\shqi(\r_{{\uj_2}'}\shqi i),
\end{align*}
where we have used \eqref{P:qqi shuffle} for the last equality.
On the other hand, the standard factorization of $\ui_1$ is $\ui_1={\uj_1}{\uj_2}'$. As $\ell({\uj_2}')<\ell({\uj_2})$,
induction applies and $\r_{\ui_1}=\r_{\uj_1}\shqqi\r_{\uj_2'}$. Hence,
\begin{align*}
\r_{\ui}
=&\r_{\ui_1}\shqqi i
 =(\r_{\uj_1}\shqqi\r_{{\uj_2}'})\shqqi i\\
    =&(\r_{\uj_1}\shq\r_{{\uj_2}'})\shq i-(\r_{\uj_1}\shq\r_{{\uj_2}'})\shqi i
      -(\r_{\uj_1}\shqi\r_{{\uj_2}'})\shq i+(\r_{\uj_1}\shqi\r_{{\uj_2}'})\shqi i\\
    =&(\r_{\uj_1}\shq\r_{{\uj_2}'})\shq i-\pi^{p({\uj_1})p(i)+p({\uj_2}')p(i)}q^{(|{\uj_1}|+|{\uj_2}'|,\af_i)}i\shq(\r_{\uj_1}\shq\r_{{\uj_2}'})\\
    &-\pi^{p({\uj_1})p(i)+p({\uj_2}')p(i)}q^{-(|{\uj_1}|+|{\uj_2}'|,\af_i)}i\shqi(\r_{\uj_1}\shqi\r_{{\uj_2}'})
    +(\r_{\uj_1}\shqi\r_{{\uj_2}'})\shqi i,
\end{align*}
where we have used \eqref{P:qqi shuffle} to obtain the last equality. Finally, using
Equation~\eqref{E:Useful commutation} and the associativity of $\shq$ and $\shqi$,
The claim ($\star$) follows.

%This completes the proof of the theorem when $(\I,\leq)$ is regular.

Finally, we consider the remaining diagrams and orderings when $\UU$ of type $F(3|1)$ or $G(3)$.
%If $(\I,\leq)$ is not regular, then the Cartan data is of type $F(3|1)$ or $G(3)$.
There are 6 orderings to consider in $F(3|1)$ and 2 orderings to consider in type $G(3)$.
Inspection of the root systems shows that the argument above proves that $\max(\r_{\ui})=\ui$
unless $|\ui|$ is either $\af_1+2\af_2+3\af_3+\af_4$ or $\af_1+2\af_2+3\af_3+2\af_4$ in type $F(3|1)$,
or $|\ui|$ is $\af_1+3\af_2+\af_3$, $\af_1+3\af_2+2\af_3$, or $\af_1+4\af_2+2\af_3$ in type $G(3)$.
A direct computation of $\r_{\ui}$ in these cases  yields the theorem.
\end{proof}

%\begin{rmk}\label{R:extension}
%It should be possible to remove the regular condition for $(\I,\leq)$
%in Lemma \ref{L:Lyndon max word} thereby avoiding the case-by-case considerations.
%\end{rmk}

%%%%%%%%%%%%%%%
%%%%%%%%%%%%%%%
\section{Orthogonal PBW Bases}\label{S:Orthogonal PBW}

In this section we will define a basis of PBW type for $\UU$ and show it is orthogonal with respect to the bilinear form on $\UU$.

%%%%%%%%%%%%%%%
\subsection{PBW Bases}

Let $\ui=\ui(\bt)\in\Lr^+$ for $\bt\in\Phi^+$. Set $d_\bt=\max\{|(\bt,\bt)|/{2},1\}$, and define the quantum numbers
$$
[n]_\bt=
\begin{cases}
[n]_i&\mbox{if }(\bt,\bt)=(\af_i,\af_i)\mbox{ and }\bt \in \Phi^+_\zero \cup \Phi^+_\iso,
  \\
\{n\}_i&\mbox{if }(\bt,\bt)=(\af_i,\af_i)\mbox{ and }\bt\in\Phi^+_\niso.
\end{cases}
$$
Let $\ui=\ui(\bt)=\ui_1\ui_2=\ui(\bt_1)\ui(\bt_2)$ be the co-standard factorization and set
$$
p_\ui=\max\{p\in\Z_{\geq0}\mid \bt_1-p\bt_2\in\Phi^+\}.
$$
Define $\kp_\ui$ inductively by the formula $\kp_\ui=1$ if $\ui=i\in\I$ and $\kp_\ui=[p_\ui+1]_{\bt_r}\kp_{\ui_1}\kp_{\ui_2}$ otherwise, where $(\bt_r,\bt_r)=\min\{(\bt_1,\bt_1),(\bt_2,\bt_2)\}$
(note that there is no ambiguity in this definition since in all cases where $\kp_\ui\neq 1$ and $(\bt_1,\bt_1)=(\bt_2,\bt_2)$ we have $p(\bt_1)=p(\bt_2)$).
Recalling the anti-automorphism $\sigma$ on $\UU$ from Proposition~\ref{P:automorphisms}
and the Lyndon basis $\{\r_\ui\mid \ui\in\Wr^+\}$ for $\UU$ from Proposition~\ref{P:Lyndon triangularity}, we define
\begin{align}\label{E:PBW Definition 1}
\E_\ui&=\kp_\ui^{-1}\sm(\r_\ui),\;\;\;\ui\in\L^+.
\end{align}
We note that in the case of Lie algebras, this renormalization factor is the one computed in \cite[Theorem 4.2]{bkmc}.

More generally, if $\ui=\ui_1^{n_1}\cdots\ui_d^{n_d}$ is the canonical factorization of $\ui$ with $\ui_1>\cdots>\ui_d$, set
\begin{align}\label{E:PBW Definition 2}
E_\ui=\E_{\ui_d}^{(n_d)}\shq\cdots\shq\E_{\ui_1}^{(n_1)}
\end{align}
where, for $\uj\in\Lr^+$, we have denoted
$$
\E_\uj^{(n)}=\E_\uj^{\shq n}/[n]_\uj!.
$$

We first state the following theorem, which is a generalization of \cite[Theorem 36]{lec} and follows from Lemma~\ref{L:Lyndon max word}.

\begin{thm}
\label{T:max word}
We have $\max(\r_\ui)=\max(\E_\ui)=\ui$, for $\ui\in\Wr^+$.
\end{thm}

\begin{proof}
It follows by Lemma~\ref{L:Lyndon max word}
that $\max(\E_\ui)=\ui$, for $\ui\in\Lr^+$, since $\E_{\ui}$ is proportional to $\sm(\r_{\ui})$.
Now the theorem follows by applying Lemma~\ref{L:max shuffle}.
\end{proof}

\begin{cor}
If $\ui\in\L^+_\one$, then $\E_\ui\shq \E_\ui=0$.
\end{cor}

\begin{proof}
By Theorem~\ref{T:max word} and Corollary~\ref{C:not max}, $\max(\E_\ui \shq \E_\ui)<\ui\ui$.
However, by \cite[Lemma~ 5.9]{klram2}, $\ui\ui$ is smaller than any dominant word of degree $2|\ui|$. Hence, $\E_\ui \shq \E_\ui$ must be 0.
\end{proof}

\begin{prp}\label{P:Bar Invariant Renormalization}
For each $\ui\in\Wr^+$, there exists $\kappa_\ui\in\A$ such that $\bar{\kappa_\ui}=\kappa_\ui$, and $\E_\ui=\kappa_\ui^{-1}\;\sm(\r_\ui)$.
\end{prp}

\begin{proof}
This is by definition, taking
\begin{align}\label{E:kappa i}
\kp_\ui=\prod_{s=1}^d\kp_{\ui_s}\,[n_s]_{\ui_s}!.
\end{align}
See \eqref{E:PBW Definition 1} and \eqref{E:PBW Definition 2} above.
\end{proof}
It follows from Propositions~\ref{P:Lyndon triangularity} and \ref{P:Bar Invariant Renormalization}
that $\{E_\ui\mid \ui\in\Wr^+\}$ forms a basis for $\UU$, which will be called a {\bf PBW basis}.

\begin{prp}\label{P:PBW Triangularity}
For $\ui\in\Wr^+$, we have
$$\E_\ui=\kp_\ui^{-1}\ep_{\tau(\ui)}+\sum_{\uj >\ui} \af_{\ui\uj} \ep_{\tau(\uj)},
\qquad \text{ for } \af_{\ui\uj} \in \Q(q).
$$
\end{prp}

\begin{proof} This is immediate from Proposition \ref{P:Lyndon triangularity} and Proposition \ref{P:Bar Invariant Renormalization}.\end{proof}

The next theorem is often referred to as the Levendorskii-Soibelman formula \cite{ls}.

\begin{thm}\label{T:LS formula}
Suppose $\ui,\uj\in\L^+$ with $\ui<\uj$. Then,
$$\E_\uj\shq\E_\ui=\sum_{\substack{\uk\in\W^+\\\ui\uj\leq\uk\leq\uj\ui}}c_{\ui,\uj}^\uk\E_\uk.$$
\end{thm}

\begin{proof}
By Proposition~\ref{P:PBW Triangularity},
\begin{align*}
\E_{\uj}\shq\E_{\ui}&=\left(\kp_\uj^{-1}\ep_{\tau(\uj)}+\sum_{\uk>\uj}\af_{\uj,\uk}\ep_{\tau(\uk)}\right) \left(\kp_\ui^{-1}\ep_{\tau(\ui)}+\sum_{\uk>\ui}\af_{\ui,\uk}\ep_{\tau(\uk)}\right)\\
    &=\sum_{\uk\in\W,\uk>\ui\uj}\bt_{\ui\uj}^\uk\ep_{\tau(\uk)}
\end{align*}
By Lemma~\ref{L:Standard Triangularity}, if $\uk\notin\W^+$, then
$$\ep_{\tau(\uk)}=\sum_{\uh\in\W^+,\uh>\uk}\gm_{\uk,\uh}\ep_{\tau(\uh)}.$$
Therefore,
$$\E_\ui\shq\E_\uj=\sum_{\substack{\uk\in\W^+\\\ui\uj\leq\uk}}c_{\ui,\uj}^\uk\E_\uk.$$
On the other hand, by Theorem \ref{T:max word}, it follows that $c_{\ui,\uj}^\uk\neq 0$ only if $\uk<\uj\ui$.
\end{proof}

%%%%%%%%%%%%%%%%%%%
\subsection{Orthogonality of PBW basis}

We will prove that the PBW basis defined in the previous section is orthogonal with respect to the bilinear form on $\UU$.

\begin{lem}\label{Conj:PBW coproduct}
For $\ui\in\Lr^+$, we have
$$
\Dt(\E_\ui)=\sum_{\ui_1,\ui_2\in\Wr^+}\vartheta_{\ui_1\ui_2}^\ui\E_{\ui_2}\otimes \E_{\ui_1},
\qquad \text{ for } \; \vartheta_{\ui_1,\ui_2}^\ui \in \Q(q),
$$
where $\vartheta_{\ui_1,\ui_2}^\ui=0$ unless $|\ui_1|+|\ui_2|=|\ui|$ and
\begin{enumerate}
\item $\ui_1\leq\ui$, and
\item $\ui\leq\ui_2$ whenever $\ui_2\neq\emptyset$.
\end{enumerate}
\end{lem}

\begin{proof}
Observe by Theorem \ref{T:max word}
that $\E_{\ui}=\sum_{\uj\leq\ui}\phi_{\ui\uj}\uj$, for some $\phi_{\ui\uj}\in\Q(q)$, so
$$\Dt(\E_\ui)=\sum_{\substack{\uj_1,\uj_2;\\\uj_1\uj_2=\uj\leq \ui}}\phi_{\ui\uj}(\uj_2\otimes\uj_1).$$
Since $\uj_1\leq\uj\leq\ui$, Part (1) follows.

We now prove (2) by induction on the length of $\ui$, the case $\ui=i\in\I$ being obvious.

To proceed to the inductive step, we need to make a few observations. First, given $\ui\in\Lr^+$,  $\E_\ui$ is proportional to
$$\sm(\r_\ui)=\sm(\r_{\ui_2})\shq\sm(\r_{\ui_1})-\pi^{p(\ui_1)p(\ui_2)}q^{-(|\ui_1|,|\ui_2|)}\sm(\r_{\ui_1})\shq\sm(\r_{\ui_1}),$$
where $\ui=\ui_1\ui_2$ is the costandard factorization of $\ui$. In turn, the right hand side of the equation above is proportional to
$$\E_{\ui_2}\shq\E_{\ui_1}-\pi^{p(\ui_1)p(\ui_2)}q^{-(|\ui_1|,|\ui_2|)}\E_{\ui_1}\shq\E_{\ui_2}
=-\pi^{p(\ui_1)p(\ui_2)}q^{-(|\ui_1|,|\ui_2|)}(\E_{\ui_1}\shqqi\E_{\ui_2}).$$
Therefore, it is sufficient to prove the lemma for $\E_{\ui_1}\shqqi\E_{\ui_2}$.

To this end, write $\ui_1=\uj$ and $\ui_2=\uk$ and note that induction applies to $\E_{\uj}$ and $\E_{\uk}$.
Observe that if $\Dt(\E_{\uj}\shq\E_{\uk})=\sum_{\uh,\ul\in\W^+}z_{\uh,\ul}(\E_\uh\otimes\E_\ul)$, then
\begin{align}\label{E:qqinv coproduct}
\Dt(\E_{\uj}\shqqi\E_{\uk})=\sum_{\uh,\ul\in\W^+}(z_{\uh,\ul}-\bar{z_{\uh,\ul}})(\E_\uh\otimes\E_\ul)
\end{align}
since, replacing $q$ with $q^{-1}$ in Proposition \ref{P:ShuffleCoproduct} shows that $\Dt$ is an algebra homomorphism
with respect to the $(q^{-1},\pi)$-bialgebra structure on $\UU\otimes\UU$:
$$(w\otimes x)\shqi (y\otimes z) = \pi^{p(x)p(y)}q^{(|x|,|y|)}(w\shqi y)\otimes(x\shqi z).$$
On the other hand,
$$\Dt(\E_{\uj}\shqqi\E_{\uk})=\Dt(\E_{\uk}\shq\E_{\uj}-\pi^{p(\uj)p(\uk)}q^{-(|\uj|,|\uk|)}\E_{\uj}\shq\E_{\uk}).$$
By Proposition \ref{P:PBW Triangularity}, the transition matrix from the PBW basis to the basis $\{\ep_{\tau(\uj)}\mid \uj\in \Wr^+\}$ is triangular.
Therefore, applying our inductive hypothesis, we have
\begin{align*}
\Dt(\E_{\uj}\shq \E_{\uk})&=\sum_{\substack{\uj_1\leq\uj\leq\uj_2\\\uk_1\leq\uk\leq\uk_2}}
\vartheta_{\uj_1\uj_2}^\uj\vartheta_{\uk_1\uk_2}^\uk(\E_{\uj_2}\otimes\E_{\uj_1})\shq(\E_{\uk_2}\otimes\E_{\uk_1})
    =\sum_{\substack{\uh\geq\uk_2\uj_2;\\\ul\geq\uk_1\uj_1}}\Theta_{\uh,\ul}\E_\uh\otimes\E_\ul
\end{align*}
and similarly
\begin{align*}
\Dt(\E_{\uk}\shq \E_{\uj})&=\sum_{\substack{\uk_1\leq\uk\leq\uk_2\\\uj_1\leq\uj\leq\uj_2}}
\vartheta_{\uk_1\uk_2}^\uk\vartheta_{\uj_1\uj_2}^\uj(\E_{\uk_2}\otimes\E_{\uk_1})\shq(\E_{\uj_2}\otimes\E_{\uj_1})
    =\sum_{\substack{\uh\geq\uj_2\uk_2;\\\ul\geq\uj_1\uk_1}}\Theta'_{\uh,\ul}\E_\uh\otimes\E_\ul.
\end{align*}
Comparing these equations to \eqref{E:qqinv coproduct} we deduce that $\Theta_{\uh,\ul}\neq0$ if and only if $\Theta'_{\uh,\ul}\neq0$.

Now, assume $z_{\uh\ul}-\bar{z_{\uh\ul}}\neq 0$. The previous paragraph implies that $\uh\geq\max\{\uj_2\uk_2,\uk_2\uj_2\}$.
If $\uj_2\neq\emptyset$, then $\uj\neq \emptyset$ and we obtain the inequality $\uh\geq\uj_2\uk_2\geq\uj\uk=\ui$
since $\uj_2\geq\uj$, $\uk_2\geq\uk$ and these are right factors of $\uj$ and $\uk$ respectively (note that if $\uj_2$ is a proper right factor,
we don't need to consider $\uk$ and $\uk_2$ at all). If $\uj_2=\emptyset$ and $\uk_2\neq\emptyset$,
we have $\uh\geq\uk_2\geq\uk>\uj\uk$ since, by Lemma~\ref{L:costd factorization}, $\uk=\uj^r\uj'j$ where $r\geq 0$,
$\uj'$ is a (possibly empty) left factor of $\uj$ and $j\in\I$ satisfies $\uj'j>\uj$. If both $\uj_2=\uk_2=\emptyset$,
the equality $|\uh|=|\uj_2|+|\uk_2|$ forces $\uh=\emptyset$. This proves part (2) and hence the lemma.
\end{proof}

\begin{thm}\label{T:OrthogonalPBW}
Let $\ui,\uj\in\Wr^+$. Then, $(\E_\ui,\E_\uj)=0$ unless $\ui=\uj$.
Moreover, if $\ui=\ui_1^{n_1}\cdots\ui_d^{n_d}$, $\ui_1>\cdots>\ui_d$ is the canonical factorization of $\ui$ into dominant Lyndon words, then,
$$
(\E_\ui,\E_\ui)=\pi^{\xi_\ui}q^{-c_\ui}  \prod_{l=1}^d \frac{(\E_{\ui_l},\E_{\ui_l})^{n_l}}{[n_l]_{\ui_l}! },
$$
where
\begin{align}\label{E:xi and c}
\xi_\ui=\sum_{l=1}^d{n_l-1\choose 2}p(\ui_l) \andeqn c_\ui=\sum_{l=1}^d{n_l\choose 2}\frac{(|\ui_l|,|\ui_l|)}{2}.
\end{align}
\end{thm}

\begin{proof}
We proceed by induction on the length of $\ui$, the case $\ui=i\in\I$ being trivial.
We first show that the theorem holds when $\ui\in\Lr^+$. Indeed, suppose $\uj\neq\ui$ and let $\uj=\uj_1\cdots\uj_r$, where
$\uj_1\geq\uj_2\cdots\geq\uj_r$, be the canonical factorization of $\uj$. Then, $(\E_\ui,\E_\uj)$ is proportional to
\begin{align}\label{E:Orthogonality of Lyndon}
\sum\vartheta_{\ui_1,\ui_2}^\ui(\E_{\ui_2}\otimes\E_{\ui_1},\E_{\uj_r}&\otimes(\E_{\uj_{r-1}}\shq\cdots\shq\E_{\uj_1}))
    &=\sum\vartheta_{\ui_1,\ui_2}^\ui(\E_{\ui_2},\E_{\uj_r})(\E_{\ui_1},(\E_{\uj_{r-1}}\shq\cdots\shq\E_{\uj_1}))
\end{align}
where the sum is over $\ui_1\leq\ui\leq\ui_2$ by Lemma \ref{Conj:PBW coproduct}. By assumption $|\uj_r|\neq|\ui|$,
so we may take the sum to be over $\ui_1<\ui<\ui_2$. Therefore, since $\uj_r\in\Lr^+$ has shorter length than $\ui$,
we may apply induction to conclude that the nonzero terms in the sum above satisfy $\ui_2=\uj_r\in\Lr^+$
and $\uj_1\cdots\uj_{r-1}=\ui_1$. But, now we have the inequalities
$$\uj_1\leq\uj_1\cdots\uj_{r-1}=\ui_1<\ui_2=\uj_r\leq \uj_1,$$
which is never satisfied. Hence, $(\E_\ui,\E_\uj)=0$.

Now, let $\ui,\uj\in\W^+_\nu$ be arbitrary and assume we have shown that $\{\E_{\uk}\mid\uk\in\W^+_\mu\}$ is an orthogonal basis for $\UU_{\mu}$ whenever $\mu<\nu$ in the dominance ordering on $Q^+$ (the base case $\nu\in\Pi$ being trivial). Let $\ui=\ui_1\cdots\ui_s$ and $\uj=\uj_1\cdots\uj_r$ be the canonical factorizations
of $\ui$ and $\uj$ into a nonincreasing product of dominant Lyndon words, and assume, without loss of generality,
that $\ui_1\leq\uj_1$. If $\ui\in\L^+$ or $\uj\in\L^+$, then we are done, so assume that both $r,s>1$. Then, $(\E_\ui,\E_\uj)$ is
proportional to (up to some suitable product of quantum factorials)
\begin{align}\label{E:Orthogonality Standard}
\nonumber(\E_{\ui_s}\shq\cdots\shq&\E_{\ui_1},\E_{\uj_r}\shq\cdots\shq\E_{\uj_1})
    =\big (\Dt(\E_{\ui_s})\shq\cdots\shq\Dt(\E_{\ui_1}),(\E_{\uj_r}\shq\cdots\shq\E_{\uj_2})\otimes\E_{\uj_1} \big)\\
    &=\sum \vartheta_{\ui_{1,2},\ldots, \ui_{s,2}}(\E_{\ui_{s,2}}\shq\cdots\shq\E_{\ui_{1,2}},\E_{\uj_r}\shq\cdots\shq\E_{\uj_2})
    (\E_{\ui_{s,1}}\shq\cdots\shq\E_{\ui_{1,1}},\E_{\uj_1})
\end{align}
where this sum is as in Lemma \ref{Conj:PBW coproduct}; in particular,
$\ui_{t,1}\leq\ui_t$, $\ui_{t,1}\in\Wr^+$, for all $1\leq t\leq s$ (note that $\ui_{t,1}$ may be $\emptyset$).

\smallskip

\noindent\textbf{Claim $(\star\star)$.}
We have $(\E_{\ui_{s,1}}\shq\cdots\shq\E_{\ui_{1,1}},\E_{\uj_1})=0$
unless there is a unique $k$ such that $\ui_{k,1}=\uj_1$ and ${\ui_{t,1}}=\emptyset$ for $t\neq k$.

\smallskip

It is not necessarily the case that $\E_{\ui_{s,1}}\shq\cdots\shq\E_{\ui_{1,1}}$ belongs to the PBW basis, so we cannot apply earlier arguments. Therefore, suppose that $k$ is maximal such that $\ui_{k,1}\neq\emptyset$. Then,
\begin{align*}
(\E_{\ui_{k,1}}\shq\cdots\shq\E_{\ui_{1,1}},\E_{\uj_1})
    &=\sum\vartheta_{\uj_{1,1},\uj_{1,2}}^{\uj_1}(\E_{\ui_{k,1}},\E_{\uj_{1,2}})(\E_{\ui_{k-1,1}}\shq\cdots\E_{\ui_{1,1}},\E_{\uj_{1,1}})
\end{align*}
where the sum is as in Lemma \ref{Conj:PBW coproduct}. Consider one such term in the sum above:
\begin{align}\label{E:Orthogonality Inductive Step}
(\E_{\ui_{k,1}},\E_{\uj_{1,2}})(\E_{\ui_{k-1,1}}\shq\cdots\E_{\ui_{1,1}},\E_{\uj_{1,1}}).
\end{align}
Assume this term is nonzero. Since $|\ui_{k,1}|\leq|\ui_k|<|\ui|$ and $|\uj_{2,1}|\leq|\uj_1|<|\uj|$ in the dominance ordering on $Q^+$, induction on $Q^+$-grading implies that $(\E_{\ui_{k,1}},\E_{\uj_{1,2}})=0$ unless $\ui_{k,1}=\uj_{1,2}$. Therefore, $\uj_{1,2}\neq\emptyset$ and
$\uj_{1,2}=\ui_{k,1}\leq\ui_k\leq\ui_1\leq\uj_1\leq\uj_{1,2}.$
Hence $\uj_{1,2}=\uj_1$ and $\uj_{1,1}=\emptyset$. Since \eqref{E:Orthogonality Inductive Step} is nonzero,
$$(\E_{\ui_{k-1,1}}\shq\cdots\E_{\ui_{1,1}},\E_{\uj_{1,1}})=(\E_{\ui_{k-1,1}}\shq\cdots\E_{\ui_{1,1}},1)\neq0,$$
so $\ui_{k-1,1}=\cdots=\ui_{1,1}=\emptyset$. Claim $(\star\star)$ follows.

Now, assume that $(\E_{\ui_{s,1}}\shq\cdots\shq\E_{\ui_{1,1}},\E_{\uj_1})\neq0$. Then, there is a unique $k$ such that $\ui_{k,1}=\uj_1$ and ${\ui_{t,1}}=\emptyset$ for $t\neq k$. Since
$\uj_1=\ui_{k,1}\leq\ui_k\leq\ui_1\leq\uj_1$,
it follows that $\ui_{k,1}=\ui_k=\ui_1=\uj_1$.

Let $n_1\geq 1$ be maximal such that $\ui_1=\ui_2=\cdots =\ui_{n_1}$.
Then, it follows from the previous arguments and the algebra structure on $\UU\otimes\UU$ that \eqref{E:Orthogonality Standard} becomes
\begin{align*}
(\E_{\ui_s} & \shq\cdots  \shq\E_{\ui_1},\E_{\uj_r}\shq\cdots\shq\E_{\uj_1})
 \\
    = &(1+\pi^{p(\ui_1)}q^{-(|\ui_1|,|\ui_1|)}+\cdots+\pi^{(n_1-1)p(\ui_1)}q^{-(n_1-1)(|\ui_1|,|\ui_1|)}) \times
    \\
&   \qquad  (\E_{\ui_s}\shq\cdots\shq\E_{\ui_2},\E_{\uj_t}\shq\cdots\shq\E_{\uj_2})(\E_{\ui_1},\E_{\ui_1}).
    \end{align*}
We may now complete by induction the computation of $(\E_{\ui_s}\shq\cdots\shq\E_{\ui_1},\E_{\uj_r}\shq\cdots\shq\E_{\uj_1})$
and then $(\E_\ui,\E_\uj)$, which yields the formula as stated in the theorem.
\end{proof}

Now we define the \textbf{dual PBW basis} for $\UU$
\begin{align}\label{E:Dual PBW}
\E^*_\ui=\E_\ui/(\E_{\ui},\E_{\ui}), \qquad \text{ for } \ui\in\Wr^+.
\end{align}

%%%%%%%%%%%%%%%%%%%%%%
%%%%%%%%%%%%%%%%%%%%%%
\section{Computations of Dominant Lyndon Words and Root Vectors}
\label{S:Computations}

In this section we will compute the dominant Lyndon words, Lyndon and (dual) PBW
root vectors explicitly for general Dynkin diagrams of type $A$-$D$.
Throughout this section, we will set $M=m+n+1$ and continue to order $\I=\{1,\ldots,M\}$
as specified in Table ~1. %\ref{T:Dynkin}}.
We also remind the reader of the notation $s_{ij}$ from \eqref{E:signs}.

%%%%%%%%%%%%%%%%%%%%%
\subsection{Type $A(m,n)$}

A general Dynkin diagram of type $A(m,n)$ is of the form
$$ 
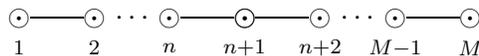
\xy
(-30,0)*{\odot};(-20,0)*{\odot}**\dir{-};(-15,0)*{\cdots};(-10,0)*{\odot};(0,0)*{\odot}**\dir{-};
(0,0)*{\odot};(10,0)*{\odot}**\dir{-};
(15,0)*{\cdots};(20,0)*{\odot};(30,0)*{\odot}**\dir{-};
(-30,-4)*{\scriptstyle 1};(-20,-4)*{\scriptstyle 2};(-10,-4)*{\scriptstyle n};(0,-4)*{\scriptstyle n+1};(10,-4)*{\scriptstyle n+2};
(20,-4)*{\scriptstyle M-1};(30,-4)*{\scriptstyle M};
(0,-8)*{};(0,8)*{};
\endxy$$

The next proposition computes the set of dominant Lyndon words inductively using Proposition \ref{P:ComputingLyndonWords}.

\begin{prp}
The set of dominant Lyndon words is
$$\Lr^+=\{(i,\ldots,j)\mid 1\leq i\leq j\leq M\}.$$
\end{prp}

Having computed $\Lr^+$, we now compute the Lyndon basis using Proposition \ref{P:rofL}.
For $\ui=(i,\ldots,j)$ with $1\le i \le j \le M$, we set
$$\varpi_A(\ui)=\prod_{k=i}^{j-1}s_{k,k+1}.
$$

\begin{prp}\label{P:Type A Lyndon vectors}
For $\ui=(i,\ldots,j)$  with $1\le i \le j \le M$, the  Lyndon root vector is
$$\r_\ui=\pi^{P(|\ui|)}\pi^{j-i}{\varpi_A(\ui)}(q-q^{-1})^{j-i}(i,\ldots,j).
$$
\end{prp}

\begin{proof}
We proceed by induction on $j-i$, the case $j-i=0$ being trivial.
Note that if $\ui=(i,\ldots,j)$, and $\ui=\ui_1\ui_2$ is the co-standard factorization of $\ui$,
then $\ui_1=(i,\ldots,j-1)$ and $\ui_2=j$. By induction, we compute
\begin{align*}
\r_\ui=&\r_{\ui_1}\shqqi\r_{\ui_2}\\
    =&\pi^{P(|\ui_1|)}\pi^{j-i+1}\varpi_A(\ui_1)(q-q^{-1})^{j-i-1}(i,\ldots,j-1)\shqqi j\\
    =&\pi^{P(|\ui_1|)}\pi^{j-i+1}\varpi_A(\ui_1)(q-q^{-1})^{j-i-1}((i,\ldots,j-2)\shqqi j)(j-1)\\
    &+\pi^{P(|\ui_1|)}\pi^{j-i+1}\varpi_A(\ui_1)(q-q^{-1})^{j-i-1}\pi^{p(i,\ldots,j-1)p(j)}(q^{-(\af_{j-1},\af_j)}-q^{(\af_{j-1},\af_j)})(i,\ldots,j)\\
    =&\pi^{P(|\ui_1|)}\pi^{j-i+1}\varpi_A(\ui_1)(q-q^{-1})^{j-i-1}\pi^{p(i,\ldots,j-1)p(j)}(q^{-(\af_{j-1},\af_j)}-q^{(\af_{j-1},\af_j)})(i,\ldots,j).
\end{align*}
The proof now follows by the observations $q^{-(\af_{j-1},\af_j)}-q^{(\af_{j-1},\af_j)}=-s_{j-1,j}(q-q^{-1})$ and $P(|\ui_1|)+p(i,\ldots,j-1)p(j)=P(|\ui|)$.
\end{proof}

\begin{cor}
Let $\ui=(i,\ldots,j)$ with $1\le i \le j \le M$. Then,
\begin{enumerate}
\item
the  PBW root vector is
$\E_\ui=\varpi_A(\ui)(q-q^{-1})^{j-i}q^{-N(|\ui|)}(i,\ldots,j)$;

\item   $(\E_\ui,\E_\ui)=\varpi_A(\ui)(q-q^{-1})^{j-i}q^{-N(|\ui|)}$;

\item  $\E_\ui^*=(i,\ldots,j)$.
\end{enumerate}
\end{cor}

\begin{proof}
The formula (1) for $\E_\ui$ is clear from the definitions, and  Part (3) follows immediately from (1) and (2).
So it remains to prove (2).
To this end, let $\ui=\ui_1\ui_2$ be the co-standard factorization of $\ui$, $\ui_1=(i,\ldots,j-1)$, $\ui_2=j$. Note that
$$\E_\ui=\E_j\shq\E_{\ui_1}-\pi^{p(j)p(\ui_1)}q^{-(\af_{j-1},\af_j)}\E_{\ui_1}\shq\E_j.$$
Therefore, using Proposition \ref{P:bilinearformonF}, we have
\begin{align*}
(\E_\ui,\E_\ui)=&\varpi_A(\ui)(q-q^{-1})^{j-i}q^{-N(|\ui|)}(\ui,\E_\ui)\\
    =&\varpi_A(\ui)(q-q^{-1})^{j-i}q^{-N(|\ui|)}(j\otimes(i,\ldots,j-1),\E_j\otimes\E_{\ui_1})\\
    =&s_{j-1,j}(q-q^{-1})q^{\frac12(2(|\ui_1|,\af_j)-(\af_j,\af_j))}(\E_j,\E_j)(\E_{\ui_1},\E_{\ui_1}).
\end{align*}
Therefore, (2) follows by induction.
\end{proof}

%%%%%%%%%%%%%%%%%%%
\subsection{Type $B(m,n+1)$}

A general Dynkin diagram of type $B(m,n+1)$ is of the form

$$ 
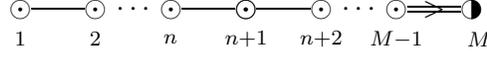
\xy
(-30,0)*{\odot};(-20,0)*{\odot}**\dir{-};(-15,0)*{\cdots};(-10,0)*{\odot};(0,0)*{\odot}**\dir{-};
(0,0)*{\odot};(10,0)*{\odot}**\dir{-};
(15,0)*{\cdots};(20,0)*{\odot};(30,0)*{\yy}**\dir{=};(25,0)*{>};
(-30,-4)*{\scriptstyle 1};(-20,-4)*{\scriptstyle 2};(-10,-4)*{\scriptstyle n};(0,-4)*{\scriptstyle n+1};(10,-4)*{\scriptstyle n+2};
(20,-4)*{\scriptstyle M-1};(31,-4)*{\scriptstyle M};
(0,-8)*{};(0,8)*{};
\endxy$$
In order to facilitate computations below, we note the following properties of  the signs $s_{ij}$ ($i,j\in\I$) given in \eqref{E:signs}.

\begin{lem}   \label{P:Type B signs}
\begin{enumerate}
\item if $a_{ii}=0$, then $s_{i-1,i}=\pi s_{i,i+1}$;
\item if $a_{ii}\neq 0$, then $s_{i-1,i}=s_{i,i+1}=\pi s_{ii}$;
\item for any $k,l\in\I$ with $k\neq l$, we have
%\begin{align}\label{E:Type B root pairing}
$(\af_k,\af_l)\in\{2s_{kl},0\}.$
%\end{align}
\end{enumerate}
\end{lem}

\begin{proof}
This follows immediately using the standard $\ep\delta$-notation for the root system and the simple systems of type $B$; cf. \cite{kac, CW12}.
The factor $2$ in Part (3) is due to the normalization of $(\cdot,\cdot)$ adopted in \S\ref{SS:Root Data}.
\end{proof}

\begin{prp} The set of dominant Lyndon words is
$$\Lr^+=\{(i,\ldots,j)\mid 1\leq i\leq j\leq M\}\cup\{(i,\ldots,M,M,\ldots,j+1)\mid 1\leq i\leq j< M\}.$$
\end{prp}

We set
$$
\varpi_B(\ui)=\begin{cases}\varpi_A(\ui)&\mbox{if }\ui=(i,\ldots,j), \text{ for } 1\le i \le j \le M,\\
\varpi_A(i,\ldots j)\pi^{p(M)}&\mbox{if }\ui=(i,\ldots,M,M,\ldots,j+1), \text{ for } 1\le i \le j<M.\end{cases}
$$
%where $\om_B(\ui)=\sum_{k=j}^{M-2}p(i,\ldots,k+1)p(k+1)$.

\begin{prp}\begin{enumerate}
\item For $\ui=(i,\ldots,j)$ ($1\leq i\leq j\leq M$), the  Lyndon root vector is
$$\r_{\ui}=\pi^{P(|\ui|)}\pi^{j-i}\varpi_B(\ui)(q^2-q^{-2})^{j-i}(i,\ldots,j).
$$

\item For $\ui=(i,\ldots,M,M,\ldots,j+1)$ with $1\leq i\leq j < M$,  the Lyndon root vector is
$$\r_\ui=\pi^{P(|\ui|)}\pi^{i+j}\varpi_B(\ui)(q^2-q^{-2})^{2M-i-j}(i,\ldots,M,M,\ldots,j+1).$$
\end{enumerate}
\end{prp}

\begin{proof}
The proof of Part (1) is same as for type $A$ in Proposition \ref{P:Type A Lyndon vectors}.

We prove (2) by downward induction on $j$. For $j=M-1$, $\ui=(i,\ldots,M,M)$
and the co-standard factorization is $\ui=\ui_1\ui_2$ where $\ui_1=(i,\ldots,M)$ and $\ui_2=M$. Therefore,
\begin{align*}
\r_\ui=&\pi^{P(|\ui_1|)}\pi^{M-i}\varpi_A(\ui_1)(q^2-q^{-2})^{M-i}(i,\ldots,M)\shqqi M\\
    =&\pi^{P(|\ui_1|)}\pi^{M-i}\varpi_A(\ui_1)(q^2-q^{-2})^{M-i}((i,\ldots,M-1)\shqqi M)M\\
    &+\pi^{P(|\ui_1|)}\pi^{M-i}\varpi_A(\ui_1)(q^2-q^{-2})^{M-i}\pi^{p(i,\ldots,M)p(M)}((i,\ldots M,M)-(i,\ldots,M,M))\\
    =&\pi^{P(|\ui_1|)}\pi^{M-i}\varpi_A(\ui_1)(q^2-q^{-2})^{M-i}((i,\ldots,M-1)\shqqi M)M\\
    =&\pi^{P(|\ui_1|)}\pi^{M-i}\varpi_A(\ui_1)(q^2-q^{-2})^{M-i}\pi^{p(i,\ldots,M-1)p(M)} \\
    &\qquad\qquad\qquad\qquad\qquad \times (q^{-(\af_{M-1},\af_M)}-q^{(\af_{M-1},\af_M)})(i,\ldots,M,M).
\end{align*}
This case now follows since
$q^{-(\af_{M-1},\af_M)}-q^{(\af_{M-1},\af_M)}=\pi s_{M-1,M}(q^2-q^{-2})$ by Lemma~\ref{P:Type B signs}(3)
and $\pi^{P(|\ui_1|)+p(i,\ldots,M-1)p(M)}=\pi^{P(|\ui|)+p(M)}$.

We now proceed to the general case $\ui=(i,\ldots,M,M,\ldots,j+1)$. Let $\ui=\ui_1\ui_2$ be the co-standard factorization,
with $\ui_1=(i,\ldots,M,M,\ldots,j+2)$ and $\ui_2=j+1$. Then,
\begin{align}\label{E:Type B Lyndon vector}
\nonumber\r_\ui=&\pi^{P(|\ui_1|)}\pi^{i+j-1}\varpi_B(\ui_1)(q^2-q^{-2})^{2M-i-j-1}(i,\ldots,M,M,\ldots,j+2)\shqqi(j+1)\\
\nonumber=&\pi^{P(|\ui_1|)}\pi^{i+j-1}\varpi_B(\ui_1)(q^2-q^{-2})^{2M-i-j-1}(i,\ldots,M,M,\ldots,j+3)\shqqi(j+1))(j+2)\\
\nonumber&+\pi^{P(|\ui_1|)}\pi^{i+j-1}\varpi_B(\ui_1)(q^2-q^{-2})^{2M-i-j-1}\\
    &\times\pi^{p(j+1)p(\ui_1)}(q^{-(|\ui_1|,\af_{j+1})}-q^{(|\ui_1|,\af_{j+1})})(i,\ldots,M,M,\ldots,j+1)
\end{align}
Using Lemma~ \ref{P:Type B signs}, we have that $$-(|\ui_1|,\af_{j+1}))=-(\af_j+\af_{j+1}+2\af_{j+2},\af_{j+1})=-2s_{j+1,j+2}.$$
Also, $P(|\ui_1|)+p(j+1)p(\ui_1)=P(|\ui|)$. Therefore, last term in \eqref{E:Type B Lyndon vector} above is
$$\pi^{P(|\ui|)}\pi^{2M-i-j}\varpi_B(\ui)(q^2-q^{-2})^{2M-i-j-1}(q^2-q^{-2})(i,\ldots,M,M,\ldots,j+1).$$
Hence, the proposition will follow if we can show that
\begin{align}\label{E:Type B Lyndon vector 2}
((i,\ldots,M,M,\ldots,j+3)&\shqqi(j+1))(j+2)=0.
\end{align}
Indeed, since $(\af_k,\af_{j+1})=0$ for $j+2<k\leq M$, we have
\begin{align*}
((i,\ldots,& M,M,\ldots,j+3)\shqqi(j+1))(j+2)\\
    =&((i,\ldots,j+2)\shqqi (j+1))(j+3,\ldots,M,M,\ldots,j+2)\\
    =&((i,\ldots,j+1)\shqqi (j+1))(j+2,\ldots,M,M,\ldots,j+2)\\
        &+\pi^{p(j+1)(p(i)+\cdots+p(j+2))}(q^{-(\af_i+\cdots+\af_{j+2},\af_{j+1})}-q^{(\af_i+\cdots+\af_{j+1},\af_{j+1})})\\
        &\times(i,\ldots,j+2,j+1,j+3\ldots,M,M,\ldots,j+2)
\end{align*}
But, using Lemma~ \ref{P:Type B signs} again, we have $(\af_i+\cdots+\af_{j+2},\af_{j+1})=0$, so
\begin{align*}
((i,\ldots,& M,M,\ldots,j+3)\shqqi(j+1))(j+2)\\
    =&((i,\ldots,j+1)\shqqi (j+1))(j+2,\ldots,M,M,\ldots,j+2)\\
    =&((i,\ldots,j)\shqqi(j+1))(j+1,\ldots,M,M,\ldots,j+2))\\
    &+\pi^{p(i,\ldots,j+1)p(j+1)}(q^{-(\af_j+\af_{j+1},\af_{j+1})}-q^{(\af_j+\af_{j+1},\af_{j+1})})\\
    &\times(i,\ldots,j+1,j+1,\ldots,M,M,\ldots,j+2)\\
    =&((i,\ldots,j-1)\shqqi(j+1))(j,\ldots,M,M,\ldots,j+2))\\
    &+\pi^{p(i,\ldots,j)p(j+1)}(q^{-(\af_j,\af_{j+1})}-q^{(\af_j,\af_{j+1})})\\
    &\times(i,\ldots,j+1,j+1,\ldots,M,M,\ldots,j+2)\\
    &+\pi^{p(i,\ldots,j+1)p(j+1)}(q^{-(\af_j+\af_{j+1},\af_{j+1})}-q^{(\af_j+\af_{j+1},\af_{j+1})})\\
    &\times(i,\ldots,j+1,j+1,\ldots,M,M,\ldots,j+2)\\
\end{align*}
Obviously, $(i,\ldots,j-1)\shqqi(j+1)=0$ since $(\af_i+\ldots+\af_{j-r},\af_{j+1})=0$ for $r\geq 1$.
To treat the last two summands above, note that either $p(j+1)=0$, or $a_{j+1,j+1}=0$. If $p(j+1)=0$, then
\begin{align*}
\pi^{p(i,\ldots,j+1)p(j+1)} &=\pi^{p(i,\ldots,j)p(j+1)},
 \\
(q^{-(\af_j,\af_{j+1})}-q^{(\af_j,\af_{j+1})})
 &=\pi(q^{-(\af_j+\af_{j+1},\af_{j+1})}-q^{(\af_j+\af_{j+1},\af_{j+1})}),
 \end{align*}
and hence \eqref{E:Type B Lyndon vector 2} holds. If $a_{j+1,j+1}=0$, then
\begin{align*}
\pi^{p(i,\ldots,j+1)p(j+1)} &=\pi^{p(i,\ldots,j)p(j+1)+1},
  \\
(q^{-(\af_j,\af_{j+1})}-q^{(\af_j,\af_{j+1})})
 &=(q^{-(\af_j+\af_{j+1},\af_{j+1})}-q^{(\af_j+\af_{j+1},\af_{j+1})}),
\end{align*}
and hence \eqref{E:Type B Lyndon vector 2} still holds. The proposition is proved.
\end{proof}

\begin{cor}
The following formulas hold for $1\leq i\leq j\leq M$:
\begin{enumerate}
\item
%$\E_\ui =
% \begin{cases}
%\varpi_B(\ui)(q^2-q^{-2})^{j-i}q^{M(|\ui|)}(i,\ldots,j), & \text{if } \ui=(i,\ldots,j),
%   \\
%\varpi_B(\ui)(q^2-q^{-2})^{2M-i-j}q^{M(|\ui|)}[2]_M^{-1}(i,\ldots,M,M,\ldots,j+1), &
%  \\
%\qquad\qquad\qquad\qquad\qquad\qquad   \text{if } \ui=(i,\ldots,M,M,\ldots,j+1). &
%\end{cases}$
%
For $\ui=(i,\ldots,j)$, the PBW root vector is
$\E_\ui=\varpi_B(\ui)(q^2-q^{-2})^{j-i}q^{-N(|\ui|)}(i,\ldots,j);$
For $\ui=(i,\ldots,M,M,\ldots,j+1)$, the PBW root vector is
$$\E_\ui=\varpi_B(\ui)(q^2-q^{-2})^{2M-i-j}q^{-N(|\ui|)}[2]_M^{-1}(i,\ldots,M,M,\ldots,j+1).$$

\item
$(\E_\ui,\E_\ui) =
 \begin{cases}
\varpi_B(\ui)(q^2-q^{-2})^{j-i}q^{-N(|\ui|)}, & \text{if }\; \ui=(i,\ldots,j),
   \\
\varpi_B(\ui)(q^2-q^{-2})^{2M-i-j}q^{-N(|\ui|)}[2]_N^{-2}, & \text{if }\; \ui=(i,\ldots,N,N,\ldots,j+1).
\end{cases}$

\item
$\E_\ui^*=
 \begin{cases}
(i,\ldots,j), & \text{if }\; \ui=(i,\ldots,j),
   \\
[2]_M(i,\ldots,M,M,\ldots,j+1), & \text{if }\; \ui=(i,\ldots,M,M,\ldots,j+1).
\end{cases}$
\end{enumerate}
\end{cor}

\begin{proof}
Parts (1) and (3) are proved in the same way as in the type $A$ case.

It remains to prove (2), the case $\ui=(i,\ldots,j)$ also being the same as in type $A(m,n)$.
Assume that $\ui=(i,\ldots,M,M)$. Then $\ui=\ui_1\ui_2$ is the co-standard factorization where $\ui_1=(i,\ldots,M)$ and $\ui_2=M$. We have
\begin{align*}
(\E_\ui,\E_\ui)&=[2]_M^{-1}(\E_\ui,M\shq\E_{\ui_1}-\pi^{p(M)}\E_\ui\shq M)\\
    &=\varpi_B(\ui)(q^2-q^{-2})^{M-i+1}q^{-N(|\ui|)}[2]_M^{-2}(\ui,M\shq\E_{\ui_1}-\pi^{p(M)}\E_{\ui_1}\shq M)\\
    &=\varpi_B(\ui)(q^2-q^{-2})^{M-i+1}q^{-N(|\ui|)}[2]_M^{-2}(M\otimes\ui_1,M\otimes\E_{\ui_1})\\
    &=\varpi_B(\ui)(q^2-q^{-2})^{M-i+1}q^{-N(|\ui|)}[2]_M^{-2}(\ui_1,\E_{\ui_1})\\
    &=s_{M-1,M}q^{-N(|\ui|)+N(|\ui_1|)}[2]_M^{-2}(\E_M,\E_M)(\E_{\ui_1},\E_{\ui_1})\\
    &=\varpi_B(\ui)(q^2-q^{-2})^{M-i+1}q^{-N(|\ui|)}[2]_M^{-2}.
\end{align*}
Finally, assume that $\ui=(i,\ldots,M,M,\ldots,j+1)$ with $i\leq j<M-1$. Then, $\ui=\ui_1\ui_2$ is the co-standard factorization,
where $\ui_1=(i,\ldots,M,M,\ldots,j+2)$ and $\ui_2=j+1$. Hence,
\begin{align*}
(\E_\ui,\E_\ui)&=(\E_{\ui},(j+1)\shq\E_{\ui_1}-\pi^{p(\ui_1)p(j+1)}q^{-(|\ui_1|,\af_j)}\E_{\ui_1}\shq(j+1))\\
    &=\varpi_B(\ui)(q^2-q^{-2})^{2M-i-j}q^{-N(|\ui|)}[2]_M^{-1}(\E_{j+1},\E_{j+1})({\ui_1},\E_{\ui_1})\\
    &=s_{j,j+1}(q^2-q^{-2})q^{-N(|\ui|)+N(|\ui_1|)}(\E_{\ui_1},\E_{\ui_1}).
\end{align*}
Therefore, (2) follows by induction.
\end{proof}

\subsection{Types $C(M)$ and $D(m,n+1)$, I}

We regard the type $C(M)$ as a limiting case of the type $D(m,n+1)$ with $m=1$ (and $M=n+2$), and will treat them simultaneously.
The Dynkin diagrams arise in two different shapes, with or without a branching node.
We separate the discussion into 2 parts, according to the shape of the Dynkin diagrams.
Here we consider a general Dynkin diagram without a branching node of the form
$$ 
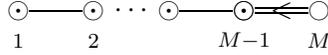
\xy
(-20,0)*{\odot};(-10,0)*{\odot}**\dir{-};(-5,0)*{\cdots};(0,0)*{\odot};(10,0)*{\odot}**\dir{-};
(10,0)*{\odot};(20,0)*{\fullmoon}**\dir{=};(15,0)*{<};
(-20,-4)*{\scriptstyle 1};(-10,-4)*{\scriptstyle 2};
(10,-4)*{\scriptstyle M-1};(20,-4)*{\scriptstyle M};
(0,-8)*{};(0,8)*{};
\endxy$$

The root system is given in \cite[Chapter 1]{CW12} and \cite[$\S3$]{ya}. We have the following properties regarding the system of signs.

\begin{lem}  \label{P:Type C signs}
\begin{enumerate}
\item We have $s_{M-1,M}=1$;
\item if $a_{ii}=0$, then $s_{i-1,i}=\pi s_{i,i+1}$;
\item if $a_{ii}\neq 0$, then $s_{i-1,i}=s_{i,i+1}=\pi s_{ii}$;
\item for any $k,l\in\I$ with $k\neq l$, we have
$
(\af_k,\af_l)\in\{(1+\dt_{kN}+\dt_{lN})s_{kl},0\}.
$
\end{enumerate}
\end{lem}

\begin{proof}
The lemma can be checked readily case-by-case by using the standard $\ep\dt$-notation for root systems and simple systems.
%As in the type $B$ case, we simply construct the simple system for $\UU$.
%To this end, let $V$ be a vector superspace with basis $\ep_1,\ldots,\ep_M$.
%We choose a system of signs $\{s_1,\ldots,s_{M-1}\}$ consisting of $m$ $s_i=+1$, $n$ $s_i=-1$, and we set $s_M=-1$.
%Define a bilinear form on $V$ by $(\ep_i,\ep_j)=s_i\dt_{ij}$.
%Then, $\af_i=\ep_i-\ep_{i+1}$ for $1\leq i<M$ and $\af_M=2\ep_M$ is the simple system for $\UU$.
%The lemma can now be checked case-by-case.
\end{proof}

The set of dominant Lyndon words are computed in the usual way.

\begin{prp}\label{P:Type C Lyndon Word}
The set of dominant Lyndon words is
\begin{align*}
\Lr^+=&\{(i,\ldots,j)\mid 1\le i\leq j \le M \}\cup\{(i,\ldots,M,\ldots,j+1)\mid 1\le i\leq j<M\}\\
    &\quad \cup\{(i,\ldots,M-1,i,\ldots,M)\mid 1\le i<M\mbox{ and }p(i,\ldots,M-1)=0\}.
\end{align*}
\end{prp}
Note the parity condition $p(i,\ldots,M-1)=0$ above
corresponds to the fact that there is no non-isotropic odd root for type $C$ and $D$.
Set
$$
\varpi_C(\ui)=
\begin{cases}
\varpi_A(\ui), &\mbox{if } \; \ui=(i,\ldots,j) \ \text{ for }\ 1\le i\leq j \le M,\\
\varpi_A(i,\ldots, j+1), &\mbox{if } \; \ui=(i,\ldots,M,\ldots,j+1) \ \text{ for }\  1\le i<j<M-1.
\end{cases}
$$

\begin{prp}
The Lyndon root vectors are given as follows:
\begin{enumerate}
\item for $\ui=(i,\ldots,j)$ with $j<M$,
$$\r_{\ui}=\pi^{P(|\ui|)}\pi^{j-i}\varpi_C(\ui)(q-q^{-1})^{j-i}(i,\ldots,j);$$
\item for $\ui=(i,\ldots,M)$,
$$\r_{\ui}=\pi^{P(|\ui|)}\pi^{M-i}\varpi_C(\ui)(q-q^{-1})^{M-i-1}(q^2-q^{-2})(i,\ldots,M);$$
\item for $\ui=(i,\ldots,M,\ldots,j+1)$,
$$\r_\ui=\pi^{P(|\ui|)}\pi^{i+j+1}\varpi_C(\ui)(q-q^{-1})^{2M-i-j-1}(q^2-q^{-2})(i,\ldots,M,\ldots,j+1).$$
\item for $\ui=(i,\ldots,M-1,i,\ldots,M)$,
$$\r_{\ui}=q^{-1}(q-q^{-1})^{2M-2i-1}(q^2-q^{-2})((i,\ldots,M-1)\shq(i,\ldots,M-1))M.$$
\end{enumerate}
\end{prp}

\begin{proof}
The proof of (1)-(3) are similar to the cases treated in types $A$ and $B$, and we omit the details.

We prove (4). To this end, note that $(i,\ldots,M-1)\shq(i,\ldots,M-1)\in\UU$ since $(i,\ldots,M-1)\in\UU$ by (1). Now, by Proposition~\ref{T:In UU}, we deduce that
$$x=((i,\ldots,M-1)\shq(i,\ldots,M-1))M\in\UU.$$
Evidently, $\max(x)=\ui$ and, therefore, $\max(x)=\max(\r_\ui)$ by Lemma~\ref{L:Lyndon max word}. Hence, we may express $x$ as
$$x=\sum_{\uj\leq\ui}\ld_\uj\r_\uj.$$
But, by Corollary~\ref{C:min std word}, $\ui$ is the smallest dominant word of its degree, so $x=\ld_\ui\r_\ui$.

We now compute the coefficient $\ld_\ui$. To this end, note that the co-standard factorization of $\ui$ is $\ui=\ui_1\ui_2$,
where $\ui_1=(i,\ldots,M-1)$ and $\ui_2=(i,\ldots,M)$. Hence, since $p(M)=p(i,\ldots,M-1)=0$ and $s_{M-1,M}=1$,
\begin{align*}
\r_\ui=&\r_{\ui_1}\shqqi\r_{\ui_2}\\
    =&-\pi^{P(|\ui_1|)+P(|\ui_2|)}\varpi_C(\ui_1)\varpi_C(\ui_2)(q-q^{-1})^{2M-2i-2}(q^2-q^{-2})(i,\ldots,M-1)\shqqi(i,\ldots,M)\\
    =&-(q-q^{-1})^{2M-2i-2}(q^2-q^{-2})((i,\ldots,M-2)\shqqi(i,\ldots,M))(M-1)\\
    &-(q-q^{-1})^{2M-2i-2}(q^2-q^{-2})(q^{-(\af_i+\cdots+\af_{M-1},\af_M)}((i,\ldots,M-1)\shq(i,\ldots,M-1))M\\
    &-q^{-(\af_i+\cdots+\af_{M-1},\af_M)}((i,\ldots,M-1)\shqi(i,\ldots,M-1))M).
\end{align*}
By the argument in the previous paragraph, $((i,\ldots,M-2)\shqqi(i,\ldots,M))(M-1)=0$. Therefore, applying the identity
$$(i,\ldots,M-1)\shqi(i,\ldots,M-1)=q^{(\af_i+\cdots+\af_{M-1},\af_i+\cdots+\af_{M-1})}(i,\ldots,M-1)\shq(i,\ldots,M-1),$$
and Lemma~\ref{P:Type C signs}, we see that
\begin{align*}
\r_\ui=&-(q-q^{-1})^{2M-2i-2}(q^2-q^{-2})\\
    &\times(q^{-(\af_{M-1},\af_M)}-q^{(\af_{M-1},\af_M)+(\af_i+\cdots+\af_{M-1},\af_i+\cdots+\af_{M-1})})((i,\ldots,M-1)\shq(i,\ldots,M-1))M\\
    =&-(q-q^{-1})^{2M-2i-2}(q^2-q^{-2})(q^{-2}-1)((i,\ldots,M-1)\shq(i,\ldots,M-1))M\\
    =&q^{-1}(q-q^{-1})^{2M-2i-1}(q^2-q^{-2})((i,\ldots,M-1)\shq(i,\ldots,M-1))M.
\end{align*}
This completes the proof.
\end{proof}

\begin{cor}
\begin{enumerate}
\item
The PBW root vectors are given as follows:

\begin{itemize}
\item[(a)]
for $\ui=(i,\ldots,j)$ with $j<M$,
$\E_{\ui}=\varpi_C(\ui)(q-q^{-1})^{j-i}q^{-N(|\ui|)}(i,\ldots,j);$

\item[(b)]
for $\ui=(i,\ldots,M)$,
$\E_{\ui}=\varpi_C(\ui)(q-q^{-1})^{M-i-1}(q^2-q^{-2})q^{-N(|\ui|)}(i,\ldots,M);$

\item[(c)]
for $\ui=(i,\ldots,M,\ldots,j+1)$, we have
$$\E_\ui=\varpi_C(\ui)(q-q^{-1})^{2M-i-j-1}(q^2-q^{-2})q^{-N(|\ui|)}(i,\ldots,M,\ldots,j+1);
$$

\item[(d)]
for $\ui=(i,\ldots,M-1,i,\ldots,M)$, we have
$$\E_{\ui}=\pi^{P(|\ui_1|)}q(q-q^{-1})^{2M-2i} q^{-N(|\ui|)}
  ((i,\ldots,M-1)\shq(i,\ldots,M-1))M,$$
where $\ui_1=(i,\ldots,M-1)$.
\end{itemize}

\item
The values of $(\E_{\ui},\E_\ui)$ are given by
$$
 \begin{cases}
\varpi_C(\ui)(q-q^{-1})^{j-i-\dt_{jM}}q^{-N(|\ui|)}, & \text{if }\; \ui=(i,\ldots,j)  \text{ for }  1\le i \le j  \le M,
   \\
\varpi_C(\ui)(q-q^{-1})^{2M-i-j-1}(q^2-q^{-2})q^{-N(|\ui|)},
  & \text{if }\; \ui=(i,\ldots,M,\ldots,j+1),
   \\
\pi^{P(|\ui_1|)}(q-q^{-1})^{2M-2i}q^{-N(|\ui|)},
  & \text{if }\; \ui=(i,\ldots,M-1,i,\ldots,M),
\end{cases}
$$
\noindent where $\ui_1=(i,\ldots,M-1)$.

\item
The dual PBW root vectors are given by
$$
\E_\ui^*=
 \begin{cases}
(i,\ldots,j), & \text{if }\; \ui=(i,\ldots,j)  \text{ for }  1\le i \le j \le M,
   \\
(i,\ldots,M,\ldots,j+1), & \text{if }\; \ui=(i,\ldots,M,\ldots,j+1),
   \\
q\, ((i,\ldots,M-1)\shq(i,\ldots,M-1))M, & \text{if }\; \ui=(i,\ldots,M-1,i,\ldots,M).
\end{cases}
$$
\end{enumerate}
\end{cor}

\begin{proof}
The formulas in (1) follow directly from the definitions.

We prove (2). Note that for
$$\ui\in\{(i,\ldots,j)\mid i\leq j\}\cup\{(i,\ldots,M,\ldots,j+1)\mid i\leq j<M\}$$
the computations are similar to those performed in types $A$ and $B$, and we omit the details. Therefore, assume that $\ui=(i,\ldots,M-1,i,\ldots,M)$. We have
\begin{align*}
\Dt(((i,\ldots,M-1)\shq(i,\ldots,M-1))M)=&(\Dt(i,\ldots,M-1)\shq\Dt(i,\ldots,M-1))(M\otimes 1)\\
    &+ 1\otimes((i,\ldots,M-1)\shq(i,\ldots,M-1))M
\end{align*}
and, therefore, $(\E_\ui,\E_\ui)$ is equal to
\begin{align*}
\pi^{P(|\ui_1|)} & q(q-q^{-1})^{2M-2i-1}(q^2-q^{-2})q^{-N(|\ui|)}[2]^{-1}(\E_\ui,((i,\ldots,M-1)\shq(i,\ldots,M-1))M)\\
    =&\pi^{P(|\ui_1|)}q(q-q^{-1})^{2M-2i-1}(q^2-q^{-2})q^{-N(|\ui|)}[2]^{-2}\\
    &\times(E_{\ui_2}\otimes\E_{\ui_1},(\Dt(i,\ldots,M-1)\shq\Dt(i,\ldots,M-1))(M\otimes 1))\\
    =&\pi^{P(|\ui_1|)}q(q-q^{-1})^{2M-2i-1}(q^2-q^{-2})q^{-N(|\ui|)}[2]^{-2}\\
    &\times(E_{\ui_2}\otimes\E_{\ui_1},(q^{-2}+1)(i,\ldots,M)\otimes(i,\ldots,M-1))\\
    =&\pi^{P(|\ui_1|)}q(q-q^{-1})q^{-N(|\ui|)+N(|\ui_1|)+N(|\ui_2|)}[2]^{-2}(q^{-2}+1)(\E_{\ui_1},\E_{\ui_1})(\E_{\ui_2},\E_{\ui_2})\\
    =&\pi^{P(\ui_1)}(q-q^{-1})^{2M-2i-1}(q^2-q^{-2})q^{-N(|\ui|)}[2]^{-1}\\
    =&\pi^{P(\ui_1)}(q-q^{-1})^{2M-2i}q^{-N(|\ui|)}.
\end{align*}
This proves (2).
Finally, (3) immediately follows from (2).
\end{proof}

%%%%%%%%%%%%%%%%%
\subsection{Type $C(M)$ and $D(m,n+1)$, II}

In this subsection, we consider the remaining simple systems of type $C(M)$ and $D(m,n+1)$,
which correspond to Dynkin diagrams with a branching node as follows:
$$
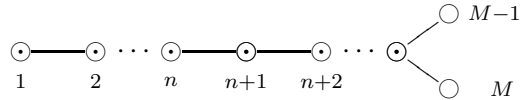
\xy
(-30,0)*{\odot};(-20,0)*{\odot}**\dir{-};(-15,0)*{\cdots};
(-10,0)*{\odot};(0,0)*{\odot}**\dir{-};(0,0)*{\odot};(10,0)*{\odot}**\dir{-};
(15,0)*{\cdots};(20,0)*{\odot};(27,5)*{\fullmoon}**\dir{-};(20,0)*{\odot};(27,-5)*{\fullmoon}**\dir{-};
(-30,-4)*{\scriptstyle 1};(-20,-4)*{\scriptstyle 2};(-10,-4)*{\scriptstyle n};(0,-4)*{\scriptstyle n+1};(10,-4)*{\scriptstyle n+2};
(33,5)*{\scriptstyle M-1};(34,-5)*{\scriptstyle M};
(0,-8)*{};(0,8)*{};
\endxy$$
and
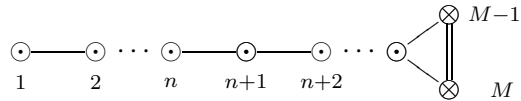
$$\xy
(-30,0)*{\odot};(-20,0)*{\odot}**\dir{-};(-15,0)*{\cdots};
(-10,0)*{\odot};(0,0)*{\odot}**\dir{-};(0,0)*{\odot};(10,0)*{\odot}**\dir{-};
(15,0)*{\cdots};(20,0)*{\odot};(27,5)*{\otimes}**\dir{-};(20,0)*{\odot};(27,-5)*{\otimes}**\dir{-};
(27,5)*{\otimes};(27,-5)*{\otimes}**\dir{=};
(-30,-4)*{\scriptstyle 1};(-20,-4)*{\scriptstyle 2};(-10,-4)*{\scriptstyle n};(0,-4)*{\scriptstyle n+1};(10,-4)*{\scriptstyle n+2};
(33,5)*{\scriptstyle M-1};(34,-5)*{\scriptstyle M};
(0,-8)*{};(0,8)*{};
\endxy$$

\begin{prp}\label{P:Type D Lyndon Words} The set of dominant Lyndon words is
\begin{align*}
\Lr^+=&\{(i,\ldots,j)\mid i\leq j\leq M-1\}\cup\{(i,\ldots,M-2,M)\mid i\leq M-2\}\\
    &\{(i,\ldots,M-2,M,M-1,\ldots,j+1)\mid i\leq j\leq M-2\}\\
    &\cup\{(i,\ldots,M-1,i,\ldots,M-2,M)\mid i<M-1, p(i,\ldots,M-1)=1\}.
\end{align*}
\end{prp}

Set
$$
\varpi_D(\ui)=
\begin{cases}\varpi_A(\ui)&\mbox{if }\ui=(i,\ldots,j), i\le j\leq M-1,\\
\varpi_A(i,\ldots,M-1)&\mbox{if }\ui=(i,\ldots,M-2,M),\\
\varpi_A(i,\ldots j+1)&\mbox{if }\ui=(i,\ldots,M-2,M,\ldots,j+1), i<j<M-1.
\end{cases}
$$

\begin{prp}
The Lyndon root vectors are given as follows:
\begin{enumerate}
\item For $\ui=(i,\ldots,j)$, $j\leq M-1$,
$$\r_{\ui}=\pi^{P(|\ui|)}\pi^{j-i}\varpi_D(\ui)(q-q^{-1})^{j-i}(i,\ldots,j);$$
\item for $\ui=(i,\ldots,M-2,M)$,
$$\r_{\ui}=\pi^{P(|\ui|)}\pi^{M-i-1}\varpi_D(\ui)(q-q^{-1})^{M-i-1}(i,\ldots,M-2,M);$$
\item for $\ui=(i,\ldots,M-2,M,M-1,\ldots,j+1)$,
\begin{align*}
\r_\ui=&\pi^{P(|\ui|)}\pi^{i+j}\varpi_D(\ui)(q-q^{-1})^{2M-i-j-2}((i,\ldots,M-1,M,M-2\ldots,j+1)\\
    &+(i,\ldots,M-2,M,M-1,\ldots,j+1)).
\end{align*}
\item for $\ui=(i,\ldots,M-1,i,\ldots,M-2,M)$,
$$\r_{\ui}=\pi(q-q^{-1})^{2M-2i-2}(q^2-q^{-2})((i,\ldots,M-2)\shq(i,\ldots,M-1))M.$$
\end{enumerate}
\end{prp}

\begin{proof}
Formulas (1)-(3) can be obtained in the same way as in previous types and we omit the details.

The proof of (4) is very similar to the long roots in type $C$ and we only outline the proof, leaving the details to the interested reader.
Indeed, let $\ui=(i,\ldots,M-1,i,\ldots,M-2,M)$ and let $\ui=\ui_1\ui_2$ be the co-standard factorization of $\ui$.
As in the type $C$ case, we deduce from
Proposition~\ref{T:In UU} that $x=((i,\ldots,M-2)\shq(i,\ldots,M-1))M\in\UU$.
Moreover, since $\max(x)=\ui$, it follows that $\r_\ui$ is proportional to $x$.
To compute the coefficient, note that $P(|\ui_1|)=P(|\ui_2|)$, $\varpi_D(\ui_1)=\varpi_D(\ui_2)$, so
\begin{align*}
\r_\ui=&(q-q^{-1})^{2M-2i-2}(i,\ldots,M-1)\shqqi(i,\ldots,M-2,M)\\
    =&(q-q^{-1})^{2M-2i-2}\pi^{p(M)}(q(i,\ldots,M-1)\shq(i,\ldots,M-2)\\
    &-q^{-1}(i,\ldots,M-1)\shqi(i,\ldots,M-2))M\\
    =&\pi(q-q^{-1})^{2M-2i-2}(q^2(i,\ldots,M-2)\shqi(i,\ldots,M-1)\\
    &-q^{-2}(i,\ldots,M-2)\shq(i,\ldots,M-1))M
\end{align*}
where we have used the fact that $p(M)+p(i,\ldots,M-2)=1=p(i,\ldots,M-1)$ to obtain the factor $\pi$ after the last equality.
Finally, the computation follows upon observing that
$$(i,\ldots,M-2)\shq(i,\ldots,M-1)=(i,\ldots,M-2)\shqi(i,\ldots,M-1).$$
This last statement can be proved as follows: first, we have $i\shq(i,\ldots,k)=i\shqi(i,\ldots,k)$ for any $k>i$ by induction on $k$, and,
for $i\leq j<k$, $(i,\ldots,j)\shq(i,\ldots,k)=(i,\ldots,j)\shqi(i,\ldots,k)$ by induction on $j$.
\end{proof}

\begin{cor}
\begin{enumerate}
\item The PBW root vectors are:
\begin{itemize}
\item[(a)]
for $\ui=(i,\ldots,j)$, $j\leq M-1$,
$\E_{\ui}=\varpi_D(\ui)(q-q^{-1})^{j-i}q^{-N(|\ui|)}(i,\ldots,j);$

\item[(b)]
for $\ui=(i,\ldots,M-2,M)$,
$\E_{\ui}=\varpi_D(\ui)(q-q^{-1})^{M-i-1}q^{-N(|\ui|)}(i,\ldots,M-2,M);$

\item[(c)]
for $\ui=(i,\ldots,M-2,M,M-1,\ldots,j+1)$,
\begin{align*}
\E_\ui=&\varpi_D(\ui)(q-q^{-1})^{2M-i-j-2}q^{-N(|\ui|)}((i,\ldots,M-1,M,M-2\ldots,j+1)\\
    &+(i,\ldots,M-2,M,M-1,\ldots,j+1));
\end{align*}

\item[(d)]
for $\ui=(i,\ldots,M-1,i,\ldots,M-2,M)$,
$$\E_{\ui}=(q-q^{-1})^{2M-2i-2}(q^2-q^{-2})[2]^{-1}q^{-N(|\ui|)}((i,\ldots,M-2)\shq(i,\ldots,M-1))M.$$
\end{itemize}

\item
The values of $(\E_{\ui},\E_{\ui})$ are given by
$$
 \begin{cases}
\varpi_D(\ui)(q-q^{-1})^{j-i}q^{-N(|\ui|)}, & \text{if }\; \ui=(i,\ldots,j) \ (j\leq M-1)
\\
 &\quad\; \text{ or }\ \ui=(i,\ldots,M-2,M),
   \\
\varpi_D(\ui)(q-q^{-1})^{2M-i-j-2}q^{-N(|\ui|)}, & \text{if }\; \ui=(i,\ldots,M-2,M,M-1,\ldots,j+1),
   \\
\frac{(q-q^{-1})^{2M-2i-1}q^{-N(|\ui|)}}{q+q^{-1}}, & \text{if }\; \ui=(i,\ldots,M-1,i,\ldots,M-2,M).
\end{cases}
$$

\item
The dual PBW root vectors are:
\begin{itemize}
\item[(a)]
For $\ui=(i,\ldots,j)$, $j\leq M-1$,
$\E_{\ui}^*=(i,\ldots,j);$

\item[(b)]
for $\ui=(i,\ldots,M-2,M)$,
$\E_{\ui}^*=(i,\ldots,M-2,M);$

\item[(c)]
for $\ui=(i,\ldots,M-2,M,M-1,\ldots,j+1)$,
\begin{align*}
\E_\ui^*=(i,\ldots,M-1,M,M-2\ldots,j+1)+(i,\ldots,M-2,M,M-1,\ldots,j+1));
\end{align*}

\item[(d)]
for $\ui=(i,\ldots,M-1,i,\ldots,M-2,M)$,
$$\E_{\ui}^*=(q+q^{-1})((i,\ldots,M-2)\shq(i,\ldots,M-1))M.$$
\end{itemize}
\end{enumerate}
\end{cor}

%%%%%%%%%%%%%%%%%%%
\subsection{Type $F(3|1)$}
Associated to the distinguished diagram
$$
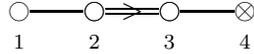
\xy (0,8)*{};
(-15,0)*{\fullmoon};(-5,0)*{\fullmoon}**\dir{-};
(-5,0)*{\fullmoon};(5,0)*{\fullmoon}**\dir{=};(0,0)*{>};(5,0)*{\fullmoon};(15,0)*{\otimes}**\dir{-};
(-15,-4)*{\scriptstyle 1};(-5,-4)*{\scriptstyle 2};(5,-4)*{\scriptstyle 3};(15,-4)*{\scriptstyle 4};
(0,-8)*{};
\endxy$$
we have the following table of dominant Lyndon words.

\begin{center}
\begin{tabular}{|c|c|}\hline
Height& Dominant Lyndon Words\\
\hline\hline
1&$1,2,3,4$\\
2&$(12),(23),(34)$\\
3&$(123),(233),(234)$\\
4&$(1233),(1234),(2343)$\\
5&$(12332),(12343)$\\
6&$(123432)$\\
7&$(1234323)$\\
8&$(12343234)$\\\hline
\end{tabular}
\end{center}

\subsection{Type $G(3)$}
Associated to the distinguished diagram

$$ 
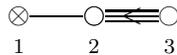
\xy(0,5)*{};
{\ar@3{-}(0,0)*{\fullmoon};(10,0)*{\fullmoon}};(5,0)*{<};(-10,0)*{\otimes};(0,0)*{\fullmoon}**\dir{-};
(-10,-4)*{\scriptstyle 1};(0,-4)*{\scriptstyle 2};(10,-4)*{\scriptstyle 3};
(0,-8)*{};(0,8)*{};
\endxy$$
we have the following table of dominant Lyndon words.

\begin{center}
\begin{tabular}{|c|c|}\hline
Height& Dominant Lyndon Words\\
\hline\hline
1&$1,2,3$\\
2&$(12),(23)$\\
3&$(123),(223)$\\
4&$(1232),(2223)$\\
5&$(12322),(22323)$\\
6&$(123223)$\\
7&$(1232233)$\\\hline
\end{tabular}
\end{center}
\vspace{.4cm}

%%%%%%%%%%%%%%%%%%%%%
%%%%%%%%%%%%%%%%%%%%%
\section{Canonical Bases}\label{S:Canonical Bases}

In this section, we shall formulate and construct the canonical basis of type $A(m,0)$, $B(0,n+1)$, and $C(n+1)$ for the standard simple system.
Table~2 below compiles a list of standard simple systems for Lie superalgebras of basic type, with $D(2|1;\alpha)$ omitted.

 \begin{center}
  \vspace{.5cm}
 {Table 2: Dynkin diagrams for standard simple systems}
  \vspace{.5cm}
%\begin{tabular}
%\begin{table}[ht]
%\label{T:Dynkin Standard}
% \caption{Dynkin diagrams for standard simple systems}
%\begin{center}
\begin{tabular}{|c|c|}
\hline
$A(m,n)$&
$$ \xy
(-30,0)*{\fullmoon};(-20,0)*{\fullmoon}**\dir{-};(-15,0)*{\cdots};(-10,0)*{\fullmoon};(0,0)*{\otimes}**\dir{-};
(0,0)*{\otimes};(10,0)*{\fullmoon}**\dir{-};
(15,0)*{\cdots};(20,0)*{\fullmoon};(30,0)*{\fullmoon}**\dir{-};
(-30,-4)*{\scriptstyle \overline{n}};(-20,-4)*{\scriptstyle \overline{n-1}};(-10,-4)*{\scriptstyle \overline{1}};(0,-4)*{\scriptstyle 0};(10,-4)*{\scriptstyle 1};
(20,-4)*{\scriptstyle m-1};(30,-4)*{\scriptstyle m};
(0,-8)*{};(0,8)*{};
\endxy$$
\\\hline
$B(m,n+1)$&
$$ \xy
(-30,0)*{\fullmoon};(-20,0)*{\fullmoon}**\dir{-};(-15,0)*{\cdots};(-10,0)*{\fullmoon};(0,0)*{\otimes}**\dir{-};
(0,0)*{\otimes};(10,0)*{\fullmoon}**\dir{-};
(15,0)*{\cdots};(20,0)*{\fullmoon};(30,0)*{\fullmoon}**\dir{=};(25,0)*{>};
(-30,-4)*{\scriptstyle \overline{n}};(-20,-4)*{\scriptstyle \overline{n-1}};(-10,-4)*{\scriptstyle \overline{1}};(0,-4)*{\scriptstyle 0};(10,-4)*{\scriptstyle 1};
(20,-4)*{\scriptstyle m-1};(30,-4)*{\scriptstyle m};
(0,-8)*{};(0,8)*{};
\endxy$$
\\\hline
$B(0,n+1)$&
$$ \xy
(-20,0)*{\fullmoon};(-10,0)*{\fullmoon}**\dir{-};(-5,0)*{\cdots};(0,0)*{\fullmoon};(10,0)*{\fullmoon}**\dir{-};
(10,0)*{\fullmoon};(20,0)*{\newmoon}**\dir{=};(15,0)*{>};
(-20,-4)*{\scriptstyle \overline{n}};(-10,-4)*{\scriptstyle \overline{n-1}};
(10,-4)*{\scriptstyle \overline{1}};(20,-4)*{\scriptstyle 0};
(0,-8)*{};(0,8)*{};
\endxy$$
\\\hline
$C(n+1)$&
$$ \xy
(-20,0)*{\otimes};(-10,0)*{\fullmoon}**\dir{-};(-5,0)*{\cdots};(0,0)*{\fullmoon};(10,0)*{\fullmoon}**\dir{-};
(10,0)*{\fullmoon};(20,0)*{\fullmoon}**\dir{=};(15,0)*{<};
(-20,-4)*{\scriptstyle \overline{0}};(-10,-4)*{\scriptstyle \overline{1}};
(10,-4)*{\scriptstyle \overline{n-1}};(20,-4)*{\scriptstyle \overline{n}};
(0,-8)*{};(0,8)*{};
\endxy$$
\\\hline
$D(m,n+1)$&
$$\xy
(-30,0)*{\fullmoon};(-20,0)*{\fullmoon}**\dir{-};(-15,0)*{\cdots};
(-10,0)*{\fullmoon};(0,0)*{\otimes}**\dir{-};(0,0)*{\otimes};(10,0)*{\fullmoon}**\dir{-};
(15,0)*{\cdots};(20,0)*{\fullmoon};(27,5)*{\fullmoon}**\dir{-};(20,0)*{\fullmoon};(27,-5)*{\fullmoon}**\dir{-};
(-30,-4)*{\scriptstyle \overline{n}};(-20,-4)*{\scriptstyle \overline{n-1}};(-10,-4)*{\scriptstyle \overline{1}};(0,-4)*{\scriptstyle 0};(10,-4)*{\scriptstyle 1};
(33,5)*{\scriptstyle m-1};(33,-5)*{\scriptstyle m};
(0,-8)*{};(0,8)*{};
\endxy$$
\\\hline
$F(3|1)$&
$$\xy (0,8)*{};
(-15,0)*{\otimes};(-5,0)*{\fullmoon}**\dir{-};
(-5,0)*{\fullmoon};(5,0)*{\fullmoon}**\dir{=};(0,0)*{<};(5,0)*{\fullmoon};(15,0)*{\fullmoon}**\dir{-};
(-15,-4)*{\scriptstyle 0};(-5,-4)*{\scriptstyle \overline{1}};(5,-4)*{\scriptstyle \overline{2}};(15,-4)*{\scriptstyle \overline{3}};
(0,-8)*{};
\endxy$$
\\\hline
$G(3)$&
$$ \xy(0,5)*{};
{\ar@3{-}(0,0)*{\fullmoon};(10,0)*{\fullmoon}};(5,0)*{<};(-10,0)*{\otimes};(0,0)*{\fullmoon}**\dir{-};
(-10,-4)*{\scriptstyle 0};(0,-4)*{\scriptstyle \overline{1}};(10,-4)*{\scriptstyle \overline{2}};
(0,-8)*{};(0,8)*{};
\endxy$$
\\\hline
\end{tabular}
\end{center}
%\end{table}

%%%%%%%%%%%%%%%%%%%%%
\subsection{Integral Forms}

We start with some general discussions of root systems of basic type in order to define
suitable integral forms of $U_q$.

We will restrict our attention to the standard simple systems in Table~2, and fix the ordered set
$$(\I,\leq)=\{\bar{n}<\cdots<\bar{1}<0<1<\cdots<m\}.$$
Following Lusztig, we call $i\in\I$ and $\af_i\in\Pi$ \textbf{special} if  $c_i\leq 1$ in the expansions of every root $\bt$
in $\widetilde{\Phi}^+$ in terms of $\Pi$, $\bt=\sum_{j\in\I}c_j\af_j.$
%We say that $i\in\I$ is \textbf{semispecial} if $c_i\leq 1$ with the exception of a unique positive root where $c_i=2$.
We will call a Dynkin diagram (or the corresponding $\UU$)
appearing in Table~2  \textbf{special} if any $i\in\I_\iso$(which is unique if it exists) is special.
Note that we take into account the entire (positive) root system $\widetilde{\Phi}^+$ as opposed to the reduced one.
By inspection we have the following.

\begin{prp}\label{P:super-special}
The Dynkin diagrams in Table~2  are special if and only if they are of type $A(m,n)$, $B(0,n+1)$,  and $C(n+1)$.
%The semispecial diagrams are of type $B(m,1)$, $D(m,1)$, $F(3|1)$, and $G(3)$.
\end{prp}

Let $\A=\Z[q,q^{-1}]$ and define $U_\A$ to be the $\A$-subalgebra of $\Uq$ generated by $e_i$ ($i\in\I_\iso$)
and the divided powers $e_i^{(k)}=e_i^k/[k]_i!$ ($i\in\I_{\zero}\sqcup\I_{\niso}$, $k\geq 1$). Set
$$U_\A^*=\{u\in\Uq\mid (u,v)\in\A\mbox{ for all }v\in U_\A\}.$$
Denote by $\W'$ the subset of words in $\ui\in\W$ of the form $\ui=i_1^{n_1}\cdots i_d^{n_d}$,
where $i_k\neq i_{k+1}$ for all $1\leq k<d$
 and $n_l \in \{0,1\}$ whenever $i_l \in \I_\iso$. For such $\ui\in \W'$, we set
$$\varsigma_\ui=[n_1]_{i_1}!\cdots[n_d]_{i_d}!$$
and write $e_\ui=e_{i_1}^{n_1}\cdots e_{i_d}^{n_d}$.
Then, $\varsigma_\ui^{-1} e_\ui$ is a product of divided powers.
Consider the free $\A$-module $\F_\A=\bigoplus_{\ui\in\W' }\A\varsigma_\ui \ui$ and define
\begin{align}\label{E:FAstar}
\UU_\A^*=\F_\A\cap\UU.
\end{align}

We have the following analogue of \cite[Lemma 8]{lec} with an entirely similar proof.

\begin{lem}\label{L:UUAstar} We have $\UU_\A^*=\Psi(U_\A^*)$.
\end{lem}

\begin{proof}
%This is proved exactly as for \cite[Lemma 8]{lec}.
Any $u\in\Uq$ belongs to $U_\A^*$ if and only if $(u,\varsigma_{\ui}^{-1}e_\ui)\in\A$ for all $\ui\in\W'$.
This holds if and only if $\Psi(u)$ is a linear combination of elements $\varsigma_\ui\ui$ for $\ui\in\W'$,
which is true if and only if $\Psi(u)\in\F_\A$.
\end{proof}

\begin{cor}\label{C:UAstar is a subalgebra}
The free $\A$-module $U_\A^*$ is an $\A$-subalgebra of $\Uq$.
\end{cor}

\begin{proof}
It is clear from the definitions that $(\F_\A,\shq)$ is an $\A$-subalgebra of $(\F,\shq)$ and, therefore, so is $\UU_\A^*$.
By Lemma \ref{L:UUAstar}, $U_\A^*$ is an $\A$-subalgebra of $\Uq$.
\end{proof}

Let $\UU_\PBW$ be the $\A$-lattice spanned by the PBW basis $\{\E_\ui\mid\ui\in\W^+\}$, and $\UU_\PBW^*$ the
$\A$-lattice spanned by  the dual PBW basis $\{\E_\ui^*\mid\ui\in\W^+\}$ in \eqref{E:Dual PBW}.

\begin{prp}\label{P:special UA basis}
Assume that $\UU$ is special. Then $\UU_\PBW^*=\UU_\A^*$ and $\UU_\PBW=\UU_\A$.
%\begin{enumerate}
%\item $\UU_\PBW^*=\UU_\A^*$, and
%\item $\UU_\PBW=\UU_\A$.
%\end{enumerate}
\end{prp}

\begin{proof}
The two identities are equivalent, and we shall prove that $\UU_\PBW^*=\UU_\A^*$.
To this end, note that by the computations in Section~\ref{S:Computations}, $\E_\ui^*\in\UU_\A^*$ for all $\ui\in\L^+$.
By Corollary~\ref{C:UAstar is a subalgebra}, it follows that
$$\UU_\PBW^*\subset\UU_\A^*.$$
We will now prove that equality holds when $\UU$ is special. To this end, suppose that
$$\sum_{\ui\in\W^+}\ld_\ui\E_\ui^*\in\UU_\A^*.$$
We will prove that all $\ld_\ui\in\A$ by induction on $|\{\ui\in\W^+\mid \ld_\ui\neq0\}|$.

First, suppose that $\ld_\ui\E_\ui^*\in\UU_\A^*$.
Suppose that $\ui=(i_1^{a_1},\ldots, i_d^{a_d})\in\L^+_\zero\sqcup\L^+_\niso$, $i_r\neq i_{r+1}$.
Note that the coefficient of $\ui$ in $\E_\ui$ is $\varsigma_\ui$ (except for the long roots in type $C$,
where we instead consider the word $\ui'=(i,i,i+1,i+1,\ldots,M-1,M-1,M)$ whose coefficient is $\varsigma_{\ui'}$). For $n\geq 1$, let
$$\ui^{(n)}=(i_1^{na_1},\ldots,i_d^{na_d}).$$
Since the diagram for $\UU$ is special, the coefficient of $\ui^{(n)}$ in $(\E_\ui^*)^{\shq n}$ is nonzero, and (up to a power of $q$) equals
$$\varsigma_\ui^n\left[{na_1\atop a_1,\ldots,a_1}\right]_{i_1}\cdots\left[{na_d\atop a_d,\ldots,a_d}\right]_{i_d}=\varsigma_{\ui^{(n)}}$$
where, for $r\geq 1$,
$$\left[{na_r\atop a_r,\ldots,a_r}\right]_{i_r}=\frac{[na_r]_{i_r}!}{([a_r]_{i_r}!)^n}$$
is the quantum multinomial coefficient.

Now, if $\ui\in\W^+$ and $\ui=\ui_1^{n_1}\cdots\ui_r^{n_r}$, $\ui_1>\cdots>\ui_r$, $\ui_s=(i_{s1}^{a_{s1}},\ldots,i_{sd_s}^{a_{sd_s}})\in\L^+$,
then the coefficient of $\widetilde{\ui}:=\ui_1^{(n_1)}\cdots\ui_r^{(n_r)}$ in $\E_\ui^*$ is nonzero
(again, because the diagram is special) and (up to a power of $q$) equals
$$\prod_{s=1}^r\varsigma_{\ui_s}^{n_s}\prod_{t=1}^{d_s}\left[{n_sa_{st}\atop a_{st},\ldots,a_{st}}\right]_{i_{st}}
=\prod_{s=1}^r\varsigma_{\ui_s^{(n_s)}}=\varsigma_{\widetilde{\ui}}.$$
(Above, we make the appropriate adjustments in type $C$ as in the last paragraph).
Hence, if $\ld_\ui\E_\ui^*\in\UU_\A^*$, then $\ld_\ui\varsigma_{\widetilde{\ui}}\in\A\varsigma_{\widetilde{\ui}}$ which forces $\ld_\ui\in\A$ as required.

We now proceed to the inductive step. Let $\uj=\max\{\ui\mid \ld_\ui\neq 0\}$. Then, the coefficient of $\widetilde{\uj}$ in $\E_\uj^*$
(making the appropriate adjustments in type $C$) is $\varsigma_{\widetilde{\uj}}$.
Moreover, $\widetilde{\uj}$ does not occur in $\E_\ui^*$ for $\ui<\uj$. It follows that $\ld_\uj\in\A$, and induction applies to
$$\sum_{\ui\neq\uj}\ld_\ui\E_\ui^*=\left(\sum_{\ui}\ld_\ui\E_\ui^*\right)-\ld_\uj\E_\uj^*\in\UU_\A^*.$$
This completes the proof.
\end{proof}

\begin{exa} It is not true that $\UU_\PBW=\UU_\A$ for non-special standard Dynkin diagrams in general.
Indeed, consider type $B(1,1)$:
$$
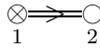
\xy
(0,0)*{\otimes};(10,0)*{\fullmoon}**\dir{=};(0,-3)*{\scriptstyle 1};(10,-3)*{\scriptstyle 2};(5,0)*{>};
\endxy
$$
The root  %$\af_1$ is semispecial with
$\bt=\af_1+\af_2$ non-isotropic. We have $\ui(\bt)=(12)$, $\E_{(12)}^*=(12)$, and
$$\E_{(1212)}^*=q^{-1}\E_{(12)}^*\shq\E_{(12)}^*=(\pi q+q^{-1})(1212).$$
In particular, $\frac{1}{\{2\}}\E_{(1212)}\in\UU_\A^*$ showing that $\UU_\PBW^*\neq\UU_\A^*$.
\end{exa}

%%%%%%%%%%%%%%%%%%%%
\subsection{Pseudo-canonical  and Canonical Bases}

\begin{lem}\label{L:Bar PBW}
For $\ui\in\Wr^+$, write
\begin{align}\label{E:Bar PBW}
\overline{\E_\ui}=\sum_{\uj\in\Wr^+}a_{\ui\uj}\E_\uj, \qquad \text{ for } \; a_{\ui\uj} \in \Q(q).
\end{align}
Then, $a_{\ui\ui}=1$ for all $\ui\in\Wr^+$ and $a_{\ui\uj}=0$ if $\ui>\uj$.
\end{lem}

\begin{proof}
This proof is identical to that of \cite[Lemma 37]{lec}. By Propositions~\ref{P:Bar Invariant Renormalization} and \ref{P:PBW Triangularity} we have
$$\ep_{\tau(\ui)}=\sum_{\uj\geq\ui}\beta_{\ui\uj}\E_\uj$$
with $\overline{\bt_{\ui\ui}}=\bt_{\ui\ui}=\kp_{\ui}$. As $\overline{\ep_{\tau(\ui)}}=\ep_{\tau(\ui)}$, substituting \eqref{E:Bar PBW} into the equation above yields
$$a_{\ui\uj}=\sum_{\ui\leq\uk\leq\uj}\overline{\af_{\ui\uk}}\,\bt_{\uk\uj}.$$
Therefore, $a_{\ui\uj}=0$ if $\ui>\uj$ and $a_{\ui\ui}=\overline{\af_{\ui\ui}}\bt_{\ui\ui}=\kp_\ui^{-1}\kp_\ui=1$ by Proposition~\ref{P:Bar Invariant Renormalization}.
\end{proof}

\begin{lem}\label{L:Integral coefficients}
Suppose that $\UU$ is special. Then, the coefficients $a_{\ui\uj}$ in  \eqref{E:Bar PBW} belong to $\A$.
\end{lem}

\begin{proof}
This is immediate since $\UU_\PBW=\UU_\A$ by Proposition~\ref{P:special UA basis} and $\UU_\A$ is clearly bar invariant.
\end{proof}

It is well known that Lemmas \ref{L:Bar PBW} and \ref{L:Integral coefficients} imply the existence of a unique basis of the form
\begin{align}\label{E:Weak Canonical}
\b_\ui=\E_\ui+\sum_{\uj>\ui}\theta_{\ui\uj}\E_\uj
\end{align}
such that $\theta_{\ui\uj}\in q\Z[q]$ and $\overline{\b_\ui}=\b_{\ui}$. We call the basis $\{\b_\ui\mid\ui\in\Wr^+\}$ a \textbf{pseudo-canonical basis} for
$\UU_\A$ or for $\UU$.

A pseudo-canonical basis will be called a \textbf{canonical basis}
if it is \textbf{almost orthogonal} in the sense that there exists $\epsilon\in\{1,-1\}$ such that, for all $\ui,\uj\in\Wr^+$,
\begin{enumerate}
\item $(\b_\ui,\b_\uj)\in\Z[q^{\epsilon}]$, and
\item $(b_\ui,\b_\uj)=\pi^\te \dt_{\ui\uj}$ (mod $q^\epsilon$) for some $\te\in\{0,1\}$.
\end{enumerate}

\begin{thm}\label{T:canonical basis}
When $\UU$ is special it admits a pseudo-canonical basis.
In types $A(m,0)$, $A(0,n)$, $B(0,n+1)$ and $C(n+1)$ the pseudo-canonical basis is canonical.
\end{thm}

\begin{proof} It has already been explained that $\UU$ has a pseudo-canonical basis when it is special.
By the computations in Section~\ref{S:Computations} to verify that in types $A(m,0)$, $A(0,n)$,
$B(0,n+1)$ and $C(n+1)$ one checks
that  the PBW basis is almost orthogonal. Hence, the pseudo-canonical basis is canonical.
\end{proof}

\begin{rmk}
The constructions in this paper (see Lemma~\ref{L:Lyndon max word}, and Theorems~\ref{T:max word}, \ref{T:OrthogonalPBW}, and \ref{T:canonical basis})
work equally well for $U_q$ associated to
semisimple Lie algebras, providing a new self-contained approach to the canonical basis in the non-super setting.
\end{rmk}

\begin{rmk}
For type $B(0,n)$, a canonical (sign) $\pi$-basis for $U_q$ was constructed in \cite{CHW} via a crystal basis approach.
The canonical basis $\bold{B}$  for $U_q$ of type $B(0,n)$ constructed in this paper is an honest basis which depends on the choice of
an ordering of $I$, We expect that the associated $\pi$-basis $\bold{B} \cup \pi \bold{B}$ will be independent of
the orderings and coincides with the one constructed in \cite{CHW}.
\end{rmk}

Given a (pseudo-)canonical basis $\mathrm{B}=\{\b_\ui\}_{\ui\in\Wr^+}$,
let $\mathrm{B}^*=\{\b^*_\ui\}_{\ui\in\Wr^+}$ be the dual (pseudo) canonical basis satisfying
$(\b^*_\ui,\b_\uj)=\dt_{\ui\uj}.$
Then, as in \cite[Proposition 39, Theorem 40]{lec} we have the following.

\begin{thm} The vector $\b^*_\ui$ is characterized by the following two properties:
\begin{enumerate}
\item $\b^*_\ui-\E^*_\ui$ is a linear combination of vectors $\E_{\uj}^*$, $\uj<\ui$, with coefficients in $q\Z[q]$;
\item The coefficients of $\b^*_\ui$ in the word basis $\W$ of $\F$ are symmetric in $q$ and $q^{-1}$.
\end{enumerate}
In particular, we have $\max(\b^*_\ui)=\ui$ for all $\ui\in\Wr^+$, and $\b^*_\ui=\E^*_\ui$ if $\ui\in\Lr^+$.
\end{thm}

%%%%%%%%%%%%%%%%%%%%%%%
%%%%%%%%%%%%%%%%%%%%%%%
\section{Canonical Bases in the $\gl(2|1)$ Case}
 \label{S:gl(21)}

%%%%%%%%
\subsection{Canonical basis for $U_q^+(\gl(2|1))$}

We now compute canonical bases arising from quantum  $\gl(2|1)$ and its modules.
The root datum in this case is given by
$$ \xy
(-30,0)*{\fullmoon};(-20,0)*{\otimes,}**\dir{-};
(-30,-4)*{\scriptstyle 1};(-20,-4)*{\scriptstyle 2};
(0,-8)*{};(0,8)*{};
\endxy\qquad \Phi^+=\set{\alpha_1,\alpha_2,\alpha_1+\alpha_2}.
%\qquad A=\begin{bmatrix}2&-1\\-1&0\end{bmatrix},
$$

The algebra $\Uq=\Uq^+(\gl(2|1))$ is generated by $\Psi^{-1}(\E_1), \Psi^{-1}(\E_2)$,
with $\Psi^{-1}(\E_2)$ odd. Abusing notation slightly, we will identify these elements with $\E_1$ and $\E_2$, respectively.
We note that by \eqref{E:PBW Definition 1},
$$\E_{(12)}:= \E_2\E_1 -q \E_1\E_2.
$$
Then since $\E_2^2=0$, we have
%\begin{align*}
$$
\E_{(12)}^2 =0,
\quad
\E_2 \E_{(12)} =-q \E_{(12)}\E_2,
  \quad
\E_1\E_{(12)}  =q \E_{(12)}\E_1,
\quad
 \E_{(12)}\E_2 =\E_2 \E_{1}\E_2.
$$
%\end{align*}
Moreover, we can verify that, for $r, s\ge 1$,
\begin{align}
% \E_1\E_2\E_1\E_2 &=\E_2\E_1\E_2\E_1,
% \\
\E_1^{(r)}\E_2\E_1\E_2 &= \E_2\E_1\E_2 \E_1^{(r)}, \label{eq:EEEE}
 \\
 \E_2\E_1^{(r)} &=q^{r}  \E_1^{(r)}\E_2  + \E_1^{(r-1)}\E_{(12)} , \notag
 \\
\E_1^{(r)}\E_2\E_1^{(s)}
&=\bbinom{r+s-1}{s}\E_1^{(r+s)}\E_2+\bbinom{r+s-1}{r}\E_2\E_1^{(r+s)}.
 \label{eq:EE13}
\end{align}
The formula \eqref{eq:EE13} is the same as for quantum $\mathfrak{sl}(3)$ \cite{lu1}. One checks that
\begin{equation}
 \label{eq:4E}
 \E_2\E_1^{(r+1)} \E_2  = \E_1^{(r)}\E_{(12)}\E_2
  =\E_1^{(r)}\E_2\E_1\E_2 =\E_2\E_1\E_2\E_1^{(r)}.
 \end{equation}

Now note that the Lyndon words are $2>12>1$, and so
relative to this ordering we see that the PBW basis
for $\Uq$ is
$$
\{\E_1^{(r)} \E_{(12)}^b \E_2^a \mid 0\le a,b \le 1, r \ge 0\}.
$$
They span a $\Z[q]$-lattice $\mc L$
of $\Uq^+$.

The following has appeared in \cite{kh}, who works with quantum $\gl(1|2)$ instead.

\begin{prp}    \label{prp:CBU+}
$\Uq^+(\mathfrak{gl}(2|1))$ admits the following canonical basis:
$$
\E_1^{(r)}, \quad \E_1^{(r)}\E_2, \quad \E_2\E_1^{(r+1)}, \quad
\E_2\E_1^{(r+1)} \E_2  \qquad (\forall
r\ge 0).
$$
\end{prp}

\begin{proof}
Now the first two elements $ \E_1^{(r)}, \E_1^{(r)}\E_2$ are already
bar-invariant PBW basis elements, whence pseudo-canonical basis elements.
Similarly, the element $\E_2\E_1^{(r+1)} \E_2$ is
bar-invariant and also a PBW element by \eqref{eq:4E},
whence a pseudo-canonical basis element.
One writes the remaining PBW  elements as $\E_{2} \E_1^{(r+1)}
=q^{r+1}\E_1^{(r+1)}\E_2 +\E_1^{(r)} \E_{(12)}$, for $r \ge 0$. Hence
$\E_2\E_1^{(r+1)}$ is a bar-invariant element, which equals a PBW
element modulo $q\mc L$, whence a pseudo-canonical basis element.

Clearly the elements as in the proposition form a basis of the
lattice $\mc L$, by comparing to the PBW basis,
hence this is the promised pseudo-canonical basis. On the other hand,
computing the norms of these elements proves that they are actually
a canonical basis.
\end{proof}

\begin{rmk}
In contrast to Proposition~\ref{prp:CBU+},
$\E_2\E_1\E_2$ is not a canonical basis element for the positive half of quantum
$\mf{sl}(3)$.
\end{rmk}

\begin{rmk}
When multiplying any canonical basis element for $\Uq^+(\gl(2|1))$
with $\E_1^{(s)}$ or $\E_2$ (either on the left or on the right) and
then expanding as a linear combination of the canonical basis, the
coefficients are always in $\Z_{\ge 0} [q,q^{-1}]$.
\end{rmk}

%\begin{proof}
%We shall check the non-obvious cases.
%First, that this holds for left or right multiplication by $\E_1^{(r)}$ is obvious in all cases
%but when the resulting element is of the form $\E_1^{(r)}\E_2\E_1^{(s)}$,
%and we compute that
%\begin{align*}
%\E_1^{(r)}\E_2\E_1^{(s)}
%&=q^{s}\bbinom{r+s}{s}\E_1^{(r+s)}\E_2+\bbinom{r+s-1}{r}\E_1^{(r+s-1)}\E_{(12)}\\
%&=\parens{q^{s}\bbinom{r+s}{r}-q^{r+s}\bbinom{r+s-1}{r}}\E_1^{(r+s)}\E_2+\bbinom{r+s-1}{r}\E_2\E_1^{(r+s)}\\
%&=\bbinom{r+s-1}{s}\E_1^{(r+s)}\E_2+\bbinom{r+s-1}{r}\E_2\E_1^{(r+s)}.
%\end{align*}
%
%Second, the lemma obviously holds for multiplication by $E_2$ on the left or right
%in every case but $\E_2\E_1^{(r)}\E_2$, in which case we compute that
%$\E_2\E_1^{(r)}\E_2=\E_1^{(r-1)}\E_{(12)}\E_2.$
%\end{proof}

Denote by
$$
{\bf B} =\{\rF_1^{(r)}, \quad \rF_2\rF_1^{(r)}, \quad \rF_1^{(r+1)} \rF_2, \quad
\rF_2\rF_1^{(r+1)} \rF_2  \; (\forall r\ge 0) \}
$$
the canonical basis of $\Uq^-$, which consists of the images of the elements
in Proposition~\ref{prp:CBU+} under the anti-isomorphism $\Uq^+\rightarrow \Uq^-$
defined by $\E_i\mapsto \rF_i$. Below we often use the identifications
$\rF_2\rF_1^{(r+1)} \rF_2 = \rF_2 \rF_{(12)} \rF_1^{(r)}$
and $\rF_{(12)}=\rF_1\rF_2-q\rF_2\rF_1$.

%%%%%%%%%%%%%%%%%
\subsection{Canonical basis for Kac modules}

The subalgebra $\Uq^0$ of $\Uq$ is generated by $K_1=q^{e_{11}}, K_2=q^{e_{22}}, K_3=q^{e_{33}}$.
Let $\Uq^{2,1}$ be the subalgebra of $\Uq$ generated by
$U_q^0$, $\E_1$,  and $\rF_1$, and
let $P_q$ be the subalgebra generated by $\Uq^{2,1}$
and $\E_2$.
Denote by $\{\dt_1, \dt_2, \ep_1\}$ the dual basis for $\{e_{11}, e_{22}, e_{33}\}$.
Let $\mu =a \delta_1+b \delta_2 + c\varep_1$, with  $a-b \in \Z_{\ge
0}$. Set $L^0(\mu)$ to be the simple $\Uq^{2,1}$-module
of highest weight $\mu$. Then $L^0(\mu)$ is a $P_q$-module
with trivial $\E_2$-action.
The {\bf Kac module} $K(\mu):=\Uq\otimes_{P_q} L^0(\mu)$ over $\Uq$
is finite dimensional and has a simple quotient $L(\mu)$.
Moreover, $\dim K(\mu) =4 \dim L^0(\mu)$.
Denote by $v_\mu$ the highest weight vector of $K(\mu)$ and by  $v_\mu^+$ the image of $v_\mu$ in $L(\mu)$.
Note that
\begin{equation}  \label{eq:Kbasis}
K(\mu) \cong L^0(\mu) \oplus \rF_2 L^0(\mu) \oplus \rF_{(12)} L^0(\mu) \oplus \rF_2\rF_{(12)}L^0(\mu).
%\Q(q)\langle 1, \rF_2, \rF_{(12)}, \rF_2\rF_{(12)}\rangle  \otimes L^0(\mu).
\end{equation}
Hence, when applying elements in $\bf B$  to $v_\mu$, the resulting elements are
nonzero exactly when $0\le r \le a-b$, thanks to
$\rF_1^{(a-b+1)} v_\mu=0$.

\begin{prp}  \label{prp:CBKac}
Let $\mu =a \delta_1+b \delta_2 + c\varep_1$, with  $a-b \in \Z_{\ge
0}$. Then we have
 $$
 \{u v_\mu  \mid u v_\mu
\neq 0, u \in {\bf B}\} = \big \{\rF_1^{(r)}v_\mu, \rF_2\rF_1^{(r)}v_\mu, \rF_1^{(r+1)}\rF_2v_\mu, \rF_2\rF_{(12)}
\rF_1^{(r)}v_\mu \mid 0\le r \le a-b \big \},
$$
and this set forms a basis of the
Kac module $K(\mu)$. It is canonical in the sense that it descends from the canonical basis.
\end{prp}

\begin{proof}
The equality of the two sets in the proposition follows by the two identities that
$\rF_1^{(a-b+1)} v_\mu=0$ and $\rF_1^{(a-b+2)} \rF_2 v_\mu=0$.

Note that $F_1^{(r)}v^0_\mu$ with $0\leq r\leq a-b$ forms a basis of $L^0(\mu)$.
Then by \eqref{eq:Kbasis}, the elements
\[\set{\rF_1^{(r)}v_\mu,\ \rF_2\rF_1^{(r)}v_\mu,\ \rF_{(12)}\rF_1^{(r)}v_\mu,\ \rF_2\rF_{(12)}\rF_1^{(r)}v_\mu\mid 0\leq r\leq a-b }\]
form a basis of $K(\mu)$. Since the transition matrix from this basis to the set given in the proposition
is upper-unitriangular, this set must form a basis of $K(\mu)$.
\end{proof}

%%%%%%%%%%%%%%%%
\subsection{Canonical basis for simple modules}

Recall that the Weyl vector for $\gl(2|1)$ is
$\rho%=\frac{1}{2}\parens{\sum_{\alpha\in \Phi_\zero^+}\alpha-\sum_{\beta\in \Phi_\one^+}\beta}
=-\delta_2+\varep_1$.
A weight $\lambda$ is called typical if $\ang{\alpha,\lambda+\rho}\neq 0$
for all $\alpha\in \Phi_{\one}^+$; otherwise, we say the weight is atypical.

Let $\mu = a \delta_1+b \delta_2 + c \varep_1$, with  $a-b \in \Z_{\ge
0}$. Then $\mu$ is typical only if $a\neq -c-1$ and $b\neq-c$.
If $\mu$ is typical, then $K(\mu)$ is irreducible.

\begin{cor}  \label{cor:typical}
If $\mu$ is typical, then $L(\mu)$ has a canonical basis given by Proposition~\ref{prp:CBKac}.
\end{cor}

Therefore, it remains
to consider $L(\mu)$ when $\mu$ is atypical. The first step is to determine
when canonical basis vectors are zero in $L(\mu)$.

\begin{lem}  \label{lem:21}
Assume that  $\mu =a \delta_1+b \delta_2 + c\varep_1$, where  $a-b \in \Z_{\ge
0}$, is atypical; that is, $a=-c-1$ or $b=-c$.  Then the following statements hold in $L(\mu)$:
\begin{enumerate}
\item
 $\rF_1^{(r)}v_\mu^+ \neq 0$ if and only if $\ 0\le r \le a-b$.

\item If $a=-1-c$, then
 $\rF_2 \rF_1^{(r)}v_\mu^+ \neq 0$ if and only if $\ 0\le r \le a-b$.

\noindent
If $b=-c$, then $\rF_2 \rF_1^{(r)}v_\mu^+ \neq 0$ if and only if $1\le r \le a-b$.

 \item $([r+b+c] \rF_1^{(r)} \rF_2 - [b+c] \rF_2 \rF_1^{(r)})v_\mu^+=0$ for all $r\ge 0$.

 \item $\rF_2\rF_1^{(r+1)} \rF_2 v_\mu^+= \rF_2\rF_{(12)}\rF_1^{(r)} v_\mu^+= \rF_1^{(r)}\rF_2\rF_{1} \rF_2   v_\mu^+=0$ for all $r\ge 0$.

 \item
 $\rF_1^{(r+1)} \rF_2 v_\mu^+ \neq 0$ if and only if $b\neq -c$ and $0\le r \le a-b$.
\end{enumerate}
\end{lem}

\begin{proof}
We will use repeatedly the fact that a  $\nu$-weight vector in $L(\mu)$ with $\nu \neq \mu$
which is annihilated by $\E_1$ and $\E_2$ must be zero.

(1) It follows from the representation theory of $\Uq(\mathfrak{sl}_2)$
generated by $\E_1$ and $\rF_1$.

(2) By a direct computation we have that % in $K(\mu)$,
$\E_2\rF_2\rF_1^{(r)} v_\mu^+ =[r+b+c]\rF_1^{(r)}v_\mu^+$.

If $a=-1-c$, then $r+b+c=0$ implies that $r=a-b+1$, and so $\rF_2 \rF_1^{(r)}v_\mu^+ \neq 0$ if $0\le r \le a-b$.
Note that $\rF_2\rF_1^{(a-b+1)} v_\mu^+=0$ since this vector is annihilated by $\E_1$ and $\E_2$ simultaneously.

If $b=-c$, then $r+b+c=0$ implies that $r=0$, and so $\rF_2 \rF_1^{(r)}v_\mu^+ \neq 0$ if $1\le r \le a-b$.
Note that $\rF_2 v_\mu^+=0$ since $\rF_2 v_\mu^+$ is annihilated by $\E_1$ and $\E_2$ simultaneously.

Hence (2) is proved when we take (1) into account.

% First we compute that
%\begin{align*}
%\E_1([b+1+c]\rF_1\rF_2-[b+c]\rF_2 \rF_1)v_\mu^+ &=([b+c+1][a+1-b] - [b+c][a-b])F_2v_\mu^+,
%  \\
%\E_2( [b+1+c] \rF_1 \rF_2 - [b+c] \rF_2 \rF_1) v_\mu^+ &=[b+c]([b+1+c]-[b+1+c])F_1v_\mu^+=0.
%\end{align*}
%If $b=-c$, then $F_2v_\mu^+=0$ by (2), and we are done.
%If $a=-1-c$, then $[b+c+1][a+1-b]-[b+c][a-b] %=[b-a][a+1-b]-[b-a-1][a-b]
%=0$.
%In either case, we see that
%$([b+1+c]\rF_1\rF_2-[b+c]\rF_2 \rF_1)v_\mu^+=0$.

(3) This is trivial for $b=-c$, since $F_2v_\mu^+=0$. So, we may assume $a=-1-c$.
We shall proceed by induction, with the case $r=0$ being trivial.
Set $d=b+c$.
Then (3) follows by the following computations (and by inductive assumption):
\begin{align*}
\E_1( [d+r] \rF_1^{(r)} \rF_2 - [d] \rF_2 \rF_1^{(r)}) v_\mu^+
%&=([d+r] [-d-(r-1)] \rF_1^{(r-1)} \rF_2-[d][-d-r]  \rF_2 \rF_1^{(r-1)})v_\mu^+\\
&=- [d+r]([d+(r-1)] \rF_1^{(r-1)} \rF_2 - [d] \rF_2 \rF_1^{(r-1)})F_2v_\mu^+=0,
 \\
\E_2( [d+r] \rF_1^{(r)} \rF_2 - [d] \rF_2 \rF_1^{(r)}) v_\mu^+ &=[d]([d+r]-[d+r])F_1v_\mu^+=0.
\end{align*}

(4)
By an $\rF$-version of \eqref{eq:EEEE}, we have
$\rF_2\rF_1^{(r+1)} \rF_2 v_\mu^+=\rF_2\rF_{(12)}\rF_1^{(r)} v_\mu^+= \rF_1^{(r)}\rF_2\rF_{1} \rF_2   v_\mu^+$.
It remains to show that $\rF_2\rF_{1} \rF_2   v_\mu^+=0$.
This follows from the computations below which use (4) in the second line:
\begin{align*}
\E_1 \rF_2 \rF_1 \rF_2 v_\mu^+ &=\rF_2 \E_1 \rF_1 \rF_2v_\mu^+=0,
  \\
\E_2 \rF_2 \rF_1 \rF_2 v_\mu^+ &=([b+1+c]\rF_1\rF_2-[b+c]\rF_2 \rF_1)v_\mu^+=0.
\end{align*}

(5)
Note that $b\neq -c$ if and only if $F_2v_{\mu}\neq 0$.
As in (1) the claim follows from
the representation theory of $\Uq(\mathfrak{sl}_2)$ generated by $\E_1$ and $\rF_1$
(when applied to the highest weight vector $F_2 v_\mu^+$).
 \end{proof}

\begin{thm}  \label{CBmod}
Assume that $\mu = a \delta_1+b \delta_2 + c \varep_1$, where  $a-b \in \Z_{\ge
0}$, is atypical; that is, $a=-c-1$ or $b=-c$.
\begin{enumerate}
\item
If $b=-c$ or $b=a=-c-1$, then  $\{u v_\mu^+  \mid u v_\mu^+
\neq 0, u \in {\bf B} \}$ forms a (canonical) basis of
$L(\mu)$.  In particular, $\dim L(\mu) =2(a-b)+1$.

\item
It $b \neq a=-c-1$, then $\{u v_\mu^+  \mid u v_\mu^+
\neq 0, u \in {\bf B} \}$ is linearly dependent in
$L(\mu)$,
but the subset
$\{\rF_1^{(r)}v_\mu^+ \, (0\leq r\leq a-b),\;  \rF_1^{(r)}\rF_2v_\mu^+ \, (0\le r \le a-b+1)\}$
is a basis for $L(\mu)$. In particular, $\dim L(\mu) =2(a-b)+3$.
\end{enumerate}
\end{thm}

\begin{proof}
For (1), there are two cases.
If $b=-c$, then Lemma \ref{lem:21} shows that
\[\set{u v_\mu^+ \mid u\in\mathbf{B},  uv_\mu^+\neq 0}=\set{\rF_1^{(r)}v_\mu^+\, (0\leq r\leq a-b),  \rF_2\rF_1^{(r)}v_\mu^+ \, (1\leq r\leq a-b)\; }.
\]
If $b=a=-1-c$, then Lemma \ref{lem:21} implies
%\[
$\set{u v_\mu^+ \mid u\in\mathbf{B},  uv_\mu^+\neq 0}=\set{v_\mu^+,\, \rF_2v_\mu^+,\, \rF_1 \rF_2v_\mu^+}.
$
%\]
In either case, the set $\set{u\in\mathbf{B}\mid uv_\mu^+\neq 0}$ spans $L(\mu)$; it is indeed a basis since
each vector lies in a different weight space.

For (2), Lemma \ref{lem:21} implies that
\[\set{u v_\mu^+ \mid u\in\mathbf{B},  uv_\mu^+\neq 0}=\set{\rF_1^{(r)}v_\mu^+, \rF_1^{(r+1)}\rF_2v_\mu^+, \rF_2\rF_1^{(r)} v_\mu^+\mid 0\leq r\leq a-b}.\]
All of these elements lie in different weight spaces except for $\rF_1^{(r)}\rF_2$ and $\rF_2\rF_1^{(r)}$ for $0\leq r\leq a-b$.
Now $(\mu-r\alpha_1-\alpha_2)$-weight space is spanned by $\rF_1^{(r)}\rF_2v_\mu^+$ and $\rF_2\rF_1^{(r)} v_\mu^+$.
However, Lemma \ref{lem:21}(4) shows that these vectors are linearly dependent.
Then we may choose one of the vectors as a basis element, and  (2) follows.
\end{proof}

 We call $L(\mu)$  a {\bf polynomial representation} of $\Uq$ if $\mu =a\dt_1 +b\dt_2 +c \ep_1$ with
$ (a, b, \underbrace{1, \ldots, 1}_c)$ being a partition (This is analogous to the polynomial
 representations of the Lie superalgebra $\gl(m|n)$; see \cite{CW12}).
% with all $a,b,c \in \Z_{\ge 0}$, $a\ge b$, and $b>0$ whenever $c>0$.
 Note that a polynomial representation $L(\mu)$ is atypical if and only if $b=c=0$.
 We have the following corollary from  Theorem~\ref{CBmod}(1) and Corollary~\ref{cor:typical}.

\begin{cor}
 \label{cor:pol}
The set $\{u v_\mu^+  \mid u v_\mu^+
\neq 0, u \in {\bf B} \}$ forms a canonical basis for every polynomial representation $L(\mu)$.
\end{cor}

In a setting similar to Proposition~\ref{prp:CBKac}, Theorem~\ref{CBmod}(1), Corollarys~\ref{cor:typical} and \ref{cor:pol},
we will simply say that the canonical basis of $U_q^-$ descends to the canonical bases of the corresponding $U_q$-modules.

%%%%%%%%%%%%%%%% \subsection{Some conjectures}

We end with formulating some general conjectures regarding canonical basis for representations of
quantum supergroup of $\gl(m+1|1)$.
let $U_q^-$ be the negative half of quantum $\gl(m+1|1)$ of type $A(m,0)$, for $m\ge 1$.
We transport the canonical basis of the positive half quantum supergroup $U_q$ (see Theorem~\ref{T:canonical basis})
to that for $U_q^-$ via an  (anti-)isomorphism sending $\E_i$ to $\rF_i$
for all $i$.

\begin{cnj}
For type $A(m,0)$,
the canonical basis of $U_q^-$ descends to the canonical bases of the Kac modules as well as those of polynomial representations of $U_q$.
\end{cnj}

For type $C(n)$, we also conjecture that the canonical basis of the negative half quantum supergroup
descends to the canonical bases of the Kac modules.


\begin{thebibliography}{CHW2}

\bibitem[A]{a} I.~Angiono, {\em A presentation by generators and relations of Nichols algebras of diagonal type}, arXiv:1008.4144.

\bibitem[BKK]{bkk} G. Benkart, S.-J. Kang and M. Kashiwara,
{\em Crystal bases for the quantum superalgebra }$U_q(\mathfrak{gl}(m,n))$,
Journal of Amer. Math. Soc.  {\bf 13} (2000), 295--331.

\bibitem[BKMc]{bkmc} J. Brundan, A. Kleshchev and P. J. McNamara, {\em Homological properties of finite type Khovanov-
Lauda-Rouquier algebras}, Duke Math. J. (to appear), arXiv:1210.6900.

\bibitem[CFLW]{cflw} S.~Clark, Z.~Fan, Y.~Li and W.~ Wang, {\em Quantum Supergroups III. Twistors}, Comm. Math. Phys. (to appear), arXiv:1307.7056.

\bibitem[CHW1]{CHW1} S.~Clark, D.~Hill and W.~Wang, {\em Quantum supergroups I. Foundations},  Transformation Groups (2013),
no. 4, 1019--1053.

\bibitem[CHW2]{CHW} S.~Clark, D.~Hill and W.~Wang,
{\em Quantum supergroups II. Canonical basis}, submitted to Rep. Theory, arXiv:1304.7837.

\bibitem[CW]{CW12} S.-J.~Cheng, W.~Wang,
{\em Dualities and Representations of Lie Superalgebras}.
Graduate Studies in Mathematics {\bf 144}, Amer. Math.  Soc., Providence, RI, 2012.

\bibitem[EL]{EL} A. Ellis and A. Lauda,
{\em An odd categorification of $U_q(\mathfrak{sl}_2)$},
arXiv:1307.7816.

\bibitem[G]{grn} J.A. Green,
{\em Quantum groups, Hall algebras and quantum shuffles}. In: Finite reductive groups (Luminy 1994), 273--290, Birkh\"{a}user Prog. Math. {\bf 141}, 1997.

\bibitem[HKS]{hks} D. Hill, J. Kujawa, and J. Sussan,
{\em Affine Hecke-Clifford algebras and type $Q$ Lie superalgebras},
Math. Z. {\bf 268} (2011), 1091--1158.

\bibitem[HMM]{hmm} D. Hill, G. Melvin and D. Mondragon,
{\em Representations of quiver Hecke algebras via Lyndon bases}, Journal of Pure and Applied Algebra,  {\bf 216} (2012), 1052--1079.

\bibitem[HW]{hw} D. Hill and W. Wang,
{\em Categorification of quantum Kac-Moody superalgebras}, Trans. AMS (to appear), arXiv:1202.2769v2.

\bibitem[Kac]{kac} V. Kac, \emph{Lie superalgebras}, Adv. Math. \textbf{26} (1977), 8-96.

\bibitem[Kas]{ka} M. Kashiwara,
{\em On crystal bases of the $q$-analogue of universal enveloping algebras}, Duke Math. J. \textbf{63} (1991), 465--516.

\bibitem[KKO]{KKO} S.-J. Kang, M. Kashiwara and S.-J. Oh,
              {\em Supercategorification of quantum Kac-Moody algebras~ II},   arXiv:1303.1916.

%\bibitem[Kato]{Kato} Syu Kato, {\em PBW bases and KLR algebras}, arXiv:1203.5245.

\bibitem[Kh]{kh} M. Khovanov,
{\em  How to categorify one-half of quantum $\mathfrak{gl}(1|2)$}, arXiv:1007.3517.

\bibitem[KS]{ks} M. Khovanov and J. Sussan, \emph{A categorification of quantum $\gl(1,n)$}, in preparation.

\bibitem[KL]{khl} M. Khovanov and A. Lauda,
{\em A diagrammatic approach to categorification of quantum groups I,} Represent. Theory \textbf{13} (2009), 309--347.
%
%\bibitem[KL2]{khl2} \bysame, A diagrammatic approach to categorification of quantum groups
II. Trans. Amer. Math. Soc. 363 (2011), 2685--2700.

\bibitem[KR]{klram2} A. Kleshchev and A. Ram,
{\em Representations of Khovanov-Lauda-Rouquier algebras and combinatorics of Lyndon words},  Math. Ann. {\bf 349} (2011),  943--975.

\bibitem[Kw1]{Kw1} J.-H. Kwon,
{\em Crystal bases of $q$-deformed Kac modules over the quantum superalgebra $U_q(\gl(m|n))$},
    arXiv:1203.5590.

\bibitem[Kw2]{Kw2} J.-H. Kwon,
{\em Super duality and crystal bases for quantum orthosymplectic superalgebras}, arXiv:1301.1756.

\bibitem[LR]{lr} P. Lalonde and A. Ram,
{\em Standard Lyndon bases of Lie algebras and enveloping algebras}, Trans. Am. Math. Soc. \textbf{347} (1995), 1821--1830.

\bibitem[Lec]{lec} B. Leclerc,
{\em Dual canonical bases, quantum shuffles and $q$-characters}, Math. Z. \textbf{246} (2004),  691--732.

\bibitem[LS]{ls} S.Z. Levendorskii, S. Soibelman, {\em Algebras of functions on compact quantum groups, Schubert cells and quantum
tori}, Comm. Math. Physics \textbf{139} (1991), 141–-170.


\bibitem[Lo]{lo} M. Lothaire,
{\em Combinatorics on Words},  Cambridge University Press, Cambridge, 1997.

\bibitem[Lu1]{lu1} G. Lusztig,
{\em Canonical bases arising from quantized enveloping algebras},  Journal of Amer. Math. Soc. \textbf{3} (1990), 447--498.

\bibitem[Lu2]{lu2} G. Lusztig, {\em Quantum groups at roots of 1}, Geom. Dedicata \textbf{35} (1990), 89--114.

\bibitem[Lu3]{lu} G. Lusztig,
{\em Introduction to quantum groups}, Progress in Mathematics {\bf 110}, Birkh\"{a}user, Boston, MA, 1993.

\bibitem[Mc]{mc} P. McNamara,  {\em Finite dimensional representations of Khovanov-Lauda-Rouquier algebras I: Finite type}, J. reine angew. Math. (to appear), arXiv:1207.5860.

\bibitem[MZ]
{MZ} I.M. Musson and Y.-M.~ Zou,
{\em  Crystal basis for $U_q(\text{osp}(1,2r))$}, J. of Alg. \textbf{210} (1998), 514--534.

\bibitem[Ro]{ro2} M. Rosso,
{\em Quantum groups and quantum shuffles}, Invent. Math. \textbf{133}  (1998), 399--416.

\bibitem[Rou]{rq} R. Rouquier, {\em 2-Kac-Moody Algebras},  arXiv:0812.5023.

\bibitem[Ya1]{ya} H. Yamane,
{\em Quantized enveloping algebras associated with simple Lie superalgebras and their universal $R$-matrices},
Publ. Res. Inst. Math. Sci. \textbf{30} (1994), 15--87.

\bibitem[Ya2]{ya2} H.~Yamane,
{\em On defining relations of affine Lie superalgebras and affine quantized universal enveloping superalgebras},
Pub. Res. Inst. Math. Sci. {\bf 35} (1999), 321--390.

\end{thebibliography}
\end{document}